\documentclass[12pt,a4paper]{article}

\usepackage[german]{babel}
\usepackage[utf8]{inputenc}
\usepackage[T1]{fontenc}
\usepackage{amsmath}
\usepackage{amsfonts}
\usepackage{amssymb}
\usepackage{amsthm}
\usepackage{mathtools}
\usepackage{parskip} 
\usepackage[left=3cm,right=3cm,top=3.4cm,bottom=4.5cm]{geometry}
\usepackage{color}
\usepackage{bibgerm,cite}
\usepackage[bottom,hang]{footmisc}
\usepackage[arrow, matrix, curve]{xy}

\swapnumbers
\numberwithin{equation}{section}

\theoremstyle{definition}
\newtheorem{definition}{Definition}[section]
\newtheorem{beispiel}[definition]{Beispiel}
\newtheorem{beispiele}[definition]{Beispiele}
\newtheorem{bemerkung}[definition]{Bemerkung}
\theoremstyle{plain}
\newtheorem{satz}[definition]{Satz}
\newtheorem{lemma}[definition]{Lemma}
\newtheorem{korollar}[definition]{Korollar}

\newcommand{\ftnote}[1]{\setlength{\footnotemargin}{6pt}\footnote{#1 \vspace{6pt}}}
\newcommand{\N}{\mathbb{N}}
\newcommand{\R}{\mathbb{R}}
\newcommand{\eps}{\varepsilon}
\newcommand{\ee}{\emph}
\newcommand{\E}{\mathcal{E}^k(A,E)}
\newcommand{\EE}{\mathcal{E}}
\newcommand{\J}{\mathcal{J}}
\newcommand{\al}{\vert\alpha\vert}
\newcommand{\bet}{\vert\beta\vert}
\newcommand{\p}{\partial^{\alpha}}
\newcommand{\A}{\mathcal{A}}
\newcommand{\B}{\mathcal{B}}

\begin{document}

\begin{center}
\section*{Der Whitneysche Fortsetzungssatz für vektorwertige Funktionen}
\large{Johanna Jakob}
\end{center}
\vspace{0.5cm}
\begin{abstract}
\noindent Sei $k\in\N_0\cup\{\infty\}$. Der Whitneysche Fortsetzungssatz besagt, dass sich jeder reellwertige Whitneysche $k$-Jet auf einer abgeschlossenen Teilmenge $A\subseteq\R^n$ zu einer $C^k$-Funktion auf $\R^n$ fortsetzen lässt. Ausgehend von Whitneys Originalarbeit \cite{Wh}  beweisen wir analoge Resultate für Jets und Funktionen mit Werten in einem reellen Hausdorffschen lokalkonvexen topologischen Vektorraum $E$.
Im Falle $k<\infty$ erhalten wir einen stetigen linearen Fortsetzungsoperator,
also eine stetige lineare Rechtsinverse der Abbildung $C^k(\R^n,E)\to \EE^k(A,E),\ f\mapsto ((\p f)\vert_A)_{\al\leq k}.$
Die Fortsetzung von $E$-wertigen Whitneyschen $\infty$-Jets gelingt uns unter der Annahme, dass $E$ metrisierbar ist.
Sei $M$ eine endlichdimensionale $C^k$-Mannigfaltigkeit (welche einen \glqq rauen Rand\grqq\text{ }haben mag). Wir befassen uns mit der Frage, wie man Whitneysche $k$-Jets, die auf einer abgeschlossenen Teilmenge $A\subseteq M$ definiert sind, zu $C^k$-Funktionen auf $M$ fortsetzen kann. Insbesondere beweisen wir für $k<\infty$, dass die Einschränkung $C^k(M,E)\to C^k(A,E),\ f\mapsto f\vert_A$ für jede abgeschlossene Untermannigfaltigkeit $A\subseteq M$ eine stetige lineare Rechtsinverse besitzt, wenn $M$ $C^k$-parakompakt und $A$ lokalkompakt (oder $M$ regulär und $A$ kompakt) ist. 
\ee{MSC2020:} 46E10, 54C20, 58A20 (primär); 46E40, 46G05, 58C25.
\end{abstract}

\section*{Einleitung}
\addcontentsline{toc}{section}{\protect\numberline{}Einleitung}

Wir beweisen Verallgemeinerungen des Whitneyschen Fortsetzungssatzes für Jets mit Werten in lokalkonvexen topologischen Vektorräumen.

Es seien $k\in\N_0\cup\{\infty\}$ und $E$ ein reeller Hausdorffscher lokalkonvexer topologischer Vektorraum. Ein \ee{Whitneyscher $k$-Jet} auf einer Teilmenge $U\subseteq\R^n$ ist eine Familie $(f_{\alpha})_{\al\leq k}$ von stetigen Funktionen $f_{\alpha}\colon U\to E$ (indiziert durch $\alpha=(\alpha_1,\ldots,\alpha_n)\in(\N_0)^n$ mit $\vert\alpha\vert=\alpha_1+\ldots+\alpha_n\leq k$), welche 
in gleicher Weise wie die partiellen Ableitungen $\p f$ einer $C^k$-Funktion $f\colon\R^n\to E$ Taylorentwicklungen besitzen (s. Definition \ref{p}).
Wir schreiben $\EE^k(U,E)$ für den Raum solcher Whitneyschen $k$-Jets.

Whitney \cite{Wh} hat in den 1930er Jahren bewiesen, dass man jeden reellwertigen Whitneyschen $k$-Jet auf einer abgeschlossenen Teilmenge $A\subseteq\R^n$ zu einer $C^k$-Funktion auf $\R^n$ fortsetzen kann. In der vorliegenden Arbeit zeigen wir, dass solches auch für vektorwertige Whitneysche $k$-Jets möglich ist:

\ee{Es seien $U\subseteq\R^n$ eine lokalkonvexe Teilmenge mit dichtem Inneren und $A\subseteq U$ eine relativ abgeschlossene, lokalkompakte Teilmenge.
Ist $k<\infty$ oder $E$ metrisierbar, so ist die Abbildung
\[C^k(U,E)\to\EE^k(A,E),\quad f\mapsto ((\p f)\vert_A)_{\al\leq k}\]
surjektiv. Im Falle $k<\infty$ besitzt sie eine stetige lineare Rechtsinverse.}

Wenn $A$ zusätzlich lokalkonvex ist und ein dichtes Inneres hat, dann sind $C^k(A,E)$ und $\EE^k(A,E)$ als topologische Vektorräume isomorph, woraus wir folgern, dass für $k<\infty$ eine stetige lineare Rechtsinverse der Einschränkung
\[C^k(U,E)\to C^k(A,E),\quad f\mapsto f\vert_A\]
existiert.

Weiterhin konstruieren wir Fortsetzungen von Whitneyschen $k$-Jets, welche auf Teilmengen einer endlichdimensionalen $C^k$"--Man"-nig"-fal"-tig"-keit $M$ mit rauem Rand definiert sind. Der letztere Begriff umfasst gewöhnliche Mannigfaltigkeiten sowie Mannigfaltigkeiten mit Rand. Sei $\A$ der zu $M$ gehörige Atlas mit Karten $\varphi\colon U_{\varphi}\to V_{\varphi}\subseteq\R^n$.
Ein \ee{Whitneyscher $k$-Jet} auf einer Teilmenge $A\subseteq M$ ist eine Familie $(f_{\varphi})_{\varphi\in\A}$ von Whitneyschen $k$-Jets \[f_{\varphi}=(f_{\varphi,\alpha})_{\al\leq k}\in\EE^k(\varphi(A\cap U_{\varphi}),E),\] welche in dem Sinne miteinander \ee{korrespondieren}, dass man für alle $\al\leq k$, $x\in A$ und Karten $\varphi,\psi\in\A$ um $x$ den Wert $f_{\varphi,\alpha}(\varphi(x))$ in gleicher Weise aus $(f_{\psi,\beta}(\psi(x)))_{\bet\leq\al}$ berechnen kann, wie man für eine $C^k$-Funktion $f\colon M\to E$ die Ableitung \[\p(f\circ\varphi^{-1})(\varphi(x))=\p\big((f\circ\psi^{-1})\circ(\psi\circ\varphi^{-1})\vert_{\varphi(U_{\varphi}\cap U_{\psi})}\big)(\varphi(x))\] mittels der Fa\`a-di-Bruno-Formel aus den Ableitungen $(\partial^{\beta}(f\circ\psi^{-1})(\psi(x)))_{\bet\leq\al}$ berechnet (s. Definition \ref{eee}). 
Sei $\EE^k_M(A,E)$ der Raum jener Whitneyschen $k$-Jets. Wir beweisen den folgenden Fortsetzungssatz:

\ee{Sei $A\subseteq M$ eine abgeschlossene Teilmenge. Angenommen, eine der beiden folgenden Bedingungen ist erfüllt:
\begin{itemize}
\item[(a)] $M$ ist $C^k$-parakompakt und $A$ lokalkompakt.
\item[(b)] $M$ ist regulär als topologischer Raum und $A$ kompakt.
\end{itemize}
Wenn $k<\infty$ oder $E$ metrisierbar ist, so ist die Abbildung
\[C^k(M,E)\to\EE^k_M(A,E),\quad f\mapsto \big((\p(f\circ\varphi^{-1}))\vert_{\varphi(A\cap U_{\varphi})}\big)_{\substack{\varphi\in\A\\ \al\leq k}}\]
surjektiv. Im Falle $k<\infty$ besitzt sie eine stetige lineare Rechtsinverse.}

Für jede abgeschlossene Untermannigfaltigkeit $A\subseteq M$, welche auch \glqq voll"-di"-men"-sio"-nal\grqq\text{ }(im Sinne von Definition \ref{ccc}) sein mag, existiert eine Einbettung von topologischen Vektorräumen $C^k(A,E)\hookrightarrow\EE^k_M(A,E)$. Daraus schließen wir, dass für $k<\infty$ eine stetige lineare Rechtsinverse der Einschränkung
\[C^k(M,E)\to C^k(A,E),\quad f\mapsto f\vert_A\]
existiert, wenn eine der Bedingungen (a) oder (b) erfüllt ist.

Eine Einordnung in die Literatur wird im Schlusskapitel gegeben. Der vorliegenden Abhandlung liegt die Masterarbeit der Autorin \cite{J} zugrunde, welche von Professor Helge Glöckner an der Universität Paderborn betreut wurde.

\section{Vorbereitung – Differenzierbarkeit vektorwertiger Funktionen}

In diesem Kapitel sammeln wir Grundlagen über den Dif"-fe"-ren"-tial"-kal"-kül in lokalkonvexen Räumen, den Ausführungen von Glöckner und Neeb \cite{Gl, Gl2} folgend.

Wir schreiben $\N\coloneqq\{1,2,\ldots\}$ und $\N_0\coloneqq\N\cup\{0\}$. Reelle Hausdorffsche lokalkonvexe topologische Vektorräume nennen wir verkürzend \ee{lokalkonvexe Räume.}
Es seien stets $E$ ein lokalkonvexer Raum und $n\in\N$. Wir bezeichnen mit $\lVert\cdot\rVert$ die euklidische Norm und mit $\lVert\cdot\rVert_{\infty}$ die Maximumnorm auf $\R^n$.
Für alle nichtleeren Teilmengen $U,V\subseteq\R^n$ und $x\in\R^n$ schreiben wir $d(U,V)\coloneqq\inf\{\lVert u-v\rVert\colon u\in U, v\in V\}$ und $d(x,V)\coloneqq d(\{x\},V)$ sowie $d_{\infty}(U,V)\coloneqq\inf\{\lVert u-v\rVert_{\infty}\colon u\in U,v\in V\}$ und $d_{\infty}(x,V)\coloneqq d_{\infty}(\{x\},V)$. 
Für alle kompakten Teilmengen $K\subseteq\R^n$ sei $\text{diam}(K)\coloneqq \sup\{\lVert x-y\rVert\colon x,y\in K\}$ der Durchmesser von $K$. Wir bezeichnen mit $e_1\coloneqq(1,0,\ldots,0),\ldots,e_n\coloneqq (0,\ldots,0,1)$ die Einheitsvektoren in $\R^n$ (bzw. die entsprechenden Multiindizes in $(\N_0)^n$, je nach Situation).
Ist $U\subseteq X$ eine Teilmenge eines topologischen Raumes $X$, so schreiben wir $U^{\circ}$ für das Innere, $\overline{U}$ für den Abschluss und $\partial U=\overline{U}\setminus U^{\circ}$ für den Rand von $U$ bezüglich $X$. Man sagt, dass $U$ ein \ee{dichtes Inneres} hat, wenn $U^{\circ}$ dicht in $U$ liegt.
Eine Teilmenge $U\subseteq E$ heißt \ee{lokalkonvex}, wenn jedes $x\in U$ eine konvexe Umgebung in $U$ besitzt. Wir notieren $\frac{x}{t}\coloneqq\frac{1}{t}x$ für alle $x\in E$ und $t\in\R\setminus\{0\}$. Es seien $B^q_r(x)\coloneqq\{y\in E\colon q(y-x)< r\}$ die offene Kugel und $\overline{B}^q_r(x)\coloneqq\{y\in E\colon q(y-x)\leq r\}$ die abgeschlossene Kugel um $x\in E$ vom Radius $r>0$ bezüglich einer stetigen Halbnorm $q$ auf $E$. Wenn $(E,\lVert\cdot\rVert_E)$ ein normierter Raum ist, setzen wir $B_r(x)\coloneqq B_r^E(x)\coloneqq B_r^{\lVert\cdot\rVert_E}(x)$ und $\overline{B}_r(x)\coloneqq \overline{B}^E_r(x)\coloneqq \overline{B}_r^{\lVert\cdot\rVert_E}(x)$.
Ist eine stetige Kurve $\gamma\colon [a,b]\to E$ mit $a<b$ in $\R$ gegeben, so bezeichnen wir mit $\int_a^b\gamma(t)dt$ das \ee{schwache Integral} von $\gamma$, sofern es existiert; dies ist der (notwendig eindeutig bestimmte) Vektor in $E$, für welchen $\lambda(\int_a^b\gamma(t)dt)=\int_a^b\lambda(\gamma(t))dt$ für alle stetigen linearen Funktionale $\lambda\colon E\to \R$ gilt.
Produkte von topologischen Räumen seien mit der Produkttopologie versehen und Teilmengen von topologischen Räumen mit der induzierten Topologie. Partitionen der Eins auf einem topologischen Raum seien wie üblich definiert und immer als lokalendlich angenommen.
Ist $f\colon X\to F$ eine Funktion auf einem topologischen Raum $X$ in einen Vektorraum $F$, so schreiben wir $\text{supp}(f)\coloneqq\overline{\{x\in X\colon f(x)\neq 0\}}$ für den Träger von $f$.
Für Multiindizes $\alpha=(\alpha_1,\ldots,\alpha_n),\ \beta=(\beta_1,\ldots,\beta_n)\in(\N_0)^n$ notieren wir 
$\lvert\alpha\rvert\coloneqq\alpha_1+\ldots+\alpha_n $, $\alpha!\coloneqq\alpha_1!\cdots\alpha_n!$, $x^{\alpha}\coloneqq x_1^{\alpha_1}\cdots x_n^{\alpha_n}$ für $x=(x_1,\ldots,x_n)\in\R^n$, $\alpha\leq\beta$, wenn $\alpha_i\leq\beta_i$ für alle $i\in\{1,\ldots,n\}$, $\alpha+\beta\coloneqq(\alpha_1+\beta_1,\ldots,\alpha_n+\beta_n)$ sowie $\binom{\alpha}{\beta}\coloneqq\binom{\alpha_1}{\beta_1}\cdots\binom{\alpha_n}{\beta_n}$.

\subsection*{Grenzwerte von Funktionen}
\addcontentsline{toc}{subsection}{\protect\numberline{}Grenzwerte von Funktionen}

\begin{definition}
Es seien $f\colon I\to E$ eine Funktion auf einer Teilmenge $I\subseteq\R$ und $t_0\in \R$ ein Häufungspunkt\ftnote{D.h. es existiert eine Folge $(t_n)_{n\in\N}$ in $I\setminus\{t_0\}$ mit $\lim\limits_{n\to\infty}t_n=t_0.$} von $I$. Man nennt einen (notwendig eindeutig bestimmten) Vektor $v\in E$ den \ee{Grenzwert} von $f$ in $t_0$ und schreibt
\[\lim_{t\to t_0}f(t)=v\] oder \[f(t)\to v \text{ für }t\to t_0,\]
wenn die beiden folgenden äquivalenten Bedingungen erfüllt sind:
\begin{itemize}
\item[(a)] Für jede Folge $(t_n)_{n\in\N}$ in $I\setminus\{t_0\}$ mit $\lim\limits_{n\to\infty}t_n=t_0$ gilt $\lim\limits_{n\to\infty} f(t_n)=v$.
\item[(b)] Für alle stetigen Halbnormen $q$ auf $E$ und $\eps>0$ gibt es ein $\delta>0$ derart, dass \[q(f(t)-v)<\eps\] für alle $t\in I$ mit $0<\lvert t-t_0\rvert<\delta$.
\end{itemize}
\end{definition}

\begin{proof}[Beweis der Äquivalenz]
Es ist nicht schwierig zu prüfen, dass sowohl (a) als auch (b) äquivalent zur Stetigkeit der Funktion 
\[g\colon I\cup\{t_0\}\to E,\quad t\mapsto\begin{cases} f(t) &\text{wenn }t\in I\setminus\{t_0\}; \\v &\text{wenn }t=t_0\end{cases}\]
in $t_0$ sind.
\end{proof}

\subsection*{$C^k$-Kurven}
\addcontentsline{toc}{subsection}{\protect\numberline{}$C^k$-Kurven}

\begin{definition}
Eine stetige Funktion $\gamma\colon I\to E$ auf einem nichtentarteten\ftnote{D.h. $I$ enthält mehr als einen Punkt.} Intervall $I\subseteq\R$ heißt \ee{$C^0$-Kurve}. Eine $C^0$-Kurve $\gamma\colon I\to E$ heißt \ee{$C^1$-Kurve}, wenn für alle $t_0\in I$ der Grenzwert
\[\gamma'(t_0)\coloneqq\lim_{t\to t_0}\frac{\gamma(t)-\gamma(t_0)}{t-t_0}\]
existiert, sprich $\gamma$ \ee{differenzierbar} in $t_0$ ist, und die Funktion $\gamma'\colon I\to E,\ t\mapsto\gamma'(t)$ stetig ist. Sei $\gamma^{(0)}\coloneqq \gamma$. Für $k\in\N$ definiert man rekursiv $\gamma$ als \ee{$C^k$-Kurve}, wenn $\gamma$ eine $C^{k-1}$-Kurve und $\gamma^{(k-1)}$ eine $C^1$-Kurve ist, und setzt $\gamma^{(k)}\coloneqq(\gamma^{(k-1)})'.$ Man nennt $\gamma$ eine \ee{$C^{\infty}$-Kurve} oder \ee{glatt}, wenn $\gamma$ eine $C^k$-Kurve für alle $k\in\N$ ist.
\end{definition}

\subsection*{$C^k$-Funktionen}
\addcontentsline{toc}{subsection}{\protect\numberline{}$C^k$-Funktionen}

Wir werden mit $C^k$-Funktionen $f\colon U\to E$ arbeiten, die auf lokalkonvexen Teilmengen $U\subseteq\R^n$ mit dichtem Inneren definiert sind.

\begin{definition}
Es seien $F$ ein lokalkonvexer Raum und $f\colon U\to E$ eine Funktion auf einer lokalkonvexen Teilmenge $U\subseteq F$ mit dichtem Inneren. Man definiert die \ee{Richtungsableitung} von $f$ an der Stelle $x\in U^{\circ}$ in Richtung $y\in F$ als den Grenzwert
\[df(x,y)\coloneqq (D_yf)(x)\coloneqq\lim_{t\to 0}\frac{f(x+ty)-f(x)}{t},\]
falls dieser existiert. Wenn $f$ stetig ist, nennt man $f$ eine \ee{$C^0$-Funktion}. Sei $k\in\N\cup\{\infty\}$. Man sagt, dass $f$ eine \ee{$C^k$-Funktion} ist, wenn $f$ stetig ist, die iterierte Richtungsableitung
\[d^{(l)}f(x,y_1,\ldots,y_l)\coloneqq(D_{y_l}\cdots D_{y_1}f)(x)\]
für alle $l\in\N$ mit $l\leq k$, $x\in U^{\circ}$ und $y_1,\ldots,y_l\in F$ existiert, die Funktion
\[d^{(l)}f\colon U^{\circ}\times F^l\to E\]
stetig ist und eine (notwendig eindeutig bestimmte) stetige Fortsetzung
\[d^{(l)}f\colon U\times F^l\to E\]
besitzt. Eine $C^{\infty}$-Funktion wird auch \ee{glatte Funktion} genannt. Wir bezeichnen mit $C^k(U,E)$ den reellen Vektorraum der $C^k$-Funktionen $f\colon U\to E$.
\end{definition}

Nach dem Satz von Schwarz \cite[Lemma 1.4.7]{Gl} ist $d^{(l)}f(x,\cdot)\colon F^l\to E$ eine stetige, symmetrische $k$-lineare Funktion für alle $l\in\N$ mit $l\leq k$ und $x\in U$.

Die Komposition zweier $C^k$-Funktionen ist wieder eine $C^k$-Funktion (s. \cite[Lemma 1.4.10]{Gl}).

Ist $\gamma\colon I\to E$ eine Funktion auf einem nichtentarteten Intervall $I\subseteq\R$, so ist $\gamma$ genau dann eine $C^k$-Kurve, wenn $\gamma$ eine $C^k$-Funktion ist (s. \cite[Lemma 1.4.13]{Gl}).

\begin{satz}\label{xx}
Es seien $k\in\N\cup\{\infty\}$ und $f\colon U\to E$ eine stetige Funktion auf einer lokalkonvexen Teilmenge $U\subseteq\R^n$ mit dichtem Inneren. Genau dann ist $f$ eine $C^k$-Funktion, wenn die partielle Ableitung
\[\frac{\partial^l f}{\partial x_{j_l}\cdots\partial x_{j_1}}(x)\coloneqq d^{(l)}f(x,e_{j_1},\ldots,e_{j_l})\]
für alle $l\in\N$ mit $l\leq k$, $x\in U^{\circ}$ und $j_1,\ldots,j_l\in\{1,\ldots,n\}$ existiert, die Funktion
\[\frac{\partial^l f}{\partial x_{j_l}\cdots\partial x_{j_1}}\colon U^{\circ}\to E\]
stetig ist und eine (notwendig eindeutig bestimmte) stetige Fortsetzung
\[\frac{\partial^l f}{\partial x_{j_l}\cdots\partial x_{j_1}}\colon U\to E\]
besitzt. Ist dies der Fall, so gilt
\begin{align}\label{gl.x}d^{(l)}f(x,y_1,\ldots,y_l)=\sum_{j_1,\ldots,j_l=1}^n y_{1,j_1}\cdots y_{l,j_l}\frac{\partial^{l} f}{\partial x_{j_l}\cdots\partial x_{j_1}}(x)\end{align}
für alle $x\in U$ und $y_i=(y_{i,1},\ldots,y_{i,n})\in\R^n$ mit $i\in\{1,\ldots,l\}$.
\end{satz}

Wir schreiben auch
\[\partial^{\alpha}f\coloneqq\frac{\partial^{\alpha}f}{\partial x^{\alpha}}\coloneqq\frac{\partial^{\al} f}{\partial x_1^{\alpha_1}\cdots\partial x_n^{\alpha_n}}\quad\text{sowie}\quad \partial_j f\coloneqq\frac{\partial f}{\partial x_j}\]
für alle $\alpha=(\alpha_1,\ldots,\alpha_n)\in (\N_0)^n$ mit $\al\leq k$ bzw. $j\in\{1,\ldots,n\}$.

Für den Beweis benötigen wir den Mittelwertsatz und die Stetigkeit parameterabhängiger Integrale.

\begin{satz}[\textbf{Mittelwertsatz}]
Es seien $F$ ein lokalkonvexer Raum und $f\colon U\to E$ eine $C^1$-Funktion auf einer lokalkonvexen Teilmenge $U\subseteq F$ mit dichtem Inneren. Dann gilt
\[f(y)-f(x)=\int_0^1 df(x+t(y-x),y-x)\, dt\]
für alle $x,y\in U$, deren Verbindungsstrecke in $U$ liegt.
\end{satz}

\begin{proof}
Siehe \cite[Proposition 1.4.8]{Gl}.
\end{proof}

\begin{lemma}[\textbf{Stetigkeit parameterabhängiger Integrale}]
Seien $P$ ein topologischer Raum und $a<b$ reelle Zahlen. Sei $f\colon P\times[a,b]\to E$ eine stetige Funktion. Wenn das schwache Integral
\[g(p)\coloneqq\int_a^b f(p,t)\, dt\]
für alle $p\in P$ in $E$ existiert, dann ist $g\colon P\to E$ stetig.
\end{lemma}

\begin{proof}
Siehe \cite[Lemma 1.1.11]{Gl}.
\end{proof}

\begin{proof}[Beweis von Satz \ref{xx}]
Wir beweisen den Satz für $k<\infty$ und offenes $U$; es ist nicht schwierig zu sehen, wie daraus der allgemeine Fall folgt. Die Notwendigkeit ist klar. Angenommen, alle partiellen Ableitungen von $f$ bis zur Ordnung $k$ existieren und sind stetig. Wir zeigen per Induktion nach $k$, dass alle Richtungsableitungen von $f$ bis zur Ordnung $k$ existieren und \eqref{gl.x} erfüllen; dann sind diese insbesondere stetig, und $f$ ist somit eine $C^k$-Funktion.

Seien $k=1$, $x=(x_1,\ldots,x_n)\in U$ und $y=(y_1,\ldots,y_n)\in\R^n$. 
Wenn der Betrag einer reellen Zahl $t\neq 0$ so klein ist, dass $B_{\lVert ty\rVert}(x)$ im Falle $y\neq 0$ in $U$ liegt, so gilt nach dem Mittelwertsatz
\begin{align*}
&f(x_1+ty_1,\ldots,x_j+ty_j,x_{j+1},\ldots,x_n)-f(x_1+ty_1,\ldots,x_{j-1}+ty_{j-1},x_{j},\ldots,x_n)\\
&=ty_j\int_0^1\frac{\partial f}{\partial x_j}(x_1+ty_1,\ldots,x_{j-1}+ty_{j-1},x_j+sty_j,x_{j+1},\ldots,x_n)\, ds\end{align*}
für alle $j\in\{1,\ldots,n\}.$
Daraus folgt
\begin{align*}
&f(x+ty)-f(x)\\ &=\sum_{j=1}^n (f(x_1+ty_1,\ldots,x_j+ty_j,x_{j+1},\ldots,x_n) -f(x_1+ty_1,\ldots,x_{j-1}+ty_{j-1},x_{j},\ldots,x_n))\\
&=t\sum_{j=1}^n y_j\int_0^1\frac{\partial f}{\partial x_j}(x_1+ty_1,\ldots,x_{j-1}+ty_{j-1},x_j+sty_j,x_{j+1},\ldots,x_n)\, ds,
\end{align*}
also
\begin{align*}
&\frac{f(x+ty)-f(x)}{t}-\sum_{j=1}^n y_j \frac{\partial f}{\partial x_j}(x)\\
&=\sum_{j=1}^n y_j\int_0^1\left(\frac{\partial f}{\partial x_j}(x_1+ty_1,\ldots,x_{j-1}+ty_{j-1},x_j+sty_j,x_{j+1},\ldots,x_n)-\frac{\partial f}{\partial x_j}(x)\right) ds.
\end{align*}
Mit der Stetigkeit parameterabhängiger Integrale konvergiert die untere Summe gegen $0$ für $t\to 0$. Somit erhalten wir
\[df(x,y)=\sum_{j=1}^n y_j \frac{\partial f}{\partial x_j}(x).\]

Nun sei $k>1$. Wir nehmen an, dass alle Richtungsableitungen von $f$ bis zur Ordnung $k-1$ existieren und \eqref{gl.x} erfüllen. Sind $j_1,\ldots,j_{k-1}\in\{1,\ldots,n\}$ gegeben, so existieren alle partiellen Ableitungen von $\frac{\partial^{k-1} f}{\partial x_{j_{k-1}}\cdots\partial x_{j_1}}$ der Ordnung 1 und sind stetig, also gilt nach Induktionsanfang
\[D_y\left(\frac{\partial^{k-1} f}{\partial x_{j_{k-1}}\cdots\partial x_{j_1}}\right)=\sum_{j_k=1}^n y_{j_k} \frac{\partial^{k} f}{\partial x_{j_{k}}\cdots\partial x_{j_1}}\]
für alle $y=(y_1,\ldots,y_n)\in\R^n$. Damit folgt
\begin{align*}
\sum_{j_1,\ldots,j_k=1}^n y_{1,j_1}\cdots y_{k,j_k}\frac{\partial^{k} f}{\partial x_{j_k}\cdots\partial x_{j_1}}
&=\sum_{j_1,\ldots,j_{k-1}=1}^n y_{1,j_1}\cdots y_{k-1,j_{k-1}}\bigg(\sum_{j_k=1}^n y_{k,j_k} \frac{\partial^{k} f}{\partial x_{j_{k}}\cdots\partial x_{j_1}}\bigg)\\
&=\sum_{j_1,\ldots,j_{k-1}=1}^n y_{1,j_1}\cdots y_{k-1,j_{k-1}}D_{y_k}\left(\frac{\partial^{k-1} f}{\partial x_{j_{k-1}}\cdots\partial x_{j_1}}\right)\\
&=D_{y_k}\bigg(\sum_{j_1,\ldots,j_{k-1}=1}^n y_{1,j_1}\cdots y_{k-1,j_{k-1}}\frac{\partial^{k-1} f}{\partial x_{j_{k-1}}\cdots\partial x_{j_1}}\bigg)\\
&=D_{y_k}D_{y_{k-1}}\cdots D_{y_1}f
\end{align*}
für alle $y_i=(y_{i,1},\ldots,y_{i,n})\in\R^n$ mit $i\in\{1,\ldots,k\}$, was den Beweis beendet.
\end{proof}

\subsection*{Leibnizformel, Kettenregel und Fa\`a-di-Bruno-Formel}
\addcontentsline{toc}{subsection}{\protect\numberline{}Leibnizformel, Kettenregel und Fa\`a-di-Bruno-Formel}

Wir behandeln hier nützliche Werkzeuge für die Berechnung von Ableitungen.

\begin{satz}[\textbf{Leibnizformel}]
Es seien $k\in\N\cup\{\infty\}$, $\phi\colon E_1\times E_2\to F$ eine stetige bilineare Abbildung zwischen lokalkonvexen Räumen und $f_1\colon U\to E_1$ sowie ${f_2\colon U\to E_2}$ beide $C^k$-Funktionen auf einer lokalkonvexen Teilmenge $U\subseteq\R^n$ mit dichtem Inneren. Dann ist \[\phi(f_1,f_2)\colon U\to F,\quad x\mapsto \phi(f_1(x),f_2(x))\] eine $C^k$-Funktion, und für alle $\lvert\alpha\rvert\leq k$ gilt
\begin{equation}\label{gl.a}
\p(\phi(f_1,f_2))=\sum_{\beta\leq\alpha}\binom{\alpha}{\beta} \phi(\partial^{\beta}f_1,\partial^{\alpha-\beta}f_2).\end{equation}
\end{satz}

\begin{proof}
Wir zeigen den Satz für offenes $U$; es ist leicht, damit auf den allgemeinen Fall zu schließen. Für $\alpha=0$ ist \eqref{gl.a} erfüllt. Sei $j\in\{1,\ldots,n\}$. Für alle $x\in U$ und $t\in\R$ mit $x+t e_j\in U$ gilt
\begin{align*}&\frac{1}{t}(\phi(f_1(x+t e_j),f_2(x+t e_j))-\phi(f_1(x),f_2(x)))\\
&=\frac{1}{t}(\phi(f_1(x+t e_j)-f_1(x),f_2(x+t e_j))+\phi(f_1(x),f_2(x+t e_j)-f_2(x)))\\ 
&=\phi\left(\frac{f_1(x+t e_j)-f_1(x)}{t},f_2(x+t e_j)\right)+\phi\left(f_1(x),\frac{f_2(x+t e_j)-f_2(x)}{t}\right)\\
&\to\phi(\partial_j f_1(x),f_2(x))+\phi(f_1(x),\partial_j f_2(x))\ \text{für }t\to 0,\end{align*}
womit \eqref{gl.a} für $\alpha=e_j$ gezeigt ist.
Angenommen, für ein $\alpha\in(\N_0)^n$ mit $\al<k$ ist \eqref{gl.a} erfüllt. Für alle $j\in\{1,\ldots,n\}$ und $\beta\in(\N_0)^n$ gilt
$\binom{\alpha+e_j}{\beta}=\binom{\alpha}{\beta}+\binom{\alpha}{\beta-e_j}$ und somit
\begin{align*}
&\sum_{\beta\leq\alpha+e_j}\binom{\alpha+e_j}{\beta}\phi(\partial^{\beta}f_1,\partial^{\alpha+e_j-\beta}f_2)\\
&=\sum_{\beta\leq\alpha}\binom{\alpha}{\beta}\phi(\partial^{\beta}f_1,\partial^{\alpha+e_j-\beta}f_2)
+\sum_{e_j\leq\beta\leq\alpha+e_j}\binom{\alpha}{\beta-e_j}\phi(\partial^{\beta}f_1,\partial^{\alpha+e_j-\beta}f_2)\\
&=\sum_{\beta\leq\alpha}\binom{\alpha}{\beta}(\phi(\partial^{\beta}f_1,\partial^{\alpha+e_j-\beta}f_2)+\phi(\partial^{\beta+e_j}f_1,\partial^{\alpha-\beta}f_2))\\
&=\sum_{\beta\leq\alpha}\binom{\alpha}{\beta}\partial_j(\phi(\partial^{\beta}f_1,\partial^{\alpha-\beta}f_2))
=\partial_j\bigg(\sum_{\beta\leq\alpha}\binom{\alpha}{\beta}\phi(\partial^{\beta}f_1,\partial^{\alpha-\beta}f_2)\bigg)\\
&=\partial_j\partial^{\alpha}(\phi(f_1,f_2))=\partial^{\alpha+e_j}(\phi(f_1,f_2)).
\end{align*}
Per Induktion folgt \eqref{gl.a} für alle $\vert\alpha\vert\leq k$. Da alle partiellen Ableitungen von $\phi(f_1,f_2)$ bis zur Ordnung $k$ stetig sind, ist $\phi(f_1,f_2)$ eine $C^k$-Funktion.
\end{proof}

\begin{satz}[\textbf{Kettenregel für partielle Ableitungen}]
Es seien $k\in\N\cup\{\infty\}$, $U\subseteq\R^s$ und $V\subseteq\R^t$ lokalkonvexe Teilmengen mit dichtem Inneren und $g=(g_1,\ldots,g_t)\colon U\to V\subseteq\R^t$ sowie $f\colon V\to E$ beide $C^k$-Funktionen.
Dann gilt
\[\partial_j(f\circ g)(x)=\sum_{i=1}^t\partial_j g_i(x)\partial_if(g(x))\]
für alle $j\in\{1,\ldots,s\}$ und $x\in U$.
\end{satz}

\begin{proof}
Folgt aus \cite[Proposition 1.4.10]{Gl}.
\end{proof}

Für die Berechnung höherer Ableitungen $\partial^{\alpha}(f\circ g)$ oder $d^{(l)}(f\circ g)$ von Kompositionen dient die Fa\`a-di-Bruno-Formel in lokalkonvexen Räumen, welche wir im Folgenden von Glöckner und Neeb \cite{Gl} zitieren. Es sei auch die Arbeit von Clark und Houssineau \cite{CH} erwähnt.

Für $k\in\N$ und $j\in\{1,\ldots,k\}$ schreiben wir $P_{k,j}$ für die Menge aller Partitionen $P=\{I_1,\ldots,I_j\}$ der Menge $\{1,\ldots,k\}$ in $j$ nichtleere disjunkte Teilmengen $I_1,\ldots,I_j\subseteq\{1,\ldots,k\}$. Es gilt also $I_a\cap I_b=\emptyset$ für $a\neq b$ und $\{1,\ldots,k\}=\bigcup_{a=1}^j I_a$. Wir schreiben $\lvert I_a\rvert$ für die Anzahl der Elemente der Menge $I_a$. Sind $y=(y_1,\ldots,y_k)\in E^k$ und $I\subseteq\{1,\ldots,k\}$ eine nichtleere Teilmenge, etwa $I=\{i_1,\ldots,i_l\}$ mit $i_1<\ldots<i_l$, so notieren wir
\[y_I\coloneqq (y_{i_1},\ldots,y_{i_l})\in E^l.\]

\newpage

\begin{satz}[\textbf{Fa\`a-di-Bruno-Formel}]\label{yy}
Es seien $E,F$ und $H$ lokalkonvexe Räume, $U\subseteq E$ und $V\subseteq F$ lokalkonvexe Teilmengen mit dichtem Inneren, $k\in\N$ und $g\colon U\to V\subseteq F$ sowie $f\colon V\to H$ beide $C^k$-Funktionen. Dann gilt
\[d^{(k)}(f\circ g)(x,y)=\sum_{j=1}^k\sum_{P\in P_{k,j}}d^{(j)}f\big(g(x),d^{(\lvert I_1\rvert)}g(x,y_{I_1}),\ldots,d^{(\lvert I_j\rvert)}g(x,y_{I_j})\big)\]
für alle $x\in U$ und $y=(y_1,\ldots,y_k)\in E^k$, wobei wir $P=\{I_1,\ldots,I_j\}$ schreiben und die rechte Seite wohldefiniert ist, also unabhängig von der Reihenfolge der $I_1,\ldots,I_j$.
\end{satz}

\begin{proof}
Analog wie in \cite[Theorem 1.3.18]{Gl}; dort wird der Satz für offene Teilmengen $U\subseteq E$ und $V\subseteq F$ bewiesen.
\end{proof}

Wir formulieren noch die Fa\`a-di-Bruno-Formel für partielle Ableitungen. Dafür verwenden wir sogenannte Fa\`a-di-Bruno-Polynome.

\begin{definition}\label{zz}
Es seien $s,t\in\N$ und $\alpha=(\alpha_1,\ldots,\alpha_s)\neq 0$ in $(\N_0)^s$. Dazu seien $k\coloneqq\al$ und $j_1\leq\ldots \leq j_{k}$ diejenigen Zahlen aus $\{1,\ldots,s\}$, für die \[\alpha_r=\lvert\{1\leq l\leq k\colon j_l=r\}\rvert\] für alle $r\in\{1,\ldots,s\}$ gilt. Für jede nichtleere Teilmenge $I=\{i_1,\ldots,i_m\}\subseteq\{1,\ldots,k\}$ seien \[(I)_r\coloneqq \lvert\{1\leq l\leq m\colon j_{i_l}=r\}\rvert\] für $r\in\{1,\ldots,s\}$ und \[(I)\coloneqq ((I)_1,\ldots,(I)_s)\in(\N_0)^s.\]
Für alle $\beta=(\beta_1,\ldots,\beta_t)\neq 0$ in $(\N_0)^t$ mit $j\coloneqq\bet\leq\al$ definieren wir die \ee{Fa\`a-di-Bruno-Polynome}
\begin{align*}p_{\alpha,\beta}\colon \R^{\{\gamma\in(\N_0)^s\colon \vert\gamma\vert\leq\al\}\times t}\to \R\end{align*}
durch
\[p_{\alpha,\beta}(x)\coloneqq\sum_{\substack{\{I_1,\ldots,I_j\}\\ \in P_{k,j}}}\sum_{\substack{1\leq i_1,\ldots,i_j\leq t\colon\\ \lvert\{l\colon i_l=r\}\rvert=\beta_r}}x_{(I_1),i_1}\cdots x_{(I_j),i_j}\]
für $x=(x_{\gamma,i})_{\gamma,i}
\in \R^{\{\gamma\in(\N_0)^s\colon \vert\gamma\vert\leq\al\}\times t}$, wobei die zweite Summe auf der rechten Seite unabhängig von der Reihenfolge der $I_1,\ldots, I_j$ ist. Wir setzen
\[p_{\alpha,0}\coloneqq 0\quad\text{und}\quad p_{0,0}\coloneqq 1.\]
\end{definition}

\newpage
\begin{satz}[\textbf{Fa\`a-di-Bruno-Formel für partielle Ableitungen}]\label{d}
Es seien $U\subseteq\R^s$ wie auch $V\subseteq\R^t$ lokalkonvexe Teilmengen mit dichtem Inneren, $\alpha\in(\N_0)^s$ und $g\colon U\to V\subseteq\R^t$ sowie $f\colon V\to E$ beide $C^{\al}$-Funktionen. Dann gilt
\[\partial^{\alpha}(f\circ g)(x)=\sum_{\bet\leq\al}p_{\alpha,\beta}\left((\partial^{\gamma}g(x))_{\vert\gamma\vert\leq\al}\right)\partial^{\beta} f(g(x))\]
für alle $x\in U$, wobei hier $\beta\in(\N_0)^t$ und $\gamma\in(\N_0)^s$ seien.\ftnote{Wir identifizieren $\R^{\{\gamma\in(\N_0)^s\colon \vert\gamma\vert\leq\al\}\times t}$ und $(\R^t)^{\{\gamma\in(\N_0)^s\colon \vert\gamma\vert\leq\al\}}$ mittels des Isomorphismus $(x_{\gamma,i})_{\gamma,i}\mapsto ((x_{\gamma,1},\ldots,x_{\gamma,t}))_{\gamma}$.}
\end{satz}

\begin{proof}
Für $\alpha=0$ ist ist die Aussage klar. Sei also $\alpha\neq 0$. Wir übernehmen die Notation aus Definition \ref{zz} und setzen $e\coloneqq(e_{j_1},\ldots,e_{j_k})\in(\R^s)^k$. Mit Satz \ref{yy} erhalten wir
\begin{align*}\partial^{\alpha}(f\circ g)(x)
&=d^{(k)}(f\circ g)(x,e)\\
&=\sum_{j=1}^k
\sum_{P\in P_{k,j}}d^{(j)}f\big(g(x),d^{(\lvert I_1\rvert)}g(x,e_{I_1}),\ldots,d^{(\lvert I_j\rvert)}g(x,e_{I_j})\big)\\
&=\sum_{j=1}^k
\sum_{P\in P_{k,j}}d^{(j)}f\big(g(x),\partial^{(I_1)}g(x),\ldots,\partial^{(I_j)}g(x)\big)\\
&=\sum_{j=1}^k
\sum_{P\in P_{k,j}}\sum_{1\leq i_1,\ldots,i_j\leq t}\partial^{(I_1)}g_{i_1}(x)\cdots\partial^{(I_j)}g_{i_j}(x)(\partial_{i_j}\cdots\partial_{i_1} f)(g(x))\\
&=\sum_{j=1}^k\sum_{\bet=j}\bigg(
\sum_{P\in P_{k,j}}\sum_{\substack{1\leq i_1,\ldots,i_j\leq t\colon\\ \lvert\{l\colon i_l=r\}\rvert=\beta_r}}\partial^{(I_1)}g_{i_1}(x)\cdots\partial^{(I_j)}g_{i_j}(x)\bigg)\partial^{\beta}f(g(x))\\
&=\sum_{\bet\leq\al}p_{\alpha,\beta}\big((\partial^{\gamma}g(x))_{\vert\gamma\vert\leq\al}\big)\partial^{\beta}f(g(x))
\end{align*}
für alle $x\in U$, wobei wir $P=\{I_1,\ldots,I_j\}$ schreiben.
\end{proof}

\subsection*{Die kompakt-offene $C^k$-Topologie}
\addcontentsline{toc}{subsection}{\protect\numberline{}Die kompakt-offene $C^k$-Topologie}

\begin{definition}\label{r}
Sind $k\in\N_0\cup\{\infty\}$ und $U\subseteq\R^n$ eine lokalkonvexe Teilmenge mit dichtem Inneren, so versehen wir $C^k(U,E)$ mit der \ee{kompakt-offenen $C^k$-Topologie}; das ist die Hausdorffsche lokalkonvexe Topologie, welche durch die Halbnormen $\lVert \cdot\rVert_{C^l,q,K}\colon C^k(U,E)\to [0,\infty[$,
\[\lVert f\rVert_{C^l,q,K}\coloneqq \max_{\al\leq l}\sup_{x\in K} q(\partial^{\alpha}f(x))\]
für $l\in\N_0$ mit $l\leq k$, stetigen Halbnormen $q$ auf $E$ und kompakten Teilmengen $K\subseteq U$ erzeugt wird.
\end{definition}

\begin{lemma}\label{pp}
Es seien $k\in\N_0\cup\{\infty\}$ und $g\colon U\to V$ eine $C^k$-Funktion zwischen lokalkonvexen Teilmengen $U\subseteq\R^s$ und $V\subseteq\R^t$ mit dichtem Inneren. Weiter seien $l\in\N_0$ mit $l\leq k$, $q$ eine stetige Halbnorm auf $E$ und $K\subseteq U$ kompakt. Dann gibt es ein $C>0$ derart, dass 
\begin{align*}\Vert f\circ g\Vert_{C^l,q,K}\leq C\Vert f\Vert_{C^l,q,g(K)}\end{align*}
für alle $f\in C^k(V,E)$. Somit ist die lineare Abbildung
\[g^{*}\colon C^k(V,E)\to C^k(U,E),\quad f\mapsto f\circ g\]
stetig. Insbesondere ist im Falle $s=t$ und $U\subseteq V$ die Einschränkung
\[C^k(V,E)\to C^k(U,E),\quad f\mapsto f\vert_U\]
stetig.
\end{lemma}

\begin{proof}
Für alle $\alpha\in(\N_0)^s$ mit $\al\leq l$ und $\beta\in(\N_0)^t$ mit $\bet\leq\al$ seien $p_{\alpha,\beta}$ die Fa\`a-di-Bruno-Polynome aus Definition \ref{zz}. Wir setzen
\[M\coloneqq\max_{\bet\leq\al\leq l}\sup_{x\in K}\left\vert p_{\alpha,\beta}\left((\partial^{\gamma}g(x))_{\vert\gamma\vert\leq\al}\right)\right\vert\]
und $C\coloneqq (l+1)^nM$. Ein $f\in C^k(V,E)$ gegeben, gilt für alle $\al\leq l$ und $x\in K$ mit der Fa\`a-di-Bruno-Formel
\begin{align*} q((\p(f\circ g))(x))
\leq\sum_{\bet\leq\al}\underbrace{\left\vert p_{\alpha,\beta}
\left((\partial^{\gamma}g(x))_{\vert\gamma\vert\leq\al}\right)\right\vert}_{\leq M} \underbrace{q(\partial^{\beta}f(g(x)))}_{\leq \Vert f\Vert_{C^l,q,g(K)}}
\leq C\Vert f\Vert_{C^l,q,g(K)}.
\end{align*}
Bei Übergang zum Supremum erhalten wir die gewünschte Abschätzung.
\end{proof}

\section{Vektorwertige Whitneysche $k$-Jets}
Der Begriff des Whitneyschen $k$-Jets verallgemeinert den einer $C^k$-Funktion und wird durch den Satz von Taylor motiviert.

\subsection*{Der Satz von Taylor}
\addcontentsline{toc}{subsection}{\protect\numberline{}Der Satz von Taylor}

Jede $C^k$-Funktion auf $\R^n$ lässt sich nahe einem vorgegebenen Punkt durch ein Polynom approximieren.

\begin{satz}[\textbf{Satz von Taylor}]\label{dd}
Es seien $k\in\N_0\cup\{\infty\}$ und $f\colon U\to E$ eine $C^k$-Funktion auf einer lokalkonvexen Teilmenge $U\subseteq\R^n$ mit dichtem Inneren. Für gegebene $l\in\N_0$ mit $l\leq k$ und $x,y\in U$ schreiben wir
\[f(x)=\sum_{\al\leq l}\frac{(x-y)^{\alpha}}{\alpha!}\p f(y)+R^l_{y}f(x).\]
\begin{itemize}
\item[(a)] Sind $l>0$ und $x,y\in U$ derart, dass $y+t(x-y)\in U$ für alle $t\in[0,1]$, so gilt
\[R^l_{y} f(x)=l\int_0^1\bigg(\sum_{\vert\alpha\vert=l}\frac{(x-y)^{\alpha}}{\alpha!}(1-t)^{l-1}(\partial^{\alpha}f(y+t(x-y))-\partial^{\alpha}f(y))\bigg) dt\]
(vgl. \cite[Theorem 1.6.14]{Gl}).
\item[(b)] Für alle $l\in\N_0$ mit $l\leq k$, $\al\leq l$, stetigen Halbnormen $q$ auf $E$, kompakten Teilmengen $K\subseteq U$ und $\eps>0$ gibt es ein $\delta>0$ derart, dass 
\[\frac{q(R^{l-\al}_{y}\p f(x))}{\lVert x-y\rVert^{l-\vert\alpha\vert}}\leq\eps\]
für alle $x,y\in K$ mit $0<\lVert x-y\rVert<\delta$.
\end{itemize}
\end{satz}

Für den Beweis benötigen wir den folgenden Spezialfall von \cite[Proposition 1.6.1 (b)]{Gl}:

\begin{lemma}\label{e}
Es seien $l\in\N$ und $\gamma\colon [0,1]\to E$ eine $C^l$-Kurve. Dann gilt
\[\gamma(1)=\sum_{j=0}^l\frac{1}{j!}\gamma^{(j)}(0)+R^l\gamma\]
mit dem Restglied
\begin{align}\label{gl.c}R^l\gamma=\frac{1}{(l-1)!}\int_{0}^{1}(1-t)^{l-1}(\gamma^{(l)}(t)-\gamma^{(l)}(0))\, dt.\end{align}
\end{lemma}

\begin{proof}[Beweis von Satz \ref{dd}]
(a) Für die $C^l$-Kurve \[\gamma\colon[0,1]\to E,\quad t\mapsto f(y+t(x-y))\] gilt
\[f(x)=\gamma(1)=\sum_{j=0}^l\frac{1}{j!}\gamma^{(j)}(0)+R^l \gamma\]
mit $R^l\gamma$ wie in \eqref{gl.c}. Wir schreiben $x=(x_1,\ldots,x_n)$ und $y=(y_1,\ldots,y_n)$.
Die Kettenregel liefert
\[\gamma'(t)=\sum_{i=1}^n(x_i-y_i)\partial_i f(y+t(x-y))\]
für alle $t\in[0,1]$, und für $j\in\{1,\ldots,l\}$ folgt per Induktion
\begin{align*}\gamma^{(j)}(t)
&=\sum_{i_1,\ldots,i_j=1}^n(x_{i_1}-y_{i_1})\cdots(x_{i_j}-y_{i_j})(\partial_{i_j}\cdots\partial_{i_1} f)(y+t(x-y))\\
&=\sum_{\vert\alpha\vert=j}\binom{j}{\alpha_1}\binom{j-\alpha_1}{\alpha_2}\cdots\binom{j-\alpha_1-\ldots-\alpha_{n-1}}{\alpha_n}(x-y)^{\alpha}\partial^{\alpha}f(y+t(x-y)),\end{align*}
wobei $\binom{j}{\alpha_1}\binom{j-\alpha_1}{\alpha_2}\cdots\binom{j-\alpha_1-\ldots-\alpha_{n-1}}{\alpha_n}=\frac{j!}{\alpha!}$ ist.
Somit gilt
\[f(x)=\sum_{\vert\alpha\vert\leq l}\frac{(x-y)^{\alpha}}{\alpha!}\partial^{\alpha}f(y)+R^l \gamma\] mit
\begin{align*}R^l\gamma
&=\frac{1}{(l-1)!}\int_0^1(1-t)^{l-1}\bigg(\sum_{\vert\alpha\vert=l}\frac{l!(x-y)^{\alpha}}{\alpha!}(\partial^{\alpha}f(y+t(x-y))-\partial^{\alpha}f(y))\bigg)dt\\
&=l\int_0^1\bigg(\sum_{\vert\alpha\vert=l}\frac{(x-y)^{\alpha}}{\alpha!}(1-t)^{l-1}(\partial^{\alpha}f(y+t(x-y))-\partial^{\alpha}f(y))\bigg)dt.\end{align*}

(b) Sei $l\in\N_0$ mit $l\leq k$. Angenommen, die Aussage gilt nicht für $\alpha=0$. Dann gibt es eine stetige Halbnorm $q$ auf $E$, eine kompakte Teilmenge $K\subseteq U$ und ein $\eps>0$, so dass wir für alle $m\in\N$ gewisse $x_m,y_m\in K$ derart wählen können, dass $0<\Vert x_m-y_m\Vert<\frac{1}{m}$ und \begin{align}\label{gl.ee}\frac{q(R_{y_m}^l f(x_m))}{\Vert x_m-y_m\Vert^l}>\eps. \end{align}
Wegen der Kompaktheit von $K$ hat $(x_m)_{m\in\N}$ eine konvergente Teilfolge $(x_{m_j})_{j\in\N}$. Nach Ersetzen von $(x_m)_{m\in\N}$ und $(y_m)_{m\in\N}$ durch $(x_{m_j})_{j\in\N}$ bzw. $(y_{m_j})_{j\in\N}$ können wir annehmen, dass $(x_m)_{m\in\N}$ gegen einen Grenzwert $z\in K$ konvergiert. Dann ist $z$ auch der Grenzwert von $(y_m)_{m\in\N}$. Im Falle $l=0$ gilt
\[q(R_y^0f(x))=q(f(x)-f(y))\to q(f(z)-f(z))=0\ \text{für }(x,y)\to(z,z),\]
ein Widerspruch zu \eqref{gl.ee}. Nehmen wir nun an, dass $l>0$ ist. Sei $V$ eine konvexe $z$-Umgebung in $U$. Für alle $\al\leq l$ und $x,y\in V$ existiert das schwache Integral
\[I^{\alpha}(x,y)\coloneqq l\int_0^1(1-t)^{l-1}(\partial^{\alpha}f(y+t(x-y))-\partial^{\alpha}f(y))\,dt\]
in der Vervollständigung $\widetilde{E}$ von $E$
(vgl. \cite[Proposition 1.1.15]{Gl}). Wegen der Stetigkeit parameterabhängiger Integrale ist die Abbildung
\[I^{\alpha}\colon V\times V\to \widetilde{E},\quad (x,y)\mapsto I^{\alpha}(x,y)\]
stetig, also gilt $I^{\alpha}(x,y)\to I^{\alpha}(z,z)=0$ für $(x,y)\to (z,z)$. Nach (a) gilt
\[R_y^lf(x)=\sum_{\al=l}  \frac{(x-y)^{\alpha}}{\alpha!} I^{\alpha}(x,y)\quad \text{für alle }(x,y)\in V\times V.\] Bezeichnen wir die Fortsetzung von $q$ zu einer stetigen Halbnorm auf $\widetilde{E}$ ebenfalls mit $q$, so folgt
\[\frac{q(R_y^lf(x))}{\Vert x-y\Vert^l}
\leq\sum_{\al=l}\underbrace{\frac{\vert (x-y)^{\alpha}\vert}{\Vert x-y\Vert^l}}_{\leq 1} q(I^{\alpha}(x,y))\to 0\]
für $(x,y)\to (z,z) \text{ in }V\times V \text{ mit }x\neq y$,
ein Widerspruch zu \eqref{gl.ee}. Somit gilt die Aussage für $\alpha=0$. Da aber $\p f$ eine $C^{l-\al}$-Funktion für alle $\al\leq l$ ist, folgt mit dem Vorigen die Behauptung.
\end{proof}

\begin{beispiel}\label{ee}
Sei $l\in\N_0$. Eine Funktion $p\colon\R^n\to E$ der Gestalt \[p(x)=\sum_{\al\leq l}x^{\alpha}v_{\alpha}\] mit Vektoren $v_{\alpha}\in E$ heißt \ee{Polynom vom Grad $\leq l$}. Wie für reellwertige Polynome sieht man, dass $p\colon\R^n\to E$ glatt ist mit $\p p=0$ für alle $\al>l$. Mit dem Satz von Taylor folgt \[p(x)=\sum_{\al\leq l}\frac{(x-y)^{\alpha}}{\alpha!}\p p(y)\] für alle $x,y\in\R^n$.
\end{beispiel}

\subsection*{Taylorpolynome von $k$-Jets}
\addcontentsline{toc}{subsection}{\protect\numberline{}Taylorpolynome von $k$-Jets}

Whitneys Ideen in \cite[\S\ 3 und \S\ 6]{Wh} können wir wie folgt für vektorwertige Funktionen formulieren.

Es seien $k\in\N_0\cup\{\infty\}$ und $U\subseteq\R^n$ eine Teilmenge.

\begin{definition}
Eine Familie $f=(f_{\alpha})_{\al\leq k}\in\prod_{\al\leq k}E^U=:\J^k(U,E)$ von Funktionen $f_{\alpha}\colon U\to E$ nennt man \ee{$k$-Jet} auf $U$ mit Werten in $E$. Wir schreiben \[f\vert_V\coloneqq (f_{\alpha}\vert_V)_{\al\leq k}\]
für jede Teilmenge $V\subseteq U$ und 
\[\text{pr}_l(f)\coloneqq(f_{\alpha})_{\al\leq l}\]
für jedes $l\leq k$. Für alle $\al\leq k$ setzen wir
\[\p f\coloneqq (f_{\alpha+\beta})_{\bet\leq k-\al}.\]
Sind ein $l\in\N_0$ mit $l\leq k$ und ein $y\in U$ gegeben, so nennt man das Polynom
$T_{y}^l f\colon\R^n\to E,$ \[T_{y}^l f(x)\coloneqq\sum_{\vert\alpha\vert\leq l}\frac{(x-y)^{\alpha}}{\alpha!}f_{\alpha}(y)\]
das \ee{Taylorpolynom $l$-ter Ordnung} von $f$ in $y$ und die Funktion $R_{y}^l f\colon U\to E,$ 
\[R^l_{y}f(x)\coloneqq f_0(x)-T_{y}^lf (x)\]
das zugehörige \ee{Taylorsche Restglied}. Es gilt also
\[f_{\alpha}(x)=T_{y}^{l-\al}\p f(x)+R_{y}^{l-\al}\p f(x)\]
für alle $\al\leq l$ und $x\in U$.
\end{definition}

\begin{lemma}\label{c}
\begin{itemize}
\item[(a)] Für alle $\al\leq l$ gilt
\[\p T_{y}^l f=T_{y}^{l-\al}\p f.\]
\item[(b)] Sind $s,t\in\R$ und $g\in\J^k(U,E)$ ein weiterer $k$-Jet, dann gilt
\[T_{y}^l(sf+tg)=sT_{y}^lf+tT_{y}^lg.\]
\item[(c)] Für alle $y,z\in U$ und $x\in\R^n$ gilt
\[T_{y}^lf(x)=T_{z}^lf(x)+\sum_{\al\leq l}\frac{(x-y)^{\alpha}}{\alpha!}R_{z}^{l-\al}\p f(y).\]
\end{itemize}
\end{lemma}

\begin{proof}
\begin{itemize}
\item[(a)] Aus $\frac{\partial^{\alpha}}{\partial x^{\alpha}}(x-y)^{\beta}=\frac{\beta!}{(\beta-\alpha)!}(x-y)^{\beta-\alpha}$ für $\bet\geq\al$ folgt
\begin{align*}\frac{\partial^{\alpha}}{\partial x^{\alpha}}T_{y}^lf(x)
&=\sum_{\bet\leq l}\frac{\frac{\partial^{\alpha}}{\partial x^{\alpha}}(x-y)^{\beta}}{\beta!}f_{\beta}(y)
=\sum_{\al\leq\bet\leq l}\frac{(x-y)^{\beta-\alpha}}{(\beta-\alpha)!}f_{\beta}(y)\\
&=\sum_{\bet\leq l-\al}\frac{(x-y)^{\beta}}{\beta!} f_{\alpha+\beta}(y)=T_{y}^{l-\al}\p f.\end{align*}
\item[(b)] Ist offensichtlich.
\item[(c)] Es gilt 
\begin{align*} T_{y}^lf(x)
&=\sum_{\vert\alpha\vert\leq l}\frac{(x-y)^{\alpha}}{\alpha!}f_{\alpha}(y)\\
&=\sum_{\vert\alpha\vert\leq l}\frac{(x-y)^{\alpha}}{\alpha!}(T_{z}^{l-\al}\p f(y)+R_{z}^{l-\al}\p f(y))\\
&=\sum_{\vert\alpha\vert\leq l}\frac{(x-y)^{\alpha}}{\alpha!}\p T_{z}^l f(y)+\sum_{\vert\alpha\vert\leq l}\frac{(x-y)^{\alpha}}{\alpha!}R_{z}^{l-\al}\p f(y)\\
&=T_{z}^l f(x)+\sum_{\vert\alpha\vert\leq l}\frac{(x-y)^{\alpha}}{\alpha!}R_{z}^{l-\al}\p f(y),\end{align*}
wobei wir für die letzte Gleichheit die Taylordarstellung von $T_{z}^l f(x)$ (vgl. Beispiel \ref{ee}) verwendet haben.
\end{itemize}
\end{proof}

\subsection*{Whitneysche $k$-Jets}
\addcontentsline{toc}{subsection}{\protect\numberline{}Whitneysche $k$-Jets}

Es seien wieder stets $k\in\N_0\cup\{\infty\}$ und $U\subseteq\R^n$ eine Teilmenge.

\begin{definition}\label{p}
Ein $k$-Jet $f=(f_{\alpha})_{\vert\alpha\vert\leq k}\in\J^k(U,E)$ heißt \ee{Whitneyscher $k$-Jet}, wenn für alle $l\in\N_0$ mit $l\leq k$, $\al\leq l$, stetigen Halbnormen $q$ auf $E$, kompakten Teilmengen $K\subseteq U$ und $\eps>0$ ein $\delta>0$ derart existiert, dass
\[\frac{q(R^{l-\al}_{y}\p f(x))}{\lVert x-y\rVert^{l-\vert\alpha\vert}}\leq\eps\]
für alle $x,y\in K$ mit $0<\lVert x-y\rVert<\delta$.

Wir schreiben $\EE^k(U,E)\subseteq \J^k(U,E)$ für den Untervektorraum aller Whitneyschen $k$-Jets auf $U$ mit Werten in $E$.\ftnote{Für $U=\emptyset$ gilt $\EE^k(\emptyset,E)=\J^k(\emptyset,E)\cong \{0\}$.}
\end{definition}

\begin{bemerkung} Sei $f\in\J^k(U,E)$.
\begin{itemize}
\item[(a)] Aus $f\in\EE^k(U,E)$ folgt $f\vert_V\in\EE^k(V,E)$ für jede Teilmenge $V\subseteq U$.
\item[(b)] Genau dann gilt $f\in\EE^k(U,E)$, wenn $f\vert_K\in\EE^k(K,E)$ für jede kompakte Teilmenge $K\subseteq U$.
\item[(c)] Genau dann gilt $f\in\EE^{k}(U,E)$, wenn
$\text{pr}_l(f)\in\EE^{l}(U,E)$ für alle $l\in\N_0$ mit $l\leq k$.
\end{itemize}
\end{bemerkung}

\begin{lemma}
Seien $k<\infty$ und $f\in\J^k(U,E)$. Genau dann gilt $f\in\EE^k(U,E)$, wenn für alle $\al\leq k$, stetigen Halbnormen $q$ auf $E$, kompakten Teilmengen $K\subseteq U$ und $\eps>0$ ein $\delta>0$ derart existiert, dass
\[\frac{q(R^{k-\al}_{y}\p f(x))}{\lVert x-y\rVert^{k-\vert\alpha\vert}}\leq\eps\]
für alle $x,y\in K$ mit $0<\lVert x-y\rVert<\delta$.
\end{lemma}

\begin{proof}
Die Notwendigkeit ist klar. Angenommen, $f$ erfüllt die letztgenannte Bedingung. Seien $l\in\N_0$ mit $l<k$, $\al\leq l$, $q$ eine stetige Halbnorm auf $E$, $K\subseteq U$ eine kompakte Teilmenge, $\eps>0$ und
\[M>\max_{\bet\leq k}\sup_{x\in K}q(f_{\beta}(x)).\] Sei $\delta>0$ mit
$\delta<\min\{1,\frac{\eps}{2(k+1)^nM}\}$ derart gewählt, dass $\frac{q(R_{y}^{k-\al}\partial^{\alpha}f(x))}{\lVert x-y\rVert^{k-\vert\alpha\vert}}\leq\frac{\eps}{2}$ für alle $x,y\in K$ mit $0<\Vert x-y\Vert<\delta$ gilt. Aus
\[R_{y}^{l-\al}\p f(x)=R_{y}^{k-\al}\p f(x)+\sum_{l-\al<\bet\leq k-\al}\frac{(x-y)^{\beta}}{\beta!}f_{\alpha+\beta}(y)\]
folgt
\begin{align*}
\frac{q(R_{y}^{l-\al}\p f(x))}{\Vert x-y\Vert^{l-\al}}
&\leq \frac{q(R_{y}^{k-\al}\p f(x))}{\Vert x-y\Vert^{k-\al}}
+\underbrace{\sum_{l-\al<\bet\leq k-\al}}_{\leq(k+1)^n\text{ Summanden}}\underbrace{\frac{\vert (x-y)^{\beta}\vert}{\Vert x-y\Vert^{l-\al}}}_{\leq\delta}\underbrace{q(f_{\alpha+\beta}(y))}_{\leq M}\\
&\leq\frac{\eps}{2}+(k+1)^n\delta M\leq\eps\end{align*}
für alle $x,y\in K$ mit $0<\Vert x-y\Vert<\delta$. Somit ist $f\in\EE^k(U,E)$.
\end{proof}

\begin{lemma}\label{g}
Ist $f=(f_{\alpha})_{\al\leq k}\in\EE^k(U,E)$, dann sind die Funktionen $f_{\alpha}\colon U\to E$ für alle $\al\leq k$ stetig.
\end{lemma}

\begin{proof}
Wir zeigen, dass $f_{\alpha}$ in jedem Punkt $y\in U$ stetig ist. Nach Definition gilt $R_y^0\p f(x)\to 0$ für $x\to y$, also folgt $f_{\alpha}(x)=f_{\alpha}(y)+R_y^0\p f(x)\to f_{\alpha}(y)$ für $x\to y$.
\end{proof}

\begin{lemma}\label{i}
Sei $f=(f_{\alpha})_{\vert\alpha\vert\leq k}\in\EE^k(U,E)$.
\begin{itemize}
\item[(a)] Sind $\al<k$, $j\in\{1,\ldots,n\}$ und $x\in U$ derart, dass es eine Nullfolge $(t_n)_{n\in\N}$ in $\R\setminus\{0\}$ mit $x+t_n e_j\in U$ für alle $n\in\N$ gibt, dann gilt
\[\lim_{t\to 0}\frac{f_{\alpha}(x+t e_j)-f_{\alpha}(x)}{t}=f_{\alpha+e_j}(x).\]
\item[(b)] Es ist $f_0\vert_{U^{\circ}}$ eine $C^k$-Funktion mit $\partial^{\alpha}(f_0\vert_{U^{\circ}})=f_{\alpha}\vert_{U^{\circ}}$ für alle $\al\leq k$.
\item[(c)] Wenn $U$ eine lokalkonvexe Teilmenge mit dichtem Inneren ist, dann ist $f_0$ eine $C^k$-Funktion mit $\partial^{\alpha} f_0=f_{\alpha}$ für alle $\al\leq k$.
\end{itemize}
\end{lemma}

\begin{proof}
\begin{itemize}
\item[(a)] Wir wählen ein $l\in\N_0$ mit $\al<l\leq k$ und schreiben $x_t\coloneqq x+t e_j$ für $t\in\R$. Aus \[f_{\alpha}(x_t)-f_{\alpha}(x)=(T_{x}^{l-\al}\p f(x_t)-f_{\alpha}(x))+R_{x}^{l-\al}\p f(x_t)\]
für alle $t\in\R$ mit $x_t\in U$ folgt

\begin{align*}\frac{1}{t}(f_{\alpha}(x_t)-f_{\alpha}(x))
&=\frac{1}{t}\bigg(\sum_{0<\vert\beta\vert\leq l-\vert\alpha\vert}\frac{(x_t-x)^{\beta}}{\beta!}f_{\alpha+\beta}(x)+R_{x}^{l-\al}\p f(x_t)\bigg)\\
&=\sum_{0<m\leq l-\vert\alpha\vert}\frac{t^{m-1}}{m!}f_{\alpha+m e_j}(x)+\frac{1}{t}R_{x}^{l-\al}\p f(x_t)\\
&\to f_{\alpha+e_j}(x)+0=f_{\alpha+e_j}(x)\ \text{für }t\to 0.\end{align*}
\item[(b)] Nach Lemma \ref{g} sind die Funktionen $f_{\alpha}\colon U\to E$ für alle $\vert\alpha\vert\leq k$ stetig. Für $\alpha=0$ gilt $\partial^{\alpha} (f_0\vert_{U^{\circ}})=f_{\alpha}\vert_{U^{\circ}}$. Angenommen, dies ist 
für ein beliebiges $\al<k$ der Fall. Mit (a) gilt
\[f_{\alpha+e_j}\vert_{U^{\circ}}=\partial_j(f_{\alpha}\vert_{U^{\circ}})=\partial_j\partial^{\alpha}(f_{0}\vert_{U^{\circ}})=\partial^{\alpha+e_j}(f_{0}\vert_{U^{\circ}})\]
für alle $j\in\{1,\ldots,n\}$. Per Induktion folgt die Behauptung.
\item[(c)] Folgt aus (b).
\end{itemize}
\end{proof}

\begin{definition}[{vgl. \cite[S. 2]{Ti}}]
Für alle $l\in\N_0$ mit $l\leq k$, stetigen Halbnormen $q$ auf $E$ und kompakten Teilmengen $K\subseteq U$ seien die Halbnormen $\lVert \cdot\rVert_{\mathcal{E}^l,q,K}'$, ${\lVert \cdot\rVert_{\mathcal{E}^l,q,K}''}$ und ${\lVert \cdot\rVert_{\mathcal{E}^l,q,K}}\colon\EE^k(U,E)\to [0,\infty[$ durch
\[\lVert f\rVert_{\mathcal{E}^l,q,K}'\coloneqq\max_{\al\leq l}\sup_{x\in K}q( f_{\alpha}(x)),\]
\[\lVert f\rVert_{\mathcal{E}^l,q,K}''\coloneqq\max_{\al\leq l}\sup\left\{\frac{q(R_{y}^{l-\al}\p f(x))}{\Vert x-y\Vert^{l-\al}}\colon x\neq y\ \text{in }K\right\},\]
bzw.
\[\lVert f\rVert_{\mathcal{E}^l,q,K}\coloneqq \Vert f\Vert_{\mathcal{E}^l,q,K}'+\Vert f\Vert_{\mathcal{E}^l,q,K}''\]
definiert.\ftnote{Man sieht wie folgt, dass $\Vert f\Vert_{\mathcal{E}^l,q,K}''<\infty$ für $f\in\EE^k(U,E)$ gilt: Für jedes $\al\leq l$ gibt es ein $\delta_{\alpha}\in ]0,1[$ derart, dass $\frac{q(R_y^{l-\al}\p f(x))}{\Vert x-y\Vert^{l-\al}}\leq 1$ für alle $x\neq y$ in $K$ mit $\Vert x-y\Vert<\delta_{\alpha}$. Weil $K\times K\to E,\ (x,y)\mapsto R_y^{l-\al}\p f(x)$ stetig ist, gibt es ein $M_{\alpha}\geq 0$ mit $q(R_y^{l-\al}\partial^{\alpha}f(x))\leq M_{\alpha}$ für alle $(x,y)\in K\times K$. Dann gilt $\frac{q(R_y^{l-\al}\p f(x))}{\Vert x-y\Vert^{l-\al}}\leq\frac{M_{\alpha}}{\delta_{\alpha}^{l}}$ für alle $x\neq y$ in $K$ mit $\Vert x-y\Vert\geq\delta_{\alpha}$. Mit $M\coloneqq\max\{M_{\alpha}\colon\al\leq l\}$ und $\delta\coloneqq\min\{\delta_{\alpha}\colon\al\leq l\}$ folgt $\Vert f\Vert_{\EE^l,q,K}''\leq\max\left\{{1,\frac{M}{\delta^l}}\right\}<\infty$.} 
Wir versehen $\EE^k(U,E)$ mit der Hausdorffschen lokalkonvexen Topologie, welche durch die Halbnormen $m\lVert \cdot\rVert_{\mathcal{E}^l,q,K}$ mit $m\in\N$ und $l,\ q,\ K$ wie oben erzeugt wird.
\end{definition}

Dass jene eine gerichtete Menge von Halbnormen bilden, folgt aus dem nächsten Lemma.

\begin{lemma}\label{mm}
Seien $q$ eine stetige Halbnorm auf $E$ und $K\subseteq U$ kompakt. Für alle $j,l\in\N_0$ mit $j\leq l\leq k$ gibt es ein $m\in\N$ derart, dass $\lVert \cdot\rVert_{\mathcal{E}^j,q,K}\leq m\lVert \cdot\rVert_{\mathcal{E}^l,q,K}$.
\end{lemma}

\begin{proof}
Wir können $j<l$ annehmen, sonst ist die Aussage klar. Seien $M\coloneqq\max\{1,\text{diam}(K)\}$ und $m\in\N$ derart, dass $m\geq 1+(1+(l+1)^n)M^{l-j}$. Weiter sei $f=(f_{\alpha})_{\al\leq k}\in\EE^k(U,E)$. Für alle $\al\leq j$ und $x\neq y$ in $K$ gilt
\begin{align*}
\frac{q(R_{y}^{j-\al}\p f(x))}{\Vert x-y\Vert^{j-\al}}
&\leq M^{l-j}\frac{q(R_{y}^{l-\al}\p f(x))}{\Vert x-y\Vert^{l-\al}}
+\sum_{j-\al<\bet\leq l-\al}\underbrace{\frac{\vert (x-y)^{\beta}\vert}{\Vert x-y\Vert^{j-\al}}}_{\leq M^{l-j}}q(f_{\alpha+\beta}(y))\\
&\leq M^{l-j}\Vert f\Vert''_{\EE^l,q,K}+(l+1)^n M^{l-j}\Vert f\Vert'_{\EE^l,q,K}\\
&\leq (1+(l+1)^n)M^{l-j}\Vert f\Vert_{\EE^l,q,K}.
\end{align*}
Bilden wir das Supremum über alle $\al\leq j$ und $x\neq y$ in $K$, folgt
\[\Vert f\Vert''_{\EE^j,q,K}\leq (1+(l+1)^n)M^{l-j}\Vert f\Vert_{\EE^l,q,K}\]
und somit
\begin{align*}\Vert f\Vert_{\EE^j,q,K}
&=\Vert f\Vert'_{\EE^j,q,K}+\Vert f\Vert''_{\EE^j,q,K}
\leq \Vert f\Vert_{\EE^l,q,K}+(1+(l+1)^n)M^{l-j}\Vert f\Vert_{\EE^l,q,K}\\
&\leq m\Vert f\Vert_{\EE^l,q,K}.
\end{align*}
\end{proof}

Aus Lemma \ref{mm} folgt zudem, dass im Falle $k<\infty$ die Topologie auf $\EE^k(U,E)$ durch die Familie von Halbnormen $\big(\lVert\cdot\rVert_{\EE^k,q,K}\big)_{q,K}$ erzeugt wird.

\begin{lemma}\label{ff}
\begin{itemize}
\item[(a)] Für alle $l\leq k$ ist die lineare Abbildung
\[\text{\ee{pr}}_l\colon\EE^k(U,E)\to\EE^l(U,E),\quad f\mapsto \text{\ee{pr}}_l(f)\]
stetig.
\item[(b)] Für jede Teilmenge $V\subseteq U$ ist die lineare Abbildung
\[\EE^k(U,E)\to\EE^k(V,E),\quad f\mapsto f\vert_V\]
stetig.
\end{itemize}
\end{lemma}

\begin{proof}
\begin{itemize}
\item[(a)] Für alle $j\in\N_0$ mit $j\leq l$, stetigen Halbnormen $q$ auf $E$ und kompakten Teilmengen $K\subseteq U$ gilt
\[\Vert \text{pr}_l(f)\Vert_{\EE^j,q,K}=\Vert f\Vert_{\EE^j,q,K}\]
für alle $f\in\EE^k(U,E)$.
\item[(b)] Für alle $j\in\N_0$ mit $j\leq k$, stetigen Halbnormen $q$ auf $E$ und kompakten Teilmengen $K\subseteq V$ gilt \[\Vert f\vert_V\Vert_{\EE^j,q,K}=\Vert f\Vert_{\EE^j,q,K}\] für alle $f\in\EE^k(U,E)$.
\end{itemize}
\end{proof}

\subsection*{Whitneysche $k$-Jets auf lokalkonvexen Mengen mit dichtem Inneren}
\addcontentsline{toc}{subsection}{\protect\numberline{}Whitneysche $k$-Jets auf lokalkonvexen Mengen mit dichtem Inneren}

Die Whitneyschen $k$-Jets auf einer lokalkonvexen Teilmenge $U\subseteq\R^n$ mit dichtem Inneren entsprechen den $C^k$"--Funk"-tio"-nen.

\begin{satz}\label{gg}
Ist $U\subseteq\R^n$ eine lokalkonvexe Teilmenge mit dichtem Inneren, so ist die Abbildung
\begin{align*}\rho_U^k\colon C^k(U,E)\to\EE^k(U,E),\quad f\mapsto(\partial^{\alpha} f)_{\lvert\alpha\rvert\leq k}\end{align*}
ein Isomorphismus von topologischen Vektorräumen mit inverser Abbildung
\[\Phi_U^k\colon\EE^k(U,E)\to C^k(U,E),\quad (f_{\alpha})_{\lvert\alpha\rvert\leq k}\mapsto f_0.\]
Insbesondere wird dann die Topologie auf $\EE^k(U,E)$ durch die Familie von Halbnormen $\big(\lVert\cdot\rVert'_{\EE^l,q,K}\big)_{l,q,K}$ erzeugt.
\end{satz}

\begin{proof}
Nach Satz \ref{dd} (b) gilt $(\p f)_{\al\leq k}\in \EE^k(U,E)$ für alle $f\in C^k(U,E)$, also ist $\rho_U^k$ sinnvoll definiert. Mit Lemma \ref{i} (c) sehen wir, dass $\Phi_U^k$ tatsächlich Werte in $C^k(U,E)$ hat und rechtsinvers zu $\rho_U^k$ ist. 
Gewiss ist auch $\rho_U^k$ rechtsinvers zu $\Phi_U^k$ und linear, also ein Isomorphismus von Vektorräumen. 
Für alle $l\in\N_0$ mit $l\leq k$, stetigen Halbnormen $q$ auf $E$ und kompakten Teilmengen $K\subseteq U$ gilt
\[\lVert\cdot\rVert_{C^l,q,K}\circ\Phi_U^k=\lVert \cdot\rVert'_{\EE^l,q,K}\leq \lVert \cdot\rVert_{\EE^l,q,K}.\]
Daraus folgt, dass $\Phi_U^k$ stetig ist und im Falle, dass auch $\rho_U^k$ stetig ist, die Topologie auf $\EE^k(U,E)$ durch die Familie von Halbnormen $\big(\lVert\cdot\rVert'_{\EE^l,q,K}\big)_{l,q,K}$ erzeugt wird.
Um die Stetigkeit von $\rho_U^k$ nachzuweisen, seien $l$, $q$, $K$ wie zuvor.
Da $K$ kompakt und $U$ lokalkonvex ist, existieren $x_1,\dots,x_m\in K$ und $r_1,\ldots,r_m>0$ mit $K\subseteq \bigcup_{i=1}^mB_{r_i}Û(x_i)$ derart, dass $\overline{B}^U_{2r_i}(x_i)$ für $i\in\{1,\ldots,m\}$ konvex ist. Dann ist
\[A_i\coloneqq \big\{x+t(y-x)\colon t\in [0,1],\ x,y\in \overline{B}^U_{2r_i}(x_i)\cap K\big\}\]
eine kompakte Teilmenge von $\overline{B}^U_{2r_i}(x_i)$.
Wir setzen $r\coloneqq\min\{1, r_1,\ldots,r_m\}$,
$M\coloneqq\max\{1,\text{diam}(K)\}$ und $A\coloneqq A_1\cup\ldots\cup A_m$. Seien nun $f\in C^k(U,E)$, $\al\leq l$ und $x\neq y$ in $K$. Um $\Vert \rho_U^k(f)\Vert''_{\EE^l,q,K}$ abschätzen zu können, führen wir eine Fallunterscheidung durch. Im Falle $\Vert x-y\Vert>r$ folgt aus
\[R_y^{l-\al}\p f(x)=\partial^{\alpha}f(x)-\sum_{\bet\leq l-\al}\frac{(x-y)^{\beta}}{\beta!}\partial^{\alpha+\beta}f(y)\]
die Abschätzung
\begin{align*}\frac{q(R_y^{l-\al}\p f(x))}{\Vert x-y\Vert^{l-\al}}
&\leq\frac{1}{r^l}\bigg( q(\p f(x))+\sum_{\bet\leq l-\al}\vert(x-y)^{\beta}\vert q(\partial^{\alpha+\beta}f(y))\bigg)\\
&\leq\frac{1}{r^l}\big(\Vert f\Vert_{C^l,q,K}+(l+1)^nM^l\Vert f\Vert_{C^l,q,K}\big)=:Q_1(f).
\end{align*}
Angenommen, es sind $\Vert x-y\Vert\leq r$ und $\al\neq l$. Sei $i\in\{1,\ldots,m\}$ derart, dass $x\in B_{r_i}^U(x_i)$. Dann liegt die Verbindungsstrecke zwischen $x$ und $y$ in $A_i$. Mit Satz \ref{dd} (a) gilt
\begin{align*}
&R_y^{l-\al}\p f(x)\\
&=(l-\al)\int_0^1\bigg(\sum_{\bet=l-\al}\frac{(x-y)^{\beta}}{\beta!}(1-t)^{l-\al-1}(\partial^{\alpha+\beta} f(y+t(x-y))-\partial^{\alpha+\beta}f(y))\bigg)dt.\end{align*}
Daraus folgt
\begin{align*}
&\frac{q(R_y^{l-\al}\p f(x))}{\Vert x-y\Vert^{l-\al}}\\
&\leq l\int_0^1\bigg(\sum_{\bet=l-\al}\frac{\vert(x-y)^{\beta}\vert}{\Vert x-y\Vert^{l-\al}}(q(\partial^{\alpha+\beta} f(y+t(x-y)))+q(\partial^{\alpha+\beta}f(y)))\bigg)dt\\
&\leq 2l(l+1)^n\Vert f\Vert_{C^l,q,A}=:Q_2(f).
\end{align*}
Falls $\al=l$ ist, gilt
\[q(R_y^0\p f(x))\leq q(\p f(x))+q(\p f(y))\leq 2\Vert f\Vert_{C^l,q,K}=:Q_3(f).\]
In allen drei Fällen ist
\[\frac{q(R_y^{l-\al}\p f(x))}{\Vert x-y\Vert^{l-\al}}\leq Q_1(f)+Q_2(f)+Q_3(f)=:Q(f).\]
Daraus folgt $\Vert \rho_U^k(f)\Vert''_{\EE^l,q,K}\leq Q(f)$ und somit
\[\Vert \rho_U^k(f)\Vert_{\EE^l,q,K}=\Vert \rho_U^k(f)\Vert'_{\EE^l,q,K}+\Vert \rho_U^k(f)\Vert''_{\EE^l,q,K}\leq\Vert f\Vert_{C^l,q,K}+Q(f),\]
wobei die rechte Seite eine stetige Halbnorm in $f$ ist. Demnach ist $\rho_U^k$ stetig.
\end{proof}

\subsection*{Verkleben von Whitneyschen $k$-Jets}
\addcontentsline{toc}{subsection}{\protect\numberline{}Verkleben von Whitneyschen $k$-Jets}

Sind $U\subseteq\R^n$ eine Teilmenge und $(U_j)_{j\in J}$ eine Überdeckung von $U$ mit relativ offenen Teilmengen $U_j\subseteq U$, so werden wir sehen, dass für jede Familie $(f_j)_{j\in J}\in\prod_{j\in J}\EE^k(U_j,E)$ mit $f_i\vert_{U_i\cap U_j}=f_j\vert_{U_i\cap U_j}$ für alle $i,j\in J$ ein $f\in\EE^k(U,E)$ derart existiert, dass $f\vert_{U_j}=f_j$ für alle $j\in J$.

\begin{lemma}\label{ss}
Es seien $K\subseteq\R^n$ eine kompakte Teilmenge und $U_1,\ldots,U_m\subseteq K$ relativ offene Teilmengen mit $K=U_1\cup\ldots\cup U_m$. 
\begin{itemize}
\item[(a)] Ist $f\in\J^k(K,E)$ ein $k$-Jet mit $f\vert_{U_j}\in\EE^k(U_j,E)$ für alle $j\in \{1,\ldots,m\}$, so gilt $f\in\EE^k(K,E)$.
\item[(b)] Sind $l\in\N_0$ mit $l\leq k$ und $q$ eine stetige Halbnorm auf $E$,
dann gibt kompakte Teilmengen $K_j\subseteq U_j$ für $j\in \{1,\ldots,m\}$ und ein $C>0$, so dass
\[\Vert f\Vert_{\EE^l,q,K}\leq C\max_{1\leq j\leq m}\Vert f\vert_{U_j}\Vert_{\EE^l,q,K_j}\]
für alle $f\in\EE^k(K,E)$.
\end{itemize}
\end{lemma}

\begin{proof}
Wir finden $x_1,\ldots,x_t\in K$ und $r_1,\ldots,r_t>0$ mit $K\subseteq \bigcup_{i=1}^tB_{r_i}(x_i)$ derart, dass es für jedes $i\in\{1,\ldots,t\}$ ein $j(i)\in\{1,\ldots,m\}$ gibt mit
\[L_i\coloneqq\overline{B}_{2 r_i}(x_i)\cap K\subseteq U_{j(i)}.\]
Dann ist
\[K_j\coloneqq\bigcup_{\substack{1\leq i\leq t\\ j(i)=j}} L_i\subseteq U_j\]
eine kompakte Teilmenge für alle $j\in\{1,\ldots,m\}$. Wir setzen $r\coloneqq\min\{1, r_1,\ldots, r_t\}$, $M\coloneqq\max\{1,\text{diam}(K)\}$.
\begin{itemize}
\item[(a)] Sei $f\in\J^k(K,E)$ derart, dass $f\vert_{U_j}\in\EE^k(U_j,E)$ für alle $j\in \{1,\ldots,m\}$. Sind $l\in\N_0$ mit $l\leq k$, $\al\leq l$, $q$ eine stetige Halbnorm auf $E$ und $\eps>0$, dann gibt es ein $\delta\in ]0,r[$, so dass
\[\frac{q(R_y^{l-\al}\p f(x))}{\Vert x-y\Vert^{l-\al}}\leq\eps\]
für alle $j\in\{1,\dots,m\}$ und $x,y\in K_j$ mit $0<\Vert x-y\Vert<\delta.$ Da aber für beliebige $x,y\in K$ mit $0<\Vert x-y\Vert<\delta$ ein $i\in\{1,\ldots,t\}$ mit $x,y\in L_i\subseteq K_{j(i)}$ existiert, folgt $f\in\EE^k(K,E)$. 
\item[(b)] Seien $l, q$ wie zuvor und $C\coloneqq 2+\frac{1}{r^l}(1+(l+1)^n M^l)$. Für jedes $f\in\EE^k(K,E)$ gilt
\[\Vert f\Vert'_{\EE^l,q,K}=\max_{1\leq j\leq m}\Vert f\vert_{U_j}\Vert'_{\EE^l,q,K_j}\leq \max_{1\leq j\leq m}\Vert f\vert_{U_j}\Vert_{\EE^l,q,K_j}.\]
Seien $\al\leq l$ und $x\neq y$ in $K$. Im Falle $\Vert x-y\Vert\leq r$ gibt es ein $i\in\{1,\ldots,m\}$ mit $x,y\in K_i$, also gilt
\[\frac{q(R_y^{l-\al}\p f(x))}{\Vert x-y\Vert^{l-\al}}\leq\Vert f\vert_{U_i}\Vert_{\EE^l,q,K_i}\leq \max_{1\leq j\leq m}\Vert f\vert_{U_j}\Vert_{\EE^l,q,K_j}.\]
Für $\Vert x-y\Vert> r$ gilt
\begin{align*}\frac{q(R_y^{l-\al}\p f(x))}{\Vert x-y\Vert^{l-\al}}
&\leq\frac{1}{r^l}\bigg( q(\p f(x))+\sum_{\bet\leq l-\al}\vert(x-y)^{\beta}\vert q(\partial^{\alpha+\beta}f(y))\bigg)\\
&\leq\frac{1}{r^l}\big(\Vert f\Vert'_{\EE^l,q,K}+(l+1)^nM^l\Vert f\Vert'_{\EE^l,q,K}\big)\\
&\leq\frac{1}{r^l}(1+(l+1)^n M^l)\max_{1\leq j\leq m}\Vert f\vert_{U_j}\Vert_{\EE^l,q,K_j}.
\end{align*}
Mit Übergang zum Supremum folgt
\[\Vert f\Vert''_{\EE^l,q,K}\leq\Big(1+\frac{1}{r^l}(1+(l+1)^n M^l)\Big)\max_{1\leq j\leq m}\Vert f\vert_{U_j}\Vert_{\EE^l,q,K_j}\]
und somit
\[\Vert f\Vert_{\EE^l,q,K}\leq C\max_{1\leq j\leq m}\Vert f\vert_{U_j}\Vert_{\EE^l,q,K_j}.\]
\end{itemize}
\end{proof}

\begin{satz}\label{tt}
Es seien $U\subseteq\R^n$ eine Teilmenge, $(U_j)_{j\in J}$ eine Überdeckung von $U$ mit relativ offenen Teilmengen $U_j\subseteq U$ und \[R\coloneqq\bigg\{(f_j)_{j\in J}\in\prod_{j\in J}\EE^k(U_j,E)\colon\; f_i\vert_{U_i\cap U_j}=f_j\vert_{U_i\cap U_j}\ \text{\ee{für alle }} i,j\in J\bigg\}.\]
Dann ist die Abbildung
\begin{align}\label{gl.pp}\EE^k(U,E)\to R,\quad f\mapsto (f\vert_{U_j})_{j\in J}\end{align}
ein Isomorphismus von topologischen Vektorräumen.
\end{satz}

\begin{proof}
Man sieht leicht, dass $\eqref{gl.pp}$ stetig und linear ist. Für ein gegebenes $(f_j)_{j\in J}\in R$ sei $f\in\J^k(U,E)$ der durch $f\vert_{U_j}=f_j$ für $j\in J$ definierte $k$-Jet. Sei $K\subseteq U$ eine kompakte Teilmenge. Dann existieren $j(1),\ldots, j(m)\in J$ mit $K=\bigcup_{i=1}^{m}(U_{j(i)}\cap K).$
Nach Lemma \ref{ss} (a) gilt $f\vert_K\in\EE^k(K,E)$. Da $K$ beliebig war, folgt $f\in\EE^k(U,E)$.
Für alle $l\in\N_0$ mit $l\leq k$ und stetigen Halbnormen $q$ auf $E$ gibt es nach Lemma \ref{ss} (b), unabhängig von $f$, kompakte Teilmengen $K_i\subseteq U_{j(i)}\cap K$ für $i\in\{1,\ldots, m\}$ und ein $C>0$ derart, dass
\[\Vert f\Vert_{\EE^l,q,K}\leq C\max_{1\leq i\leq m}\Vert f\vert_{U_{j(i)}\cap K}\Vert_{\EE^l,q,K_i}=C\max_{1\leq i\leq m}\Vert f_{j(i)}\Vert_{\EE^l,q,K_{i}};\]
dabei ist die rechte Seite eine stetige Halbnorm in $(f_j)_{j\in J}$. Wir sehen also, dass die zu \eqref{gl.pp} inverse Abbildung
\[R\to\EE^k(U,E),\quad (f_j)_{j\in J}\mapsto f\ \text{mit }f\vert_{U_j}=f_j\ \text{für }j\in J\]
stetig ist.
\end{proof}

\section{Der Whitneysche Fortsetzungssatz für vektorwertige Funktionen}

Wir verfolgen das Ziel, den Whitneyschen Fortsetzungssatz für vektorwertige Funktionen zu beweisen.

\begin{satz}[\textbf{Whitneyscher Fortsetzungssatz}]\label{k}
Es seien $k\in\N_0$, $U\subseteq\R^n$ eine lokalkonvexe Teilmenge mit dichtem Inneren und $A\subseteq U$ eine relativ abgeschlossene, lokalkompakte Teilmenge. Dann besitzt die stetige lineare Abbildung
\[\rho_{A,U}^k\colon C^k(U,E)\to \EE^k(A,E),\quad f\mapsto ((\p f)\vert_A)_{\al\leq k}\]
eine stetige lineare Rechtsinverse
\[\Phi_{U,A}^k\colon \EE^k(A,E)\to C^k(U,E).\]
Letztere nennt man einen \ee{stetigen linearen Fortsetzungsoperator}.
\end{satz}

\begin{korollar}[\textbf{Fortsetzung von $C^k$-Funktionen}]\label{hh}
Es seien $k\in\N_0$, $U\subseteq\R^n$ eine lokalkonvexe Teilmenge mit dichtem Inneren und $A\subseteq U$ eine relativ abgeschlossene, lokalkompakte, lokalkonvexe Teilmenge mit dichtem Inneren. Dann besitzt die Einschränkung
\begin{align}\label{gl.kk}C^k(U,E)\to C^k(A,E),\quad f\mapsto f\vert_A\end{align}
eine stetige lineare Rechtsinverse.
\end{korollar}

\begin{proof}[Beweis von Korollar \ref{hh}]
Betrachten wir das kommutative Diagramm
\[\begin{xy}\xymatrix{
C^k(U,E)\ar[rd]_{\eqref{gl.kk}}\ar[r]^{\rho_{A,U}^k}& \EE^k(A,E)\ar[d]^{\Phi_A^k}_{\cong}\\ & C^k(A,E)  
}\end{xy},\]
wobei $\Phi_A^k$ nach Satz \ref{gg} ein Isomorphismus von topologischen Vektorräumen ist.
Da $\rho_{A,U}^k$ eine stetige lineare Rechtsinverse besitzt, hat auch \eqref{gl.kk} eine solche.
\end{proof}

\subsection*{Die Whitneysche Partition der Eins}
\addcontentsline{toc}{subsection}{\protect\numberline{}Die Whitneysche Partition der Eins}

Der stetige lineare Fortsetzungsoperator wird mithilfe einer Whitneyschen Partition der Eins konstruiert. Dazu folgen wir der Argumentation von Whitney \cite[\S\ 7–10]{Wh} und Bierstone \cite[Lemma 2.5]{Bi}.

Sei $A\subseteq\R^n$ eine abgeschlossene, nichtleere Teilmenge. Für alle $j\in\N_0$ sei $L_j$ die Menge der abgeschlossenen Würfel der Form
\[\left[\frac{z_1}{2^j},\frac{z_1+1}{2^j}\right]\times\ldots\times\left[\frac{z_n}{2^j},\frac{z_n+1}{2^j}\right]\subseteq\R^n\ \text{mit }z_1,\ldots,z_n\in\mathbb{Z}.\]
Sei $K_0$ die Menge der Würfel $C\in L_0$ mit $d(C,A)\geq 4\sqrt{n}$. Für $j\in\N$ definiert man rekursiv $K_j$ als die Menge derjenigen Würfel $C\in L_j$ mit \[d(C,A)\geq\frac{4\sqrt{n}}{2^j},\] die in keinem der Würfel aus $K_0\cup\ldots\cup K_{j-1}$ enthalten sind. Offenbar ist $\R^n\setminus A$ die Vereinigung aller Würfel aus $W\coloneqq\bigcup_{j\in\N_0}K_j$, und es gilt
\[\text{diam}(C)=\frac{\sqrt{n}}{2^j}\]
für alle $j\in\N_0$ und $C\in K_j$. Für jedes $C\in W$ seien $y_C$ der Mittelpunkt von $C$, $l_C$ die Seitenlänge von $C$, $x_C\in A$ ein fest gewählter Punkt mit $d(y_C,A)=\lVert y_C-x_C\rVert$ und
\[D_C\coloneqq\Big\{x\in\R^n \colon \lVert x-y_C\rVert_{\infty}<\frac{3}{4}l_C\Big\}.\]

\begin{lemma}\label{l}
\begin{itemize}
\item[(a)] Für alle $j\in\N$ und $C\in K_j$ gilt
\[d(C,A)<\frac{10\sqrt{n}}{2^{j}}.\]
\item[(b)] An jeden Würfel aus $W$ können höchstens Würfel aus $W$ mit halber, gleicher oder doppelter Seitenlänge grenzen.
\item[(c)] Für alle $C,C_*\in W$ gilt $D_C\cap D_{C_*}\neq\emptyset$ genau dann, wenn $C\cap C_*\neq\emptyset$.
\item[(d)] Es gibt ein $c\in\N$ derart, dass
\[\lvert\{C\in W\colon C\cap C_*\neq \emptyset\}\rvert\leq c\]
für alle $C_*\in W$.
\item[(e)] Die offene Überdeckung $(D_C)_{C\in W}$ von $\R^n\setminus A$ ist lokalendlich.
\item[(f)] Es seien $C,C_*\in W$ mit $C\cap C_*\neq\emptyset$ und $y\in C,y_*\in C_*$. Wir setzen $\delta\coloneqq d(y,A)$ und $\delta_*\coloneqq d(y_*,A)$ (oder $\delta\coloneqq\lVert y-x_0\rVert$ und $\delta_*\coloneqq\lVert y_*-x_0\rVert$ für einen gegebenen Punkt $x_0\in A$). Dann gilt
\[\delta<2\delta_*.\]
\item[(g)] Für alle $C\in W\setminus K_0$ und $y_*\in C$ gilt
\[d(y_*,A)<14 \sqrt{n}l_C.\]
\end{itemize}
\end{lemma}

\begin{proof}
\begin{itemize}
\item[(a)] Es liegt $C$ in einem Würfel $C'\in L_{j-1}\setminus K_{j-1}$, welcher in keinem der Würfel aus $K_0\cup\ldots\cup K_{j-2}$ enthalten ist. Da für diesen $d(C',A)<\frac{4\sqrt n}{2^{j-1}}$ gelten muss, folgt
\[d(C,A)\leq d(C',A)+\text{diam}(C')<\frac{4\sqrt n}{2^{j-1}}+\frac{\sqrt n}{2^{j-1}}=\frac{10\sqrt n}{2^{j}}.\]
\item[(b)] Seien $C\in K_i$ und $C'\in K_{i+j}$ mit $i,j\in\N_0$ und $j\geq 2$. Mit (a) folgt
\begin{align*}d(C,C')&\geq d(C,A)-d(C',A)-\text{diam}(C')> \frac{4\sqrt{n}}{2^i}-\frac{10\sqrt n}{2^{i+j}}-\frac{\sqrt{n}}{2^{i+j}}\\
&=\frac{(2^j\cdot 4-11)\sqrt{n}}{2^{i+j}}>0,\end{align*}
also grenzt $C$ nicht an $C'$.
\item[(c)] Wir können ohne Einschränkung annehmen, dass $l_{C_*}\leq l_C$. Es ist klar, dass $D_C\cap D_{C_*}\neq\emptyset$ aus $C\cap C_*\neq\emptyset$ folgt. Es gelte nun $C\cap C_*=\emptyset$ und somit $d_{\infty}(C,C_*)>0$. Dann gibt es nach Konstruktion von $W$ einen Würfel $C'\in W$ mit $l_{C'}\leq d_{\infty}(C,C_*)$, welcher an $C$ grenzt. Nach (b) gilt aber $\frac{1}{2}l_C\leq l_{C'}$, also
\begin{align*}
\frac{1}{2}l_C\leq d_{\infty}(C,C_*).\end{align*}
Angenommen, es existiert ein $x\in D_C\cap D_{C_*}$. Dann gilt $d_{\infty}(x,C)<\frac{1}{4}l_C$ und  $d_{\infty}(x,C_*)<\frac{1}{4}l_{C_*}\leq \frac{1}{4}l_C$, woraus
\[d_{\infty}(C,C_*)\leq d_{\infty}(x,C)+d_{\infty}(x,C_*)<\frac{1}{2}l_C\]
folgt, ein Widerspruch. 
Somit gilt $D_C\cap D_{C_*}=\emptyset$.

\item[(d)] Folgt aus (b).
\item[(e)] Folgt aus (c) und (d).
\item[(f)] Sei $j\in\N_0$ derart, dass $C_*\in K_j$. Da die Seitenlänge von $C$ höchstens doppelt so groß wie die von $C_*$ sein kann, gilt $\text{diam}(C)\leq\frac{\sqrt{n}}{2^{j-1}}$ und damit
\[\lVert y-y_*\rVert\leq \text{diam}(C)+\text{diam}(C_*)\leq\frac{\sqrt{n}}{2^{j-1}}+\frac{\sqrt{n}}{2^{j}}=\frac{3\sqrt{n}}{2^j}.\]
Zusammen mit $\delta_*\geq\frac{4\sqrt{n}}{2^j}$ folgt $\delta\leq\delta_*+\Vert y-y_*\Vert<2\delta_*$.
\item[(g)] Seien $\delta_*\coloneqq d(y_*,A)$ und $j\in\N$ derart, dass $C\in K_j$. Aus $\text{diam}(C)=\frac{\sqrt{n}}{2^j}\leq \frac{1}{4}\delta_*$ folgt
\[\delta_*\leq d(C,A)+\text{diam}(C)\leq d(C,A)+\frac{1}{4}\delta_*,\] also
\[\frac{3}{4}\delta_* \leq d(C,A)<\frac{10\sqrt{n}}{2^j}\]
und somit
\[\delta_*< \frac{40\sqrt{n}}{3\cdot 2^j}<\frac{14\sqrt{n}}{2^j}=14\sqrt{n}l_C.\]
\end{itemize}
\end{proof} 

\begin{lemma}[\textbf{Whitneysche Partition der Eins}]\label{j}
Auf $\R^n\setminus A$ existiert eine $C^{\infty}$-Partition der Eins $(\varphi_C)_{C\in W}$ mit den folgenden Eigenschaften:
\begin{itemize}
\item[(a)] Für alle $C\in W$ gilt $\text{\ee{supp}}(\varphi_C)\subseteq D_C$.
\item[(b)] Für alle $k\in\N_0$ gibt es ein $N_k>0$ derart, dass 
\[\vert\p\varphi_C(y_*)\vert<\frac{N_k}{d(y_*,A)^{\al}}\]
für alle $\al\leq k$, $C\in W$ und $y_*\in\R^n\setminus A$ mit $d(y_*,A)<4\sqrt{n}$.
\end{itemize}
\end{lemma}

\begin{proof}
Sei $\psi\colon\R^n\to\R$ eine glatte Funktion mit $\psi(\R^n)\subseteq [0,1]$, $\psi\vert_{\left[-\frac{1}{2},\frac{1}{2}\right]^n}=1$ und $\text{supp}(\psi)\subseteq \left]-\frac{3}{4},\frac{3}{4}\right[^n$.
Für jedes  $C\in W$ ist
\[\psi_C\colon\R^n\to\R,\quad \psi_C(x)\coloneqq\psi\left(\frac{x-y_C}{l_C}\right)\]
eine glatte Funktion mit $\psi_C(\R^n)\subseteq[0,1]$, $\psi_C\vert_C=1$ und $\text{supp}(\psi_C)\subseteq D_C.$
Da $(D_C)_{C\in W}$ eine lokalendliche offene Überdeckung von $\R^n\setminus A$ ist und ${\sum_{C\in W}\psi_{C}(x)>0}$ für alle $x\in\R^n\setminus A$, sehen wir, dass
\[\varphi_C\colon\R^n\setminus A\to\R,\quad \varphi_C(x)\coloneqq\frac{\psi_C(x)}{\sum_{C'\in W}\psi_{C'}(x)}\]
für jedes $C\in W$ sinnvoll definiert und glatt ist. Wir stellen fest, dass $(\varphi_C)_{C\in W}$ eine $C^{\infty}$-Partition der Eins ist, welche (a) erfüllt.

Zwei Würfel $P,Q\in W$ seien äquivalent, wenn die folgende Bedingung erfüllt ist: Sind $P_1,\ldots,P_l$ mit $l\in\N$ diejenigen Würfel in $W$, welche $C$ schneiden, dann sind \[y_{Q}+\frac{l_{Q}}{l_P}(P_1-y_P),\ldots, y_{Q}+\frac{l_{Q}}{l_P}(P_l-y_P)\]  genau diejenigen Würfel in $W$, die $Q$ schneiden. Dadurch wird eine Äquivalenzrelation auf $W$ mit einer endlichen Anzahl $m\in\N$ von Äquivalenzklassen definiert. Wir wählen dazugehörige Repräsentanten $C_1,\ldots, C_m\in W$.

Sei ein $k\in\N_0$ gegeben. Da die glatten Funktionen $\p\varphi_C$ mit $\al\leq k$ und $C\in W$ auf der kompakten Menge $C_1\cup\ldots \cup C_m$ beschränkt sind und nur endlich viele davon auf $C_1\cup\ldots \cup C_m$ nicht verschwinden, 
gibt es ein $N_k>0$ derart, dass
\begin{align}\label{gl.ff}\vert\p\varphi_C(x)\vert<\frac{N_k}{(14\sqrt{n})^k}\end{align}
für alle $\al\leq k$, $C\in W$ und $x\in C_1\cup\ldots\cup C_m$. 

Seien $\al\leq k$ und $y_*\in\R^n\setminus A$ mit $d(y_*,A)<4\sqrt{n}$ beliebig. Es gibt ein $P\in W\setminus K_0$ mit $y_*\in P$ und ein $Q\in\{C_1,\ldots,C_m\}$, welches zu $P$ äquivalent ist. Seien $P_1,\ldots, P_l$ diejenigen Würfel aus $W$, die $P$ schneiden, und $Q_i\coloneqq y_{Q}+\frac{l_{Q}}{l_{P}}(P_i-y_{P})$ für $i\in\{1,\ldots,l\}$.
Dann gilt
\[y_{Q_i}+\frac{l_{Q_i}}{l_{P_i}}(x-y_{P_i})=y_{Q}+\frac{l_{Q}}{l_{P}}(x-y_{P})\] 
für alle $x\in\R^n$, also
\[\psi_{P_i}(x)=\psi_{Q_i}\left(y_{Q}+\frac{l_{Q}}{l_{P}}(x-y_{P})\right).\]
Für alle $x\in P$ gilt
\begin{align*}
\sum_{C\in W}\psi_{C}(x)
=\sum_{j=1}^l\psi_{P_j}(x)
&=\sum_{j=1}^l\psi_{Q_j}\left(y_{Q}+\frac{l_{Q}}{l_{P}}(x-y_{P})\right)
=\sum_{C\in W}\psi_{C}\left(y_{Q}+\frac{l_{Q}}{l_{P}}(x-y_{P})\right)
\end{align*}
und somit
\[\varphi_{P_i}(x)=\varphi_{Q_i}\left(y_{Q}+\frac{l_{Q}}{l_{P}}(x-y_{P})\right),\]
so dass mit der Kettenregel
\[\p\varphi_{P_i}(x)=\left(\frac{l_{Q}}{l_{P}}\right)^{\al}\p\varphi_{Q_i}\left(y_{Q}+\frac{l_{Q}}{l_{P}}(x-y_{P})\right)\]
folgt. Mit \eqref{gl.ff} und Lemma \ref{l} (g) schließen wir
\begin{align*} \vert \p\varphi_{P_i}(y_*)\vert
&=\left(\frac{l_{Q}}{l_P}\right)^{\al}\left\vert\p\varphi_{Q_i}\left(y_{Q}+\frac{l_{Q}}{l_P}(y_*-y_{P})\right)\right\vert
<\left(\frac{1}{l_P}\right)^{\al}\frac{N_k}{(14\sqrt{n})^k}\\
&\leq\left(\frac{14\sqrt{n}}{d(y_*,A)}\right)^{\al}\frac{N_k}{(14\sqrt{n})^k}
\leq \frac{N_k}{d(y_*,A)^{\al}}.
\end{align*}
Zusammen mit $\p\varphi_{C}(y_*)=0$ für alle $C\in W\setminus\{P_1,\ldots, P_l\}$ folgt, dass (b) erfüllt ist.\end{proof}

\subsection*{Beweis des Whitneyschen Fortsetzungssatzes}
\addcontentsline{toc}{subsection}{\protect\numberline{}Beweis des Whitneyschen Fortsetzungssatzes}

Für den Beweis des Fortsetzungssatzes benötigen wir weitere Hilfsmittel.

\begin{lemma}[{vgl. \cite[Lemma 1]{Wh}}]\label{b}
Es seien $\gamma\colon I\to E$ eine stetige Kurve auf einem nichtentarteten Intervall $I\subseteq\R$, $A^*\subseteq\R$ eine abgeschlossene Teilmenge mit $A^*\subseteq I$, $t_0\in A^*$ und $v\in E$. Für alle stetigen Halbnormen $q$ auf $E$ und $\eps>0$ gebe es ein $\delta>0$, welches den beiden folgenden Bedingungen genügt:
\begin{itemize}
\item[(1)] Für alle $t\in A^*$ mit $0<\lvert t-t_0\rvert<\delta$ gilt \[q\left(\frac{\gamma(t)-\gamma(t_0)}{t-t_0}-v\right)<\eps.\]
\item[(2)] Für alle $t\in I\setminus A^*$ mit $\lvert t-t_0\rvert<\delta$ ist $\gamma$ in $t$ differenzierbar und es gilt \[q(\gamma'(t)-v)<\eps.\]
\end{itemize}
Dann  ist $\gamma$ in $t_0$ differenzierbar mit $\gamma'(t_0)=v$.
\end{lemma}

\begin{proof}
Zu einer gegebenen stetigen Halbnorm $q$ auf $E$ und einem $\eps>0$ sei $\delta>0$ wie oben. Wir werden zeigen, dass \[q\left(\frac{\gamma(t)-\gamma(t_0)}{t-t_0}-v\right)<2\eps\] für alle $t\in I\setminus A^*$ mit $\lvert t-t_0\rvert<\delta$ gilt. Zusammen mit (1) folgt dann die Behauptung. Angenommen, wir haben ein solches $t$ gegeben mit $t>t_0$ (der Fall $t<t_0$ ist analog). Wir setzen $s\coloneqq\text{max}([t_0,t]\cap A^*)$. Nach dem Hahn-Banachschen Fortsetzungssatz
gibt es ein stetiges lineares Funktional $\lambda\colon E\to\R$ derart, dass \[\lambda\left(\frac{\gamma(t)-\gamma(s)}{t-s}-v\right)=q\left(\frac{\gamma(t)-\gamma(s)}{t-s}-v\right)\] und $\vert\lambda(x)\vert\leq q(x)$ für alle $x\in E$. 
Wegen $]s,t]\subseteq I\setminus A^*$ ist $\gamma$ nach (2) auf $]s,t]$ differenzierbar und damit auch $\lambda\circ\gamma$, wobei $(\lambda\circ\gamma\vert_{]s,t]})'=\lambda\circ(\gamma\vert_{]s,t]})'$ gilt. Gemäß dem reellen Mittelwertsatz gibt es ein $\xi\in]s,t[$ mit
\[\frac{\lambda(\gamma(t))-\lambda(\gamma(s))}{t-s}=(\lambda\circ\gamma)'(\xi)=\lambda(\gamma'(\xi)).\]
Daraus folgt
\begin{align*} q\left(\frac{\gamma(t)-\gamma(s)}{t-s}-v\right)&=\left\vert\lambda\left(\frac{\gamma(t)-\gamma(s)}{t-s}-v\right)\right\vert=\left\vert \frac{\lambda(\gamma(t))-\lambda(\gamma(s))}{t-s}-\lambda(v)\right\vert\\ &=\vert\lambda(\gamma'(\xi))-\lambda(v)\vert=\vert\lambda(\gamma'(\xi)-v)\vert\leq q(\gamma'(\xi)-v)<\eps.
\end{align*}
Im Falle $s=t_0$ sind wir fertig. Andernfalls gilt (1) zufolge \[q\left(\frac{\gamma(s)-\gamma(t_0)}{s-t_0}-v\right)<\eps\]
und somit
\begin{align*}
&q\left(\frac{\gamma(t)-\gamma(t_0)}{t-t_0}-v\right)\\
&=q\left(\frac{t-s}{t-t_0}\frac{\gamma(t)-\gamma(s)}{t-s}+\frac{s-t_0}{t-t_0}\frac{\gamma(s)-\gamma(t_0)}{s-t_0}-\frac{t-s+s-t_0}{t-t_0}v\right)\\ 
&\leq \frac{t-s}{t-t_0}q\left(\frac{\gamma(t)-\gamma(s)}{t-s}-v\right)+\frac{s-t_0}{t-t_0}q\left(\frac{\gamma(s)-\gamma(t_0)}{s-t_0}-v\right)<2\eps.
\end{align*}
\end{proof}

\begin{lemma}\label{n}
Es seien $A\subseteq\R^n$ eine abgeschlossene, nichtleere Teilmenge und $W$ wie im vorigen Abschnitt definiert. Ist $K\subseteq\R^n$ kompakt, dann sind die Mengen
\[K_A\coloneqq\overline{\{x\in A\colon\exists\, y\in K\text{ \ee{mit} }\Vert y-x\Vert=d(y,A)\}}\]
und
\[K^A\coloneqq\overline{\left\{x_C\colon C\in W,\ \exists\, C'\in W,\ y\in C'\cap K\text{ \ee{mit} } C\cap C'\neq\emptyset\text{ \ee{und} }d(y,A)<1.\right\}}\]
kompakt.
\end{lemma}

\begin{proof}
Wir zeigen, dass die Mengen, deren Abschluss wir hier betrachten, beschränkt sind. Sei $M\coloneqq\sup\{\Vert x\Vert\colon x\in K\}$. Wenn es für ein $x\in A$ ein $y\in K$ mit $\Vert y-x\Vert=d(y,A)$ gibt, dann gilt
\[\Vert x\Vert\leq\Vert y\Vert+\Vert y-x\Vert\leq M+d(y,A)\leq M+d(K,A)+\text{diam}(K).\]
Angenommen, für ein $C\in W$ existieren $C'\in W$ und $y\in C'\cap K$ derart, dass $C\cap C'\neq\emptyset$ und $d(y,A)<1$. Mit zweimaligem Anwenden von Lemma \ref{l} (f) gilt $\Vert y-x_C\Vert< 2\Vert y_C-x_C\Vert<4d(y,A)<4$, also
\[\Vert x_C\Vert\leq\Vert y\Vert+\Vert y-x_C\Vert\leq M+4.\]
\end{proof}

\begin{lemma}\label{rr}
Jede lokalkonvexe Teilmenge $U\subseteq\R^n$ mit dichtem Inneren ist $C^{\infty}$-parakompakt, d.h. für jede Überdeckung $(U_j)_{j\in J}$ von $U$ mit relativ offenen Teilmengen $U_j\subseteq U$ existiert eine $C^{\infty}$-Partition der Eins $(h_j)_{j\in J}$ auf $U$ mit $\text{supp}(h_j)\subseteq U_j$ für $j\in J$.
\end{lemma}

\begin{proof}
Siehe Beispiel \ref{nn} (b).
\end{proof}

\begin{lemma}\label{ii}
Es seien $k\in\N_0\cup\{\infty\}$ und $U\subseteq\R^n$ eine lokalkonvexe Teilmenge mit dichtem Inneren.
\begin{itemize}
\item[(a)] Für jedes $g\in C^k(U,\R)$ ist die lineare Abbildung
\[\lambda_g\colon C^k(U,E)\to C^k(U,E),\quad f\mapsto g\cdot f\]
stetig.
\item[(b)] Sei $(W_j)_{j\in J}$ eine lokalendliche Familie von Teilmengen $W_j\subseteq U$. Wir schreiben $C^k_{W_j}(U,E)$ für den Untervektorraum aller $f\in C^k(U,E)$ mit $\text{\ee{supp}}(f)\subseteq W_j$. Dann ist die lineare Abbildung
\begin{align*}\prod_{j\in J}C^k_{W_j}(U,E)\to C^k(U,E),\quad (f_j)_{j\in J}\mapsto \sum_{j\in J}f_j\end{align*}
mit $\big(\sum_{j\in J}f_j\big)(x)\coloneqq\sum_{j\in J}f_j(x)$ für $x\in U$ stetig.
\end{itemize}
\end{lemma}

\begin{proof}
Wir verweisen auf Lemma \ref{y} (b) und (d); dort werden allgemeinere Aussagen bewiesen.
\end{proof}

Der Beweis des Whitneyschen Fortsetzungssatzes für vektorwertige Funktionen basiert auf Whitneys Beweis \cite[\S\ 11]{Wh} für den reellwertigen Fall.

\begin{proof}[Beweis von Satz \ref{k}]
Sei $k\in\N_0$. Zunächst zeigen wir die Aussage für den Spezialfall $U=\R^n$, wobei der Fall $A=\emptyset$ trivial ist.  
Sei nun $A\subseteq\R^n$ eine abgeschlossene, nichtleere Teilmenge, und sei $(\varphi_C)_{C\in W}$ eine $C^{\infty}$-Partition der Eins auf $\R^n\setminus A$ wie in Lemma \ref{j}.
Für jedes $f=(f_{\alpha})_{\al\leq k}\in\EE^k(A,E)$ definiert man die Funktion
\begin{align}\label{gl.h}\Phi_{\R^n,A}^k(f)\colon \R^n\to E,\quad x\mapsto\begin{cases} \sum_{C\in W}\varphi_C(x)T_{x_C}^k f(x) &\text{wenn }x\in \R^n\setminus A; \\ f_0(x) &\text{wenn }x\in A.\end{cases}\end{align}

\ee{Erster Schritt: Für jedes $f=(f_{\alpha})_{\al\leq k}\in\EE^k(A,E)$ ist $F\coloneqq\Phi_{\R^n,A}^k(f)$ eine $C^k$-Funktion mit $\rho_{A,\R^n}^k(F)=f$. Somit ist die Abbildung
\[\Phi_{\R^n,A}^k\colon\EE^k(A,E)\to C^k(\R^n,E),\quad f\mapsto \Phi_{\R^n,A}^k(f)\]
rechtsinvers zu $\rho_{A,\R^n}^k$.}

Für alle $\al\leq k$ ist zu zeigen, dass $\p F(x)$ für alle $x\in\R^n$ existiert und $\p F$ stetig ist mit $(\p F)\vert_A=f_{\alpha}$. Zunächst stellen wir fest, dass die Funktion $F\vert_{\R^n\setminus A}$ glatt ist, denn sowohl $\varphi_C$ als auch $T_{x_C}^k f$ sind für alle $C\in W$ glatt.

\begin{samepage}
\textit{Behauptung: Für alle $\al\leq k$, $x_0\in\partial A$, stetigen Halbnormen $q$ auf $E$ und $\eps>0$ gibt es ein $\delta>0$ derart, dass
\[q(\partial^{\alpha}F(y)-f_{\alpha}(x_0))<\eps\]
für alle $y\in\R^n\setminus A$ mit $\Vert y-x_0\Vert<\delta.$}
\end{samepage}

Wenn die Behauptung nachgewiesen ist, kann man den ersten Beweisschritt per Induktion nach $\al$ wie folgt beenden. Sei $\alpha=0$. Wegen der Stetigkeit von $F\vert_A=f_0$ gibt es für jedes $x_0\in\partial A$ und für alle stetigen Halbnormen $q$ auf $E$ und $\eps>0$ ein $\delta>0$ derart, dass
\[q(F(x)-f_0(x_0))<\eps\]
für alle $x\in A$ mit $\Vert x-x_0\Vert<\delta$. Mit der Behauptung können wir nach Verkleinern von $\delta$ annehmen, dass diese Abschätzung auch für alle $x\in\R^n\setminus A$ mit $\Vert x-x_0\Vert<\delta$ gilt. Somit ist $F$ stetig in $x_0$. Ebenso ist $F$ in allen $x\in\R^n\setminus\partial A$ stetig. Nehmen wir nun für ein $\al<k$ an, dass $\p F(x)$ für alle $x\in\R^n$ existiert und $\p F$ stetig ist mit $(\p F)\vert_A=f_{\alpha}$. Sei $j\in\{1,\ldots,n\}.$ Nach Lemma \ref{i} (b) ist $F\vert_{A^{\circ}}=f_0\vert_{A^{\circ}}$ eine $C^k$-Funktion mit $\partial^{\alpha+e_j}F(x)=f_{\alpha+e_j}(x)$ für alle $x\in A^{\circ}$. Sei $x_0=(x_{0,1},\ldots,x_{0,n})\in\partial A$. Wir betrachten die abgeschlossene Teilmenge
\[A^*\coloneqq\{t\in\R\colon (x_{0,1},\ldots, x_{0,j-1},t,x_{0,j+1},\ldots, x_{0,n})\in A\}\subseteq\R\]
sowie die stetige, auf $\R\setminus A^*$ differenzierbare Funktion
\[\gamma\colon\R\to E,\quad \gamma(t)\coloneqq \p F(x_{0,1},\ldots,t,\ldots, x_{0,n})\]
und setzen $t_0\coloneqq x_{0,j}\in A^*$ und $v\coloneqq f_{\alpha+e_j}(x_0)$.
Für alle stetigen Halbnormen $q$ auf $E$ und $\eps>0$ liefert die Behauptung ein $\delta>0$ derart, dass
\[q(\gamma'(t)-v)=q(\partial^{\alpha+e_j}F(x_{0,1},\ldots,t,\ldots, x_{0,n})-f_{\alpha+e_j}(x_0))<\eps\]
für alle $t\in\R\setminus A^*$ mit $\vert t-t_0\vert<\delta$. Mit Lemma \ref{i} (a) können wir nach Verkleinern von $\delta$ annehmen, dass
\[q\left(\frac{\gamma(t)-\gamma(t_0)}{t-t_0}-v\right)=q\left(\frac{f_{\alpha}(x_{0,1},\ldots,t,\ldots, x_{0,n})-f_{\alpha}(x_0)}{t-t_0}-f_{\alpha+e_j}(x_0)\right)<\eps\]
für alle $t\in A^*$ mit $0<\vert t-t_0\vert<\delta$.
Mit Lemma \ref{b} folgt, dass $\gamma'(t_0)=\partial^{\alpha+e_j}F(x_0)$ existiert und mit $v=f_{\alpha+e_j}(x_0)$ übereinstimmt. Somit ist gezeigt, dass $\partial^{\alpha+e_j}F(x)$ für alle $x\in\R^n$ existiert und $(\partial^{\alpha+e_j}F)\vert_A=f_{\alpha+e_j}$ gilt. Analog wie im Fall $\alpha=0$ folgert man aus der Behauptung und der Stetigkeit von $f_{\alpha+e_j}$, dass $\partial^{\alpha+e_j}F$ stetig ist, was den Induktionsbeweis beendet.

Zum Beweis der Behauptung seien $\al\leq k$, $x_0\in\partial A$, $q$ eine stetige Halbnorm auf $E$ und $\eps>0$. Zudem wählen wir ein
\[0<\eta<\min\left\{\frac{\eps}{c((k+2)!)^n4^{k+1}N_k},\ \frac{\eps}{4}\right\}\]
mit $c$ wie in Lemma \ref{l} (d) und $N_k$ wie in Lemma \ref{j} sowie ein
\[M>\max\{q(f_{\beta}(x))\colon \bet\leq k,\ x\in A\cap\overline{B}_1(x_0)\}.\]
Da $f$ ein Whitneyscher $k$-Jet ist, können wir ein
\[0<\delta<\min\left\{\frac{\eps}{8(k+1)^n M},\ \frac{1}{4}\right\}\]
so klein wählen, dass 
\begin{align}\label{gl.g}q(R_{x'}^{k-\bet}\partial^{\beta} f(x))\leq\Vert x-x'\Vert^{k-\bet}\eta\end{align}
für alle $\bet\leq k$ und $x,x'\in A\cap \overline{B}_1(x_0)$ mit $\Vert x-x'\Vert<4\delta$.
Sei $y_*\in\R^n\setminus A$ mit $\lVert y_*-x_0\rVert <\delta$ beliebig. Wir wählen ein $x_*\in A$ mit $\delta_*\coloneqq d(y_*,A)=\Vert y_*-x_*\Vert.$ Es gelten $\delta_*<\delta$ und $\lVert x_*-x_0\rVert<2\delta$.
Betrachten wir die Abschätzung 
\begin{align}\label{gl.d}
q(\p F(y_*)-f_{\alpha}(x_0))&\leq q(\p F(y_*)-T_{x_*}^{k-\al}\p f(y_*))+q(T_{x_*}^{k-\al}\p f(y_*)-f_{\alpha}(x_*))\notag\\
&+q(f_{\alpha}(x_*)-T_{x_0}^{k-\al}\p f(x_*))+q(T_{x_0}^{k-\al}\p f(x_*)-f_{\alpha}(x_0)).
\end{align}
Nach Definition des Taylorpolynoms ist
\[T_{x_*}^{k-\al}\p f(y_*)-f_{\alpha}(x_*)=\sum_{0<\vert\beta\vert\leq k-\al}\frac{(y_*-x_*)^{\beta}}{\beta!}f_{\alpha+\beta}(x_*).\]
Mit $\vert (y_*-x_*)^{\beta}\vert \leq\Vert y_*-x_*\Vert^{\bet}=\delta_*^{\bet}\leq \delta$ für alle $\bet>0$ folgt
\[q(T_{x_*}^{k-\al}\p f(y_*)-f_{\alpha}(x_*))\leq\sum_{0<\vert\beta\vert\leq k-\al}\vert(y_*-x_*)^{\beta}\vert q(f_{\alpha+\beta}(x_*))\leq(k+1)^n\delta M<\frac{\eps}{4},\]
und mit  $\vert (x_*-x_0)^{\beta}\vert\leq \Vert x_*-x_0\Vert^{\bet}\leq (2\delta)^{\bet}\leq 2\delta$ für alle $\bet>0$ 
erhält man analog
\[q(T_{x_0}^{k-\al}\p f(x_*)-f_{\alpha}(x_0))\leq \sum_{0<\vert\beta\vert\leq k-\al}\vert (x_*-x_0)^{\beta}\vert q(f_{\alpha+\beta}(x_0))
\leq(k+1)^n 2\delta M<\frac{\eps}{4}.\]
Wegen $f_{\alpha}(x_*)-T_{x_0}^{k-\al}\p f(x_*)=R_{x_0}^{k-\al}\p f(x_*)$ und \eqref{gl.g} gilt weiter
\[q(f_{\alpha}(x_*)-T_{x_0}^{k-\al}\p f(x_*))\leq\Vert x_*-x_0\Vert^{k-\al}\eta\leq\eta<\frac{\eps}{4}.\]
Seien $C\in W$ mit $y_*\in C$ und $C_1,\ldots,C_l$ diejenigen Würfel in $W$, die $C$ schneiden, wobei $l\leq c$ gilt. Wegen $\sum_{j=1}^l\varphi_{C_j}(x)=1$ gilt
\[F(x)=T_{x_*}^{k}f(x)+\sum_{j=1}^l\varphi_{C_j}(x)(T_{x_{C_j}}^{k}f(x)-T_{x_*}^{k}f(x))\]
für alle $x\in C$. Mit der Leibnizformel und Lemma \ref{c} (a) folgt
\begin{align}\label{gl.hh}\p F(y_*)&=T_{x_*}^{k-\al}\p f(y_*)\notag\\
&+\sum_{j=1}^l\sum_{\beta\leq\alpha}\binom{\alpha}{\beta}\partial^{\beta}\varphi_{C_j}(y_*)(T_{x_{C_j}}^{k-\al+\bet}\partial^{\alpha-\beta}f(y_*)-T_{x_*}^{k-\al+\bet}\partial^{\alpha-\beta}f(y_*)).\end{align}
Für alle $j\in\{1,\ldots,l\}$ und $\beta\leq\alpha$ gilt nach Lemma \ref{c} (c) aber
\begin{align}\label{gl.gg}&T_{x_{C_j}}^{k-\al+\bet}\partial^{\alpha-\beta}f(y_*)-T_{x_*}^{k-\al+\bet}\partial^{\alpha-\beta}f(y_*)\notag\\
&=\sum_{\vert\gamma\vert\leq k-\al+\bet}\frac{(y_*-x_{C_j})^{\gamma}}{\gamma!}R_{x_*}^{k-\al+\bet-\vert\gamma\vert} \partial^{\alpha-\beta+\gamma}f(x_{C_j}).\end{align}
Es gilt $\Vert x_{C_j}-x_*\Vert\leq \Vert x_{C_j}-y_{C_j}\Vert+\Vert y_{C_j}-x_*\Vert<4\delta_*< 4\delta$ mit Lemma \ref{l} (f). Ebenso wie $x_*$ ist auch $x_{C_j}$ in $\overline{B}_1(x_0)$, denn $\Vert x_{C_j}-x_0\Vert\leq \Vert x_{C_j}-y_{C_j}\Vert +\Vert y_{C_j}-x_0\Vert<4\delta<1$. Mit \eqref{gl.g} folgt daher
\[q(R_{x_*}^{k-\al+\bet-\vert\gamma\vert}\partial^{\alpha-\beta+\gamma}f(x_{C_j}))\leq \Vert x_{C_j}-x_*\Vert^{k-\al+\vert\beta\vert-\vert\gamma\vert}\eta\leq (4\delta_*)^{k-\al+\vert\beta\vert-\vert\gamma\vert}\eta\]
für alle $\vert\gamma\vert\leq k-\al+\bet$. Weiter gilt $\vert(y_*-x_{C_j})^{\gamma}\vert\leq\Vert y_*-x_{C_j}\Vert^{\vert\gamma\vert}\leq (2\Vert y_{C_j}-x_{C_j}\Vert)^{\vert\gamma\vert}\leq (4\delta_*)^{\vert\gamma\vert}.$
Dies ergibt
\begin{align*}&q(T_{x_{C_j}}^{k-\al+\bet}\partial^{\alpha-\beta}f(y_*)-T_{x_*}^{k-\al+\bet}\partial^{\alpha-\beta}f(y_*))\\
&\leq\sum_{\vert\gamma\vert\leq k-\al+\bet}\vert(y_*-x_{C_j})^{\gamma}\vert q(R_{x_*}^{k-\al+\bet-\vert\gamma\vert}\partial^{\alpha-\beta+\gamma}f(x_{C_j}))\\ 
&\leq (k+1)^n(4\delta_*)^{k-\al+\bet}\eta\leq (k+1)^n4^k\delta_*^{\bet}\eta.\end{align*}
Zusammen mit $\binom{\alpha}{\beta}\leq\alpha!\leq (k!)^n$ und $\vert\partial^{\beta}\varphi_{C_j}(y_*)\vert<\frac{N_k}{\delta_*^{\bet}}$ schließen wir
\begin{align*}&q(\partial^{\alpha}F(y_*)-T_{x_*}^{k-\al}\p f(y_*))\\
&\leq\sum_{j=1}^l\sum_{\beta\leq\alpha}\binom{\alpha}{\beta}\vert\partial^{\beta}\varphi_{C_j}(y_*)\vert q(T_{x_{C_j}}^{k-\al+\bet}\partial^{\alpha-\beta}f(y_*)-T_{x_*}^{k-\al+\bet}\partial^{\alpha-\beta}f(y_*))\\ 
&\leq c((k+2)!)^n 4^k N_k\eta<\frac{\eps}{4}.
\end{align*}
Mit $\eqref{gl.d}$ folgt $q(\partial^{\alpha}F(y_*)-f_{\alpha}(x_0))<\eps$, was zu zeigen war.

\ee{Zweiter Schritt: Die Abbildung $\Phi_{\R^n,A}^k$ ist stetig und linear}.

Die Linearität von $\Phi_{\R^n,A}^k$ folgt aus Lemma \ref{c} (b). Zum Beweis der Stetigkeit seien $q$ eine stetige Halbnorm auf $E$, $K\subseteq \R^n$ kompakt, $f=(f_{\alpha})_{\al\leq k}\in\E$ und $F\coloneqq\Phi_{\R^n,A}^k(f)$. Zudem seien $\al\leq k$ und $y_*\in K$. Falls $y_*\in K\cap A$ gilt, ist $\partial^{\alpha}F(y_*)=f_{\alpha}(y_*)$, also
\[q(\partial^{\alpha}F(y_*))\leq\Vert f\Vert_{\mathcal{E}^k,q,K\cap A}=: Q_1(f).\]
Angenommen, $y_*$ liegt in $K\setminus A$ mit $\delta_*\coloneqq d(y_*,A)<1$. Seien $x_*\in A$ mit $\Vert y_*-x_*\Vert=\delta_*$ und $C\in W$ mit $y_*\in C$. Weiter seien $C_1,\ldots,C_l$ diejenigen Würfel in $W$, die $C$ schneiden, und $K_A$ sowie $K^A$ kompakte Mengen wie in Lemma \ref{n}. Zum einen gilt
\begin{align*}
q(T_{x_*}^{k-\al}\p f(y_*))\leq \sum_{\bet\leq k-\al}\underbrace{\vert(y_*-x_*)^{\beta}\vert}_{\leq\delta_*^{\bet}\leq 1}\underbrace{q(f_{\alpha+\beta}(x_*))}_{\leq \Vert f\Vert'_{\EE^k,q,K_A}}\leq (k+1)^n\Vert f\Vert_{\EE^k,q,K_A}.
\end{align*}
Zum anderen gilt nach \eqref{gl.gg} für alle $j\in\{1,\ldots,l\}$ mit $x_{C_j}\neq x_*$ und $\beta\leq\alpha$ die Abschätzung
\begin{align*}
&q(T_{x_{C_j}}^{k-\al+\bet}\partial^{\alpha-\beta}f(y_*)-T_{x_*}^{k-\al+\bet}\partial^{\alpha-\beta}f(y_*))\\
&\leq\sum_{\vert\gamma\vert\leq k-\al+\bet}\underbrace{\vert(y_*-x_{C_j})^{\gamma}\vert}_{\leq(4\delta_*)^{\vert\gamma\vert}}\underbrace{\Vert x_{C_j}-x_*\Vert^{k-\al+\bet-\vert\gamma\vert}}_{\leq(4\delta_*)^{k-\al+\bet-\vert\gamma\vert}} \underbrace{\frac{q(R_{x_*}^{k-\al+\bet-\vert\gamma\vert}\partial^{\alpha-\beta+\gamma}f(x_{C_j}))}{\Vert x_{C_j}-x_*\Vert^{k-\al+\bet-\vert\gamma\vert}}}_{\leq\Vert f\Vert''_{\EE^k,q,K_A\cup K^A}}\\
&\leq (k+1)^n 4^k\delta_*^{\bet}\Vert f\Vert_{\EE^k,q,K_A\cup K^A}.
\end{align*}
Zusammen mit \eqref{gl.hh} und $\vert\partial^{\beta}\varphi_{C_j}(y_*)\vert<\frac{N_k}{\delta_*^{\bet}}$ folgt
\begin{align*}
q(\p F(y_*))
&\leq q(T_{x_*}^{k-\al}\p f(y_*))\\
&+\sum_{j=1}^l\sum_{\beta\leq\alpha}\binom{\alpha}{\beta}\vert\partial^{\beta}\varphi_{C_j}(y_*)\vert q(T_{x_{C_j}}^{k-\al+\bet}\partial^{\alpha-\beta}f(y_*)-T_{x_*}^{k-\al+\bet}\partial^{\alpha-\beta}f(y_*))\\
&\leq (k+1)^n\Vert f\Vert_{\EE^k,q,K_A}+c((k+2)!)^n 4^k N_k\Vert f\Vert_{\EE^k,q,K_A\cup K^A}=:Q_2(f).
\end{align*}
Zuletzt nehmen wir an, dass $y_*$ in $K'\coloneqq\{y\in K\colon d(y,A)\geq 1\}$ liegt. Da $K'$ kompakt ist mit $K'\subseteq\R^n\setminus A$, gibt es endlich viele Würfel $C'_1,\ldots,C'_m\in W$, welche $K'$ überdecken. Seien $C_1,\ldots,C_l$ diejenigen Würfel in $W$, die jeweils mindestens einen der Würfel $C'_1,\ldots,C'_m$ schneiden, und $L\coloneqq \{x_{C_1},\ldots,x_{C_l}\}$. Für alle $x\in C_1'\cup\ldots\cup C_m'$ gilt
\[F(x)=\sum_{j=1}^l\varphi_{C_j}(x)T_{x_{C_j}}^k f(x),\]
also folgt
\[\p F(y_*)=\sum_{j=1}^l\sum_{\beta\leq\alpha}\binom{\alpha}{\beta}\partial^{\beta}\varphi_{C_j}(y_*)T_{x_{C_j}}^{k-\al+\bet}\partial^{\alpha-\beta}f(y_*).\]
Setzen wir \[M\coloneqq\max(\{1\}\cup\{\Vert y-x_{C_j}\Vert\colon y\in K', 1\leq j\leq l\}),\] so gilt
\begin{align*}
q(T_{x_{C_j}}^{k-\al+\bet}\partial^{\alpha-\beta}f(y_*))
&\leq\sum_{\vert\gamma\vert\leq k-\al+\bet}\underbrace{\vert (y_*-{x_{C_j}})^{\gamma}\vert}_{\leq M^{\vert\gamma\vert}\leq M^k} \underbrace{q(f_{\alpha-\beta+\gamma}(x_{C_j}))}_{\leq \Vert f\Vert'_{\EE^k,q,L}}\\
&\leq (k+1)^n M^k\Vert f\Vert_{\EE^k,q,L}
\end{align*}
für alle $j\in\{1,\ldots,l\}$ und $\beta\leq\alpha$. Zusammen mit $\vert\partial^{\beta}\varphi_{C_j}(y_*)\vert<\frac{N_k}{d(y_*,A)^{\bet}}\leq N_k$ folgt
\begin{align*}q(\p F(y_*))&\leq\sum_{j=1}^l\sum_{\beta\leq\alpha}\binom{\alpha}{\beta}\vert\partial^{\beta}\varphi_{C_j}(y_*)\vert q(T_{x_{C_j}}^{k-\al+\bet}\partial^{\alpha-\beta}f(y_*))\\
&\leq c((k+2)!)^n N_k M^k\Vert f\Vert_{\EE^k,q,L}=:Q_3(f).\end{align*}
In allen drei Fällen gilt demnach
\[q(\p F(y_*))\leq Q_1(f)+Q_2(f)+Q_3(f).\]
Daraus folgt 
\[\Vert F\Vert_{C^k,q,K}\leq Q_1(f)+Q_2(f)+Q_3(f),\]
wobei die rechte Seite eine stetige Halbnorm in $f$ ist. Somit ist $\Phi_{\R^n,A}^k$ stetig und der Satz für $U=\R^n$ bewiesen.

Nun seien $U\subseteq\R^n$ eine lokalkonvexe Teilmenge mit dichtem Inneren und $A\subseteq U$ eine relativ abgeschlossene, lokalkompakte Teilmenge. Jedes $x\in A$ hat eine kompakte Umgebung $V_x$ in $A$. Sei $U_x\subseteq U$ eine relativ offene Umgebung von $x$ mit $U_x\cap A\subseteq V_x$. Nach Lemma \ref{rr} gibt es eine $C^{\infty}$-Partition der Eins $(h_x)_{x\in A}$, $h$ auf $U$ mit $\text{supp}(h_x)\subseteq U_x$ für $x\in A$ und $\text{supp}(h)\subseteq U\setminus A$.
Für jedes $x\in A$ ist $K_x\coloneqq\text{supp}(h_x)\cap A\subseteq V_x$ kompakt, also abgeschlossen in $\R^n$. Nach dem vorigen Beweisteil existiert ein stetiger linearer Fortsetzungsoperator
\[\Phi_{\R^n,K_x}^k\colon \EE^k(K_x,E)\to C^k(\R^n,E).\]
Ein $f=(f_{\alpha})_{\al\leq k}\in\EE^k(A,E)$ gegeben, schreiben wir $f_x\coloneqq \big(\Phi_{\R^n,K_x}^k\big(f\vert_{K_x}\big)\big)\vert_U$. Dann gilt $\p f_x(v)=f_{\alpha}(v)$ für alle $\al\leq k$ und $v\in K_x$, und die Abbildung
\[\EE^k(A,E)\to C^k(U,E),\quad f\mapsto f_x\]
ist stetig und linear. Zusammen mit Lemma \ref{ii} sehen wir, dass die Abbildung
\[\Phi_{U,A}^k\colon \EE^k(A,E)\to C^k(U,E),\quad f\mapsto \sum_{x\in A}h_x\cdot f_x\]
stetig und linear ist. 
Es bleibt zu zeigen, dass $\Phi_{U,A}^k$ rechtsinvers zu $\rho_{A,U}^k$ ist. Seien dazu $f=(f_{\alpha})_{\al\leq k}\in\EE^k(A,E)$, $F\coloneqq \Phi_{U,A}^k(f)$, $\al\leq k$ und $v\in A$. Sei $x\in A$ derart, dass $v\in K_x$. Es ist $V\coloneqq U\setminus\text{supp}(h)$ eine relativ offene Teilmenge von $U$, die $A$ enthält. Für alle $w\in V$ gilt $\sum_{y\in A} h_y(w)=1$, also \[f_x(w)=\sum_{y\in A} h_y(w)f_x(w).\] Daraus folgt
\begin{align*}
f_{\alpha}(v)
&=\p f_x(v)
=\sum_{y\in A}\sum_{\beta\leq\alpha}\binom{\alpha}{\beta}\partial^{\beta}h_y(v)\partial^{\alpha-\beta} f_x(v)
=\sum_{y\in A}\sum_{\beta\leq\alpha}\binom{\alpha}{\beta}\partial^{\beta}h_y(v) f_{\alpha-\beta}(v)\\
&=\sum_{y\in A}\sum_{\beta\leq\alpha}\binom{\alpha}{\beta}\partial^{\beta}h_y(v)\partial^{\alpha-\beta} f_y(v)
=\p F(v),
\end{align*}
wobei wir bei der vorletzten Gleichheit benutzt haben, dass $f_{\alpha-\beta}(v)=\partial^{\alpha-\beta} f_y(v)$ im Falle $v\in K_y$ gilt und $\partial^{\beta}h_y(v)=0$ für $v\notin K_y$. Somit ist $\rho_{A,U}^k(F)=f$.
\end{proof}

\subsection*{Glatte Fortsetzungen}
\addcontentsline{toc}{subsection}{\protect\numberline{}Glatte Fortsetzungen}

Soeben haben wir gezeigt, dass im Falle $k<\infty$ stetige lineare Fortsetzungsoperatoren $\EE^k(A,E)\to C^k(U,E)$ existieren. Wenn $E$ metrisierbar ist, dann besitzt jedes $f\in\EE^{\infty}(A,E)$ eine Fortsetzung in $C^{\infty}(U,E)$.

\begin{satz}\label{ll}
Es seien $E$ metrisierbar, $U\subseteq\R^n$ eine lokalkonvexe Teilmenge mit dichtem Inneren und $A\subseteq U$ eine relativ abgeschlossene, lokalkompakte Teilmenge.
Dann ist die Abbildung
\[\rho_{A,U}^{\infty}\colon C^{\infty}(U,E)\to \EE^{\infty}(A,E),\quad f\mapsto ((\p f)\vert_A)_{\alpha\in(\N_0)^n}\]
surjektiv.
\end{satz}

Mit dem Isomorphismus $\EE^{\infty}(A,E)\cong C^{\infty}(A,E)$ für lokalkonvexe Teilmengen $A\subseteq U$ mit dichtem Inneren folgt:

\begin{korollar}
Es seien $E$ metrisierbar, $U\subseteq\R^n$ eine lokalkonvexe Teilmenge mit dichtem Inneren und $A\subseteq U$ eine relativ abgeschlossene, lokalkompakte, lokalkonvexe Teilmenge mit dichtem Inneren. Dann ist die Einschränkung
\[C^{\infty}(U,E)\to C^{\infty}(A,E),\quad f\mapsto f\vert_A\]
surjektiv.
\end{korollar}

Den nun folgenden Ausführungen dient Whitneys Beweis \cite[\S\ 12]{Wh} zum Vorbild.

\begin{proof}[Beweis von Satz \ref{ll}]
Wir beweisen den Satz zuerst für den Spezialfall $U=\R^n,$ wobei der Fall $A=\emptyset$ klar ist. Sei $A\subseteq \R^n$ eine abgeschlossene, nichtleere Teilmenge, und
sei $f=(f_{\alpha})_{\alpha\in(\N_0)^n}\in\EE^{\infty}(A,E)$. Es seien $\{q_j\colon j\in\N\}$ eine abzählbare, die Topologie von $E$ definierende Menge von Halbnormen $q_j$ auf $E$, $a\in A$ und $(\varphi_C)_{C\in W}$ eine $C^{\infty}$-Partition der Eins auf $\R^n\setminus A$ wie in Lemma \ref{j}. Für jedes $i\in\N$ wählen wir ein
\[M_i>\max\{q_j(f_{\alpha}(x))\colon j\leq i,\al\leq i, x\in A\cap\overline{B}_{2^i}(a)\}\] und ein \[0<\delta_i<\frac{1}{c((i+2)!)^n8^{i+1}M_{i+1}N_i}\]
mit $c$ wie in Lemma \ref{l} (d) und $N_i$ wie in Lemma \ref{j}. 
Nach rekursivem Verkleinern können wir annehmen, dass \[\delta_{i+1}<\frac{\delta_i}{2}\] für alle $i\in\N$. Für alle $C\in W$ setzen wir
\[g_C\coloneqq\max(\{i\in\N\colon d(y_C,A)<\delta_i\}\cup\{0\}).\]
Wir werden sehen, dass
\begin{align*}F\colon\R^n\to E,\quad x\mapsto \begin{cases} \sum_{C\in W}\varphi_C(x)T_{x_C}^{g_C}f(x) &\text{wenn }x\in \R^n\setminus A; \\ f_0(x) &\text{wenn }x\in A \end{cases} \end{align*}
ein Urbild von $f$ unter $\rho_{A,\R^n}^{\infty}$ ist.

\ee{Behauptung: Für alle $\alpha\in(\N_0)^n$, $x_0\in\partial A$, $j\in\N$ und $\eps>0$ gibt es ein $\delta>0$ derart, dass \[q_j(\partial^{\alpha}F(y)-f_{\alpha}(x_0))<\eps\] für alle $y\in\R^n\setminus A$ mit $\Vert y-x_0\Vert<\delta$.}

Ist dies gezeigt, kann man analog wie im Beweis des Satzes \ref{k} mithilfe von Lemma \ref{b} argumentieren, dass $F$ eine $C^{\infty}$-Funktion mit $(\partial^{\alpha}F)\vert_A=f_{\alpha}$ für alle $\alpha\in(\N_0)^n$ ist.

Zum Beweis der Behauptung seien $\alpha$, $x_0$, $j$, $\eps$ wie oben und $k\coloneqq\al$. 
Weiter sei $\widetilde{F}\coloneqq\Phi_{\R^n,A}^k(\text{pr}_k(f))$ die in \eqref{gl.h} definierte Fortsetzung des Whitneyschen $k$-Jets $\text{pr}_k(f)$.
Wir wählen ein $i\in\N$ so groß, dass $i\geq\max\{j,k+2\}$, $\frac{1}{2^i}<\eps$ und $x_0\in \overline{B}_{2^i}(a)$ erfüllt sind, und ein $\delta>0$ so klein, dass $\delta<\min\{\delta_i,\frac{1}{2}\}$ und
\begin{align}\label{gl.ii}q_j(\partial^{\alpha}\widetilde{F}(y)-f_{\alpha}(x_0))<\frac{\eps}{2}\end{align}
für alle $y\in \R^n\setminus A$ mit $\Vert y-x_0\Vert<\delta$ gelten. Sei $y_*\in\R^n\setminus A$ mit $\Vert y_*-x_0\Vert<\delta$ beliebig. Zudem seien $\delta_*\coloneqq d(y_*,A)$, $C\in W$ mit $y_*\in C$ und $C_1,\ldots,C_m$ diejenigen Würfel in $W$, welche $C$ schneiden. Dann ist
\[r\coloneqq\max\{s\in\N\colon\delta_*<\delta_s\}\geq i.\]
Für jedes $l\in\{1,\ldots,m\}$ gilt $k< g_{C_l}\leq r+1$; zum einen ist nämlich 
\[\delta_{g_{C_{l}}+1}\leq d(y_{C_l},A)< 2\delta_*<2\delta_i<\delta_{i-1}\] 
und somit $g_{C_l}>i-2\geq k$, und zum anderen gilt
\[\delta_{g_{C_l}}>d(y_{C_l},A)> \frac{\delta_*}{2}\geq\frac{\delta_{r+1}}{2}>\delta_{r+2},\]
woraus $g_{C_l}\leq r+1$ folgt. 
Für alle $x\in C$ gilt
\[F(x)-\widetilde{F}(x)=\sum_{l=1}^m \varphi_{C_l}(x)(T_{x_{C_l}}^{g_{C_l}}f(x)-T_{x_{C_l}}^kf(x));\]
dies ergibt
\[\p F(y_*)-\p\widetilde{F}(y_*)=\sum_{l=1}^m\sum_{\beta\leq\alpha}\binom{\alpha}{\beta}\partial^{\beta}\varphi_{C_l}(y_*)\big(T_{x_{C_l}}^{g_{C_l}-k+\bet}\partial^{\alpha-\beta}f(y_*)-T_{x_{C_l}}^{\bet}\partial^{\alpha-\beta}f(y_*)\big)\]
mit
\[T_{x_{C_l}}^{g_{C_l}-k+\bet}\partial^{\alpha-\beta}f(y_*)-T_{x_{C_l}}^{\bet}\partial^{\alpha-\beta}f(y_*)=\sum_{\bet<\vert\gamma\vert\leq g_{C_l}-k+\bet}\frac{(y_*-x_{C_l})^{\gamma}}{\gamma!}f_{\alpha-\beta+\gamma}(x_{C_l})\]
für alle $l\in\{1,\ldots,m\}$ und $\beta\leq\alpha$.
Es gilt $x_{C_l}\in\overline{B}_{2^{r+1}}(a)$, denn aus $\Vert x_0-a\Vert\leq 2^i\leq 2^r$ und $\Vert x_{C_l}-x_0\Vert< 4\delta<2\leq 2^r$ folgt $\Vert x_{C_l}-a\Vert\leq\Vert x_{C_l}-x_0\Vert+\Vert x_0-a\Vert\leq2^{r+1}.$
Somit ist
\begin{align*}
&q_j(T_{x_{C_l}}^{g_{C_l}-k+\bet}\partial^{\alpha-\beta}f(y_*)-T_{x_{C_l}}^{\bet}\partial^{\alpha-\beta}f(y_*))\\
&\leq\underbrace{\sum_{\bet<\vert\gamma\vert\leq g_{C_l}-k+\bet}}_{\leq(r+2)^n\text{ Summanden} }\underbrace{\vert(y_*-x_{C_l})^{\gamma}\vert}_{\leq(4\delta_*)^{\vert\gamma\vert}} \underbrace{q_j(f_{\alpha-\beta+\gamma}(x_{C_l}))}_{\leq M_{r+1}}\\
&\leq (r+2)^n 4^{r+1}M_{r+1}\delta_*^{\bet}\delta_r.
\end{align*}
Zusammen mit $\vert\partial^{\beta}\varphi_{C_l}(y_*)\vert<\frac{N_r}{\delta_*^{\bet}}$ folgt
\begin{align*}
&q_j(\p F(y_*)-\p\widetilde{F}(y_*))\\
&\leq \sum_{l=1}^m\sum_{\beta\leq\alpha}\binom{\alpha}{\beta}\vert\partial^{\beta}\varphi_{C_l}(y_*)\vert q_j(T_{x_{C_l}}^{g_{C_l}-k+\bet}\partial^{\alpha-\beta}f(y_*)-T_{x_{C_l}}^{\bet}\partial^{\alpha-\beta}f(y_*))\\
&\leq c((r+2)!)^n 4^{r+1}M_{r+1}N_r\delta_r<\frac{1}{2^{r+1}}\leq\frac{1}{2}\cdot\frac{1}{2^i}<\frac{\eps}{2}.
\end{align*}
Gemeinsam mit \eqref{gl.ii} ergibt dies
\[q_j(\partial^{\alpha} F(y_*)-f_{\alpha}(x_0))\leq q_j(\partial{^\alpha}F(y_*)-\partial{^\alpha}\widetilde{F}(y_*))+q_j(\partial^{\alpha}\widetilde{F}(y_*)-f_{\alpha}(x_0))<\eps,\]
was die Behauptung beweist. Damit ist der Satz für $U=\R^n$ gezeigt.

Nun seien $U\subseteq\R^n$ eine lokalkonvexe Teilmenge mit dichtem Inneren und $A\subseteq U$ eine relativ abgeschlossene, lokalkompakte Teilmenge. Wie im Beweis von Satz \ref{k} können wir eine $C^{\infty}$-Partition der Eins $(h_x)_{x\in A}$, $h$ auf $U$ derart wählen, dass $K_x\coloneqq \text{supp}(h_x)\cap A$ für alle $x\in A$ kompakt und $\text{supp}(h)\subseteq U\setminus A$ ist. Sei $f\in\EE^{\infty}(A,E)$ beliebig. Für jedes $x\in A$ besitzt $f\vert_{K_x}$ ein Urbild $f_x$ unter $\rho_{K_x,\R^n}^{\infty}\colon C^{\infty}(\R^n,E)\to\EE^{\infty}(K_x,E)$. Wir setzen
\[F\coloneqq \sum_{x\in A}h_x\cdot (f_x\vert_U)\in C^{\infty}(U,E).\]
Analog wie im Beweis von Satz \ref{k} sieht man, dass $F$ ein Urbild von $f$ unter $\rho^{\infty}_{A,U}$ ist.
\end{proof}

\subsection*{Zurückholen von Whitneyschen $k$-Jets}
\addcontentsline{toc}{subsection}{\protect\numberline{}Zurückholen von Whitneyschen $k$-Jets}

Es seien $k\in\N_0\cup\{\infty\}$ und $g\colon U\to V$ eine $C^k$-Funktion zwischen lokalkonvexen Teilmengen $U\subseteq\R^s$ und $V\subseteq\R^t$ mit dichtem Inneren. Wir erinnern an die stetige lineare Abbildung \[g^*\colon C^k(V,E)\to C^k(U,E),\quad f\to f\circ g.\] Für gegebene Teilmengen $A\subseteq U$ und $B\subseteq V$ mit $g(A)\subseteq B$ konstruieren wir eine stetige lineare Abbildung
\[(g\vert_A)^*\colon\EE^k(B,E)\to\EE^k(A,E).\]

\begin{definition}
Für alle $f=(f_{\beta})_{\bet\leq k}\in\J^k(B,E)$ setzen wir \[(g\vert_A)^*(f)\coloneqq \bigg(A\ni x\mapsto\sum_{\bet\leq\al}p_{\alpha,\beta}\left((\partial^{\gamma}g(x))_{\vert\gamma\vert\leq\al}\right) f_{\beta}(g(x))\bigg)_{\al\leq k}\in\J^k(A,E),\]
wobei hier die $p_{\alpha,\beta}$ für $\alpha\in(\N_0)^s$ mit $\al\leq k$ und $\beta\in(\N_0)^t$ mit $\bet\leq\al$
die Fa\`a-di-Bruno-Polynome aus Definition \ref{zz} seien.
\end{definition}

\begin{bemerkung}\label{ddd}
\begin{itemize}
\item[(a)] Ist $l\leq k$, dann gilt
\[\text{pr}_l((g\vert_A)^*(f))=(g\vert_A)^*(\text{pr}_l(f))\]
für alle $f\in\J^k(B,E)$.
\item[(b)] Ist $A'\subseteq A$ eine Teilmenge, so gilt
\[((g\vert_A)^*(f))\vert_{A'}=(g\vert_{A'})^*(f\vert_{g(A')})\]
für alle $f\in\J^k(B,E)$.
\end{itemize}
\end{bemerkung}

\begin{bemerkung}\label{ggg}
Mit der Fa\`a-di-Bruno-Formel gilt
\[((\p(f\circ g))\vert_A)_{\al\leq k}=(g\vert_A)^*(((\partial^{\beta}f)\vert_B)_{\bet\leq k}).\]
für alle $f\in C^k(V,E)$.
\end{bemerkung}

\begin{satz}\label{uu}
Es seien $U\subseteq\R^s$,$V\subseteq\R^t$ und $W\subseteq\R^u$ lokalkonvexe Teilmengen mit dichtem Inneren, $g\colon U\to V$ sowie $h\colon V\to W$ beide $C^k$-Funktionen und $A\subseteq U$, $B\subseteq V$ sowie $C\subseteq W$ Teilmengen mit $g(A)\subseteq B$ und $h(B)\subseteq C$.
\begin{itemize}
\item[(a)] Für alle $f\in\EE^k(B,E)$ gilt $(g\vert_A)^*(f)\in\EE^k(A,E)$, und die lineare Abbildung
\[(g\vert_A)^*\colon\EE^k(B,E)\to\EE^k(A,E),\quad f\mapsto (g\vert_A)^*(f)\]
ist stetig. Wir schreiben auch $g^*\coloneqq(g\vert_A^B)^*\coloneqq (g\vert_A)^*$.
\item[(b)] Es gilt
\[(g\vert_A)^*\circ(h\vert_B)^*=((h\circ g)\vert_A)^*.\]
als Abbildung $\EE^k(C,E)\to\EE^k(A,E)$.
\item[(c)] Sei $\ee{id}_U\colon U\to U$ die Identität. Dann gilt
\[(\ee{id}_U\vert_A^A)^*=\ee{id}_{\EE^k(A,E)}.\]
Allgemeiner: Für jede Teilmenge $A'\subseteq A$ ist $(\text{\ee{id}}_U\vert_{A'}^A)^*$ gegeben durch
\[(\text{\ee{id}}_U\vert_{A'}^A)^*\colon\EE^k(A,E)\to\EE^k(A',E),\quad f\mapsto f\vert_{A'}.\]
\end{itemize}
Aus (b) und (c) folgt: Falls $g$ ein $C^k$-Diffeo"-mor"-phis"-mus ist und $g(A)=B$, so ist $(g\vert_A)^*$ ein Iso"-mor"-phis"-mus von topologischen Vektorräumen mit $((g\vert_A)^*)^{-1}=((g^{-1})\vert_B)^*.$
\end{satz}

\begin{proof}
\begin{itemize}
\item[(a)] Wir zeigen die Aussage zunächst für den Spezialfall, dass $k<\infty$ und $B$ kompakt ist.
Nach Satz \ref{k} existiert ein stetiger linearer Fortsetzungsoperator
\[\Phi_{V,B}^k\colon \EE^k(B,E)\to C^k(V,E).\]
Zudem betrachten wir die stetigen linearen Abbildungen
\[g^*\colon C^k(V,E)\to C^k(U,E),\quad f\to f\circ g\]
und
\[\rho_{A,U}^k\colon C^k(U,E)\to \EE^k(A,E),\quad f\mapsto ((\p f)\vert_A)_{\al\leq k}.\]
Mit Bemerkung \ref{ggg} gilt
\[(g\vert_A)^*=\rho^k_{A,U}\circ g^*\circ \Phi^k_{V,B},\]
womit gezeigt ist, dass $(g\vert_A)^*$ tatsächlich Werte in $\EE^k(A,E)$ hat und stetig ist.

Nun nehmen wir an, dass $k\in\N_0\cup\{\infty\}$ beliebig sei und $B$ nicht notwendig kompakt. Seien $f\in\EE^k(B,E)$, $l\in\N_0$ mit $l\leq k$, $q$ eine stetige Halbnorm auf $E$ und $K\subseteq A$ eine kompakte Teilmenge. Nach dem Vorigen liegt
\[\text{pr}_l\big(((g\vert_A)^*(f))\vert_K\big)=(g\vert_K)^*\big(\text{pr}_l\big(f\vert_{g(K)}\big)\big)\]
in $\EE^l(K,E)$. Da $l$ und $K$ beliebig waren, folgt $(g\vert_A)^*(f)\in\EE^k(A,E)$. Weiterhin gilt
\[\Vert(g\vert_A)^*(f)\Vert_{\EE^l,q,K}=\big\Vert(g\vert_K)^*\big(\text{pr}_l\big(f\vert_{g(K)}\big)\big)\big\Vert_{\EE^l,q,K}\to 0\ \text{für }f\to 0,\]
was die Stetigkeit von $(g\vert_A)^*$ beweist.

\item[(b)] Zuerst seien $k<\infty$ und $C$ kompakt.
Da das Diagramm
\[\begin{xy}\xymatrix{
\EE^k(C,E)\ar[d]_{\Phi^k_{W,C}}\ar[rr]^{(h\vert_B)^*} & & \EE^k(B,E)\ar[rr]^{(g\vert_A)^*} & & \EE^k(A,E)\\
C^k(W,E)\ar[rr]^{h^*} & & C^k(V,E)\ar[rr]^{g^*}\ar[u]^{\rho^k_{B,V}} & & C^k(U,E)\ar[u]_{\rho^k_{A,U}}
}\end{xy}\]
mit einem Fortsetzungsoperator $\Phi^k_{W,C}$ kommutativ ist, erhalten wir
\begin{align*}
(g\vert_A)^*\circ(h\vert_B)^*
&=(g\vert_A)^*\circ (\rho^k_{B,V}\circ h^*\circ\Phi^k_{W,C})
=((g\vert_A)^*\circ \rho^k_{B,V})\circ (h^*\circ\Phi^k_{W,C})\\
&=(\rho_{A,U}^k\circ g^*)\circ(h^*\circ\Phi^k_{W,C})
=\rho_{A,U}^k\circ (g^*\circ h^*)\circ\Phi^k_{W,C}\\
&=\rho^k_{A,U}\circ (h\circ g)^*\circ\Phi^k_{W,C}
=((h\circ g)\vert_A)^*.
\end{align*}

Nun seien $k\in\N_0\cup\{\infty\}$ beliebig und $C$ nicht notwendig kompakt. Für jedes $f\in\EE^k(C,E)$ und für alle $l\in\N_0$ mit $l\leq k$ und kompakte Teilmengen $K\subseteq A$ gilt nach dem Vorigen
\begin{align*}
\text{pr}_l\big((((g\vert_A)^*\circ(h\vert_B)^*)(f))\vert_K\big)
&=\big((g\vert_K)^*\circ(h\vert_{g(K)})^*\big)\big(\text{pr}_l\big(f\vert_{h(g(K))}\big)\big)\\
&=((h\circ g)\vert_K)^*\big(\text{pr}_l\big(f\vert_{h(g(K))}\big)\big)\\
&=\text{pr}_l\big((((h\circ g)\vert_A)^*(f))\vert_K\big)
\end{align*}
und somit
\[((g\vert_A)^*\circ(h\vert_B)^*)(f)=((h\circ g)\vert_A)^*(f).\]
\item[(c)] Es genügt, den ersten Teil der Aussage für $k<\infty$ und kompaktes $A$ zu zeigen. Mit einem Fortsetzungsoperator $\Phi^k_{U,A}\colon \EE^k(A,E)\to C^k(U,E)$ gilt
\[(\text{id}_U\vert_A^A)^*=\rho^k_{A,U}\circ (\text{id}_U)^*\circ\Phi^k_{U,A}=\rho^k_{A,U}\circ \text{id}_{C^k(U,E)}\circ\Phi^k_{U,A}=\text{id}_{\EE^k(A,E)}.\]
Für jede Teilmenge $A'\subseteq A$ und alle $f\in\EE^k(A,E)$ folgt
\[(\text{id}_U\vert_{A'}^A)^*(f)=((\text{id}_U\vert_{A}^A)^*(f))\vert_{A'}=f\vert_{A'}.\]
\end{itemize}
\end{proof}


\begin{beispiel}\label{fff}
Es seien $m\in\{1,\ldots,n\}$ und $A\subseteq\R^m$ sowie $B\subseteq\R^{n-m}$ mit $0\in B$ zwei Teilmengen. 
Wir zeigen, dass es eine Einbettung von topologischen Vektorräumen $\EE^k(A,E)\hookrightarrow\EE^k(A\times B,E)$ gibt: Die Inklusion
\[j\colon\R^m\to\R^n,\quad x\mapsto (x,0)\]
und die Projektion
\[\pi\colon \R^n=\R^m\times\R^{n-m}\to\R^m,\quad (x,y)\mapsto x\]
induzieren stetige lineare Abbildungen
\[(j\vert_A)^*\colon\EE^k(A\times B,E)\to\EE^k(A,E),\quad (f_{\alpha})_{\substack{\begin{subarray}{l}\alpha\in(\N_0)^n\colon\\ \al\leq k\end{subarray}}}
\mapsto(f_{(\beta,0)}\circ j\vert_A)_{\substack{\begin{subarray}{l}\beta\in(\N_0)^m\colon\\ \bet\leq k\end{subarray}}}\]
bzw.
\[(\pi\vert_{A\times B})^*\colon\EE^k(A,E)\to\EE^k(A\times B,E),\quad (f_{\beta})_{\substack{\begin{subarray}{l}\beta\in(\N_0)^m\colon\\ \bet\leq k\end{subarray}}}\mapsto (g_{\alpha})_{\substack{\begin{subarray}{l}\alpha\in(\N_0)^n\colon\\ \al\leq k\end{subarray}}}\]
mit 
\[g_{\alpha}=
\begin{cases}
f_{\beta}\circ\pi\vert_{A\times B}&\text{wenn }\alpha=(\beta,0)\text{ für ein }\beta\in(\N_0)^m;\\
0&\text{sonst}
\end{cases}\]
für $\alpha\in(\N_0)^n$ mit $\al\leq k$. Anhand von
\[(j\vert_A)^*\circ(\pi\vert_{A\times B})^*=((\pi\circ j)\vert_A)^*=(\text{id}_{\R^m}\vert_A)^*=\text{id}_{\EE^k(A,E)}\]
sehen wir, dass $(\pi\vert_{A\times B})^*$ eine Einbettung von topologischen Vektorräumen ist.
\end{beispiel}

\newpage

\section{Ein Fortsetzungssatz für Funktionen auf Mannigfaltigkeiten}

In diesem Kapitel definieren wir Whitneysche $k$-Jets auf Teilmengen einer endlichdimensionalen $C^k$-Mannigfaltigkeit mit rauem Rand und verallgemeinern für solche den Whitneyschen Fortsetzungssatz.

\begin{definition}[s. {\cite[Definition 3.5.1 (a)]{Gl}}]
Es seien $k\in\N_0\cup\{\infty\}$ und $M$ ein Hausdorffscher topologischer Raum. Eine \ee{raue $E$-Karte} für $M$ ist ein Homöomorphismus $\varphi\colon U_{\varphi}\to V_{\varphi}$ auf einer offenen Teilmenge $U_{\varphi}\subseteq M$ in eine lokalkonvexe Teilmenge $V_{\varphi}\subseteq E$ mit dichtem Inneren. Ein \ee{$C^k$-Atlas} für $M$ ist eine Menge $\A$ von rauen $E$-Karten für $M$, die den folgenden Bedingungen genügt:
\begin{itemize}
\item[(A1)] $M=\bigcup_{\varphi\in\A}U_{\varphi}$;
\item[(A2)] Alle $\varphi,\psi\in\A$ sind \ee{$C^k$-kompatibel}, d.h. die \ee{Kartenwechsel}
\[\varphi\circ\psi^{-1}\colon \psi(U_{\varphi}\cap U_{\psi})\to V_{\varphi}\subseteq E\]
sind $C^k$-Funktionen;
\item[(A3)] Für alle $\varphi,\psi\in\A$ und $x\in U_{\varphi}\cap U_{\psi}$ gilt $\varphi(x)\in\partial V_{\varphi}$ genau dann, wenn $\psi(x)\in\partial V_{\psi}$.
\end{itemize}
\end{definition}

Wenn $k\geq 1$ oder $E$ endlichdimensional ist, dann folgt (A3) automatisch aus (A1) und (A2) (s. \cite[Lemma 3.5.6 und Remark 3.5.8]{Gl}).

\begin{definition}[s. {\cite[Definition 3.5.1 (b)]{Gl}}]
Eine \ee{auf $E$ modellierte $C^k$-Man"-nig"-fal"-tig"-keit mit rauem Rand} ist ein Hausdorffscher topologischer Raum $M$ zusammen mit einem maximalen $C^k$-Atlas $\A$ von rauen $E$-Karten für $M$. Ist ein solches Paar $(M,\A)$ gegeben, so verwenden wir das Wort \glqq Karte\grqq\ nur für die rauen $E$-Karten $\varphi\in\A$.
Eine \ee{Karte um $x\in M$} ist eine ist eine Karte $\varphi\colon U_{\varphi}\to V_{\varphi}$ in $\A$ mit $x\in U_{\varphi}$. Eine auf $\R^n$ modellierte $C^k$-Man"-nig"-fal"-tig"-keit mit rauem Rand bezeichnet man als \ee{$n$-dimensional}.
\end{definition}

\begin{definition}[s. {\cite[Definition 3.5.19]{Gl}}]\label{ccc}
Eine Teilmenge $A\subseteq M$ heißt \ee{volldimensionale Untermannigfaltigkeit}, wenn für jedes $x\in A$ eine Karte $\varphi\colon U_{\varphi}\to V_{\varphi}\subseteq E$ für $M$ um $x$ existiert, so dass $\varphi(A\cap U_{\varphi})$ eine lokalkonvexe Teilmenge von $E$ mit dichtem Inneren ist. Solche Karten $\varphi$ nennt man \ee{an $A$ angepasst}, und die Funktionen $\varphi\vert_{A\cap U_{\varphi}}\colon A\cap U_{\varphi}\to \varphi(A\cap U_{\varphi})$ bilden einen $C^k$-Atlas für $A$. Mit dem zugehörigen maximalen $C^k$-Atlas ist $A$ eine $C^k$-Mannigfaltigkeit mit rauem Rand.
\end{definition}

Jede lokalkonvexe Teilmenge $U\subseteq\R^n$ mit dichtem Inneren können wir als volldimensionale $C^k$-Unter"-mannig"-faltig"-keit von $\R^n$ auffassen, versehen mit dem maximalen $C^k$-Atlas von rauen Karten, welcher $\{\text{id}_U\}$ enthält (vgl. \cite[Example 3.5.3 (b)]{Gl}).

\newpage
\begin{definition}[s. {\cite[Definition 3.5.11]{Gl}}]
Es seien $M$ und $N$ beide $C^k$-Man"-nig"-fal"-tig"-kei"-ten mit rauem Rand, modelliert auf lokalkonvexen Räumen $E$ bzw. $F$. Eine Funktion $f\colon M\to N$ heißt \ee{$C^k$-Funktion}, wenn $f$ stetig ist und
\[\psi\circ f\circ\varphi^{-1}\colon E\supseteq\varphi(f^{-1}(U_{\psi})\cap U_{\varphi})\to F\]
für jede Karte $\varphi\colon U_{\varphi}\to V_{\varphi}$ von $M$ und $\psi\colon U_{\psi}\to V_{\psi}$ von $N$ eine $C^k$-Funktion ist. Wir schreiben $C^k(M,N)$ für die Menge aller $C^k$-Funktionen $f\colon M\to N$.
\end{definition}

\subsection*{$C^k$-Parakompaktheit und $C^k$-Regularität}
\addcontentsline{toc}{subsection}{\protect\numberline{}$C^k$-Parakompaktheit und $C^k$-Regularität}

Von besonderem Interesse für uns sind Mannigfaltigkeiten, die $C^k$-parakompakt oder $C^k$-regulär sind.

\begin{definition}[s. {\cite[Definition 7.11]{Gl2}}]
Es seien $k\in\N_0\cup\{\infty\}$ und $M$ eine $C^k$-Mannigfaltigkeit mit rauem Rand. Eine Partition der Eins $(h_j)_{j\in J}$ auf dem topologischen Raum $M$ heißt \ee{$C^k$-Partition der Eins}, wenn jedes $h_j$ eine $C^k$-Funktion ist.
Man nennt $M$ \ee{$C^k$-parakompakt}, wenn für jede offene Überdeckung $(U_j)_{j\in J}$ von $M$ eine $C^k$-Partition der Eins $(h_j)_{j\in J}$ auf $M$ existiert, die $(U_j)_{j\in J}$ in dem Sinne \ee{untergeordnet} ist, dass $\text{supp}(h_j)\subseteq U_j$ für alle $j\in J$ gilt.
\end{definition}

\begin{satz}
Sei $f\colon N\to M$ eine $C^k$-Funktion zwischen $C^k$-Mannigfaltigkeiten mit rauem Rand, die eine topologische Einbettung ist. Wenn jede offene Teilmenge von $M$ $C^k$-parakompakt ist, so  ist auch $N$ $C^k$-parakompakt.
\end{satz}

\begin{proof}
Sei $(U_j)_{j\in J}$ eine offene Überdeckung von $N$. Dann ist $(f(U_j))_{j\in J}$ eine Überdeckung von $f(N)$ mit relativ offenen Teilmengen $f(U_j)\subseteq f(N)$. Es gibt offene Teilmengen $V_j\subseteq M$ mit $f(U_j)=f(N)\cap V_j$ für $j\in J$. Nach Voraussetzung ist die offene Menge $V\coloneqq\bigcup_{j\in J}V_j$ $C^k$-parakompakt, also existiert eine $C^k$-Partition der Eins $(f_j)_{j\in J}$ auf $V$, die $(V_j)_{j\in J}$ untergeordnet ist. Die Funktionen 
\[h_j\coloneqq f_j\circ f\vert^V\colon N\to\R\]
bilden eine $C^k$-Partition der Eins $(h_j)_{j\in J}$ auf $N$, die $(U_j)_{j\in J}$ untergeordnet ist.
\end{proof}

\begin{beispiele}\label{nn}
\begin{itemize}
\item[(a)] Ist $M$ eine $C^k$-Mannigfaltigkeit mit rauem Rand, deren offene Teilmengen alle $C^k$-parakompakt sind, so ist jede volldimensionale Untermannigfaltigkeit $A\subseteq M$ $C^k$-parakompakt.
\item[(b)] Bekanntlich sind in $\R^n$ alle offenen Teilmengen $C^{\infty}$-parakompakt. Somit ist jede lokalkonvexe Teilmenge $U\subseteq \R^n$ mit dichtem Inneren $C^{\infty}$-parakompakt.
\end{itemize}
\end{beispiele}

\begin{definition}[s. {\cite[Proposition 3.5.29 und Definition 3.5.30]{Gl}}]\label{oo}
Sei $k\in\N_0\cup\{\infty\}$. Eine $C^k$-Mannigfaltigkeit $M$ mit rauem Rand heißt \ee{$C^k$-regulär}, wenn sie die folgenden äquivalenten Bedingungen erfüllt:\newpage
\begin{itemize}
\item[(a)] Die Topologie auf $M$ ist initial bezüglich $C^k(M,\R)$.
\item[(b)] Für jedes $x\in M$ und jede $x$-Umgebung $U\subseteq M$ existiert eine $C^k$-Funktion $f\colon M\to\R$ mit $f(x)\neq 0$ und $\text{supp}(f)\subseteq U$.
\item[(c)] Für jedes $x\in M$ und jede $x$-Umgebung $U\subseteq M$ existiert eine $C^k$-Funktion $f\colon M\to\R$ mit $f(M)\subseteq [0,1]$, $\text{supp}(f)\subseteq U$ und $f\vert_V=1$ für eine $x$-Umgebung $V\subseteq M$.
\end{itemize}
\end{definition}

\begin{satz}
Sei $g\colon N\to M$ eine $C^k$-Funktion zwischen $C^k$-Mannigfaltigkeiten mit rauem Rand, die eine topologische Einbettung ist. Wenn $M$ $C^k$-regulär ist, so auch $N$.
\end{satz}

Beispielsweise ist jede volldimensionale Untermannigfaltigkeit $A\subseteq M$ $C^k$-regulär, wenn $M$ $C^k$-regulär ist.

\begin{proof}
Seien $x\in N$ und $U\subseteq N$ eine $x$-Umgebung. Dann ist $g(U)$ eine $g(x)$-Um"-ge"-bung in $g(N)$, also gibt es eine offene $g(x)$-Umgebung $V\subseteq M$ mit $V\cap g(N)\subseteq g(U)$. Da $M$ $C^k$-regulär ist, existiert eine $C^k$-Funktion $h\colon M\to\R$ mit $h(g(x))\neq 0$ und $\text{supp}(h)\subseteq V$. Somit ist $f\coloneqq h\circ g\colon N\to\R$ eine $C^k$-Funktion mit $f(x)\neq 0$ und $\text{supp}(f)\subseteq U$.
\end{proof}

Welche Mannigfaltigkeiten sind $C^k$-regulär oder $C^k$-parakompakt?

\begin{definition}
Ein topologischer Raum $X$ heißt \ee{regulär}, wenn er Hausdorffsch ist und für jedes $x\in X$ jede $x$-Umgebung $U\subseteq X$ eine abgeschlossene $x$-Umgebung $A\subseteq X$ enthält.
\end{definition}

Sei $M$ eine auf $E$ modellierte $C^k$-Man"-nig"-fal"-tig"-keit mit rauem Rand. Wenn $M$ ein regulärer topologischer Raum und $E$ $C^k$-regulär ist, so ist $M$ $C^k$-regulär (s. \cite[Proposition 3.5.31]{Gl}).

Bekanntlich ist $\R^n$ $C^k$-regulär. Somit ist eine endlichdimensionale $C^k$-Mannig"-fal"-tig"-keit mit rauem Rand genau dann $C^k$-regulär, wenn sie regulär ist. Insbesondere ist jede lokalkompakte $C^k$-Mannig"-fal"-tig"-keit mit rauem Rand endlichdimensional und regulär, also $C^k$-regulär (vgl. ebd.).

\begin{bemerkung}
Sei $M$ eine endlichdimensionale $C^k$-Mannigfaltigkeit mit rauem Rand, welche eine der beiden folgenden Eigenschaften hat:
\begin{itemize}
\item[(a)] $M$ hat eine abzählbare Basis der Topologie und ist als topologischer Raum regulär.
\item[(b)] $M$ ist metrisierbar und separabel.
\end{itemize}
Dann ist $M$ $C^k$-parakompakt.
\end{bemerkung}

\begin{proof}
\begin{itemize}
\item[(a)] $M$ hat die Lindelöf-Eigenschaft\ftnote{D.h. jede offene Überdeckung von $M$ besitzt eine abzählbare Teilüberdeckung.} und ist $C^k$-regulär. Nach \cite[Theorem 16.10]{Mi} ist $M$ also $C^k$-parakompakt.
\item[(b)] Nach dem Metrisationssatz von Urysohn sind (a) und (b) äquivalent.
\end{itemize}
\end{proof}

Die nächsten Aussagen über $C^k$-reguläre Mannigfaltigkeiten werden dem Beweis des Fortsetzungssatzes dienlich sein.

\begin{lemma}\label{aaa}
Es seien $M$ eine $C^k$-reguläre $C^k$-Mannigfaltigkeit mit rauem Rand, $A\subseteq M$ eine kompakte Teilmenge und $U\subseteq M$ eine offene Teilmenge mit $A\subseteq U$. Dann existiert eine $C^k$-Funktion $f\colon M\to\R$ mit $f(M)\subseteq[0,1]$ und $\text{\ee{supp}}(f)\subseteq U$, so dass $f\vert_W=1$ für eine offene Teilmenge $W\subseteq M$ mit $A\subseteq W$ gilt.
\end{lemma}

\begin{proof}
Für jedes $x\in A$ gibt es eine $C^k$-Funktion $g_x\colon M\to\R$ mit $g_x\geq 0$, $g_x(x)>0$ und $\text{supp}(g_x)\subseteq U$. Wir finden $x_1,\ldots,x_m\in A$ derart, dass $A\subseteq g_{x_1}^{-1}(]0,\infty[)\cup\ldots\cup g_{x_m}^{-1}(]0,\infty[)$ gilt, also $g_{x_1}(y)+\ldots+g_{x_m}(y)>0$ für alle $y\in A$. Sei $a>0$ das Minimum der $C^k$-Funktion  $g\coloneqq g_{x_1}+\ldots+g_{x_m}\colon M\to\R$ auf $A$. Dann ist
${W\coloneqq\{y\in M\colon g(y)>\frac{a}{2}\}}$
eine offene Teilmenge von $M$ mit $A\subseteq W$. Wir wählen eine $C^k$-Funktion $h\colon\R\to\R$ mit $h(\R)\subseteq[0,1]$, $\text{supp}(h)\subseteq]0,\infty[$ und $h(y)=1$ für alle $y\geq\frac{a}{2}$. Damit erhalten wir eine $C^k$-Funktion $f\coloneqq h\circ g\colon M\to\R$, für die $f(M)\subseteq[0,1]$, $\text{supp}(f)\subseteq U$ und $f\vert_W=1$ gilt.
\end{proof}

\begin{lemma}\label{bbb}
Es seien $M$ eine $C^k$-reguläre $C^k$-Mannigfaltigkeit mit rauem Rand und $A\subseteq M$ eine kompakte Teilmenge. Ist $(U_j)_{j\in J}$ eine Familie von offenen Teilmengen $U_j\subseteq M$ mit $A\subseteq\bigcup_{j\in J}U_j$, dann existieren $C^k$-Funktionen $h_1,\ldots,h_m\colon M\to\R$ mit den folgenden Eigenschaften:
\begin{itemize}
\item[(a)] $h_i(M)\subseteq[0,1]$ für alle $i\in\{1,\ldots,m\}$;
\item[(b)] $(h_1+\ldots+h_m)\vert_W=1$ für eine offene Teilmenge $W\subseteq M$ mit $A\subseteq W$;
\item[(c)] $h_1+\ldots+h_m\leq 1$;
\item[(d)]Für jedes $i\in\{1,\ldots,m\}$ existiert ein $j(i)\in J$ mit $\text{\ee{supp}}(h_i)\subseteq U_{j(i)}.$
\end{itemize}
\end{lemma}

\begin{proof}
Für jedes $x\in A$ gibt es ein $j(x)\in J$ und eine $C^k$-Funktion $g_x\colon M\to\R$ mit $g_x\geq 0$, $g_x(x)>0$ und $\text{supp}(g_x)\subseteq U_{j(x)}$. Es existieren $x_1,\ldots,x_m\in A$ mit $g_{x_1}(y)+\ldots+g_{x_m}(y)>0$ für alle $y\in A$. Setzen wir $g\coloneqq g_{x_1}+\ldots+g_{x_m}\colon M\to\R$, so ist $Q\coloneqq \{y\in M\colon g(y)>0\}$ eine offene Teilmenge von $M$ mit $A\subseteq Q$. Da $M$ regulär und $A$ kompakt ist, gibt es eine offene Teilmenge $P\subseteq M$ mit $A\subseteq P$ und $\overline{P}\subseteq Q$. Lemma \ref{aaa} liefert uns eine $C^k$-Funktion $f\colon M\to\R$ mit $f(M)\subseteq[0,1]$ und $\text{supp}(f)\subseteq P$, so dass $f\vert_W=1$ für eine offene Teilmenge $W\subseteq M$ mit $A\subseteq W$ gilt. Für die $C^k$-Funktionen
\[h_i\colon M\to\R,\quad h_i(y)\coloneqq\begin{cases} f(y)\frac{g_{x_i}(y)}{g(y)} & \text{wenn }y\in Q;\\ 0 & \text{wenn } y\in M\setminus\text{supp}(f)\end{cases}\]
mit $i\in\{1,\ldots,m\}$ gilt $(h_1+\ldots+h_m)\vert_W=1$. Auch haben sie die übrigen gewünschten Eigenschaften.\end{proof}

\subsection*{Die kompakt-offene $C^k$-Topologie auf $C^k(M,E)$}
\addcontentsline{toc}{subsection}{\protect\numberline{}Die kompakt-offene $C^k$-Topologie auf $C^k(M,E)$}

Von nun an seien stets $k\in\N_0\cup\{\infty\}$ und $(M,\A)$ eine $n$-dimensionale $C^k$-Man"-nig"-fal"-tig"-keit mit rauem Rand, wenn nichts anderes gesagt wird.

Wir diskutieren die Eigenschaften der kompakt-offenen $C^k$-Topologie auf dem reellen Vektorraum $C^k(M,E)$ der $C^k$-Funktionen $f\colon M\to E$.

\begin{definition}[vgl. {\cite[Definition 4.1.2]{Gl}}]\label{jj}
Für alle $l\in\N_0$ mit $l\leq k$, stetigen Halbnormen $q$ auf $E$, Karten $\varphi\colon U_{\varphi}\to V_{\varphi}\subseteq\R^n$ von $M$ und kompakten Teilmengen $K\subseteq U_{\varphi}$ sei die Halbnorm $\lVert\cdot\rVert_{C^l,q,\varphi,K}\colon C^k(M,E)\to [0,\infty[$,
\begin{align*}
\lVert f \rVert_{C^l,q,\varphi, K}
\coloneqq\lVert f\circ\varphi^{-1}\rVert_{C^l,q,\varphi(K)}
\end{align*}
definiert. Wir versehen $C^k(M,E)$ mit der \ee{kompakt-offenen $C^k$-Topologie}; das ist die Hausdorffsche lokalkonvexe Topologie, welche durch die endlichen Summen solcher Halbnormen erzeugt wird.
\end{definition}

\begin{lemma}\label{u}
Sei $\B\subseteq \A$ ein Atlas von $M$. Dann stimmt die Topologie, welche durch die endlichen Summen der Halbnormen $\lVert\cdot\rVert_{C^l,q,\varphi,K}$ mit $l\in\N_0$ und $l\leq k$, stetigen Halbnormen $q$ auf $E$, Karten $\varphi\in\B$ und kompakten Teilmengen $K\subseteq U_{\varphi}$ erzeugt wird, mit der kompakt-offenen $C^k$-Topologie auf $C^k(M,E)$ überein.
\end{lemma}

\begin{proof}
Wir bezeichnen mit $\mathcal{O}_{\B}$ die erstgenannte Topologie auf $C^k(M,E)$ und mit $\mathcal{O}_{\text{co}}$ die kompakt-offene $C^k$-Topologie. Offensichtlich gilt $\mathcal{O}_{\B}\subseteq \mathcal{O}_{\text{co}}$.
Seien $l\in\N_0$ mit $l\leq k$, $q$ eine stetige Halbnorm auf $E$, $\varphi\colon U_{\varphi}\to V_{\varphi}\subseteq\R^n$ eine Karte von $M$ und $K\subseteq U_{\varphi}$ kompakt. Es gibt kompakte Teilmengen $K_1,\ldots, K_m\subseteq M$ mit $K=K_1\cup\ldots\cup K_m$, so dass für jedes $j\in\{1,\ldots,m\}$ eine Karte $\varphi_j\colon U_j\to V_j$ in $\B$ mit $K_j\subseteq U_{j}$ existiert.\ftnote{Als kompakte Menge ist $K$ lokalkompakt. Wählt man zu $x\in K$ eine Karte $\varphi_x\colon U_x\to V_x$ aus $\B$ um $x$, so existiert eine kompakte Umgebung $K_x$ von $x$ in $K$ mit $K_x\subseteq K\cap U_x$. Es gibt endlich viele $x_1,\ldots,x_m\in K$, so dass $K$ von den Inneren der Mengen $K_{x_1},\ldots,K_{x_m}$ bezüglich $K$ überdeckt wird. 
Dann ist $K_j\coloneqq K_{x_j}\subseteq U_{x_j}=:U_j$ für die Karte $\varphi_j\coloneqq \varphi_{x_j}$ mit $V_j\coloneqq V_{x_j}$.}
Nach Lemma \ref{pp} gibt es für alle $j\in\{1,\ldots,m\}$ ein $C_j>0$ derart, dass
\[\big\Vert f_j\circ \big(\varphi_j\circ \varphi^{-1}\vert_{\varphi(U_j\cap U_{\varphi})}\big)\big\Vert_{C^l,q,\varphi(K_j)}
\leq C_j\Vert f_j\Vert_{C^l,q,\varphi_j(K_j)}\]
für alle $f_j\in C^k(V_j,E).$ Daraus folgt
\begin{align*} \Vert f\Vert_{C^l,q,\varphi,K}
&=\Vert f\circ \varphi^{-1}\Vert_{C^l,q,\varphi(K)}
\leq \sum_{j=1}^m\Vert f\circ \varphi^{-1}\Vert_{C^l,q,\varphi(K_j)}\\
&=\sum_{j=1}^m \big\Vert (f\circ\varphi_j^{-1})\circ(\varphi_j\circ\varphi^{-1}\vert_{\varphi(U_j\cap U_{\varphi})}\big)\big\Vert_{C^l,q,\varphi(K_j)}\\ 
&\leq \sum_{j=1}^m C_j\Vert f\circ\varphi_j^{-1}\Vert_{C^l,q,\varphi_j(K_j)}
=\sum_{j=1}^m C_j\Vert f\Vert_{C^l,q,\varphi_j,K_j}
\end{align*}
für alle $f\in C^k(M,E)$.
Somit gilt $\mathcal{O}_{\text{co}}\subseteq\mathcal{O}_{\B}$, also Gleichheit.
\end{proof}

Sei $U\subseteq\R^n$ eine lokalkonvexe Teilmenge mit dichtem Inneren. Mit Lemma \ref{u} (für $U$ statt $M$ und $\{\text{id}_U\}$ als $\B$) folgt, dass die Definitionen \ref{r} und \ref{jj} der kompakt-offenen $C^k$-Topologie auf $C^k(U,E)$ konsistent sind.

\begin{lemma}\label{kk}
Für jede volldimensionale Untermannigfaltigkeit $A\subseteq M$ ist die Einschränkung
\[C^k(M,E)\to C^k(A,E),\quad f\mapsto f\vert_A\]
stetig.
\end{lemma}

\begin{proof}
Seien $l\in\N_0$ mit $l\leq k$, $q$ eine stetige Halbnorm auf $E$, $\varphi\colon U_{\varphi}\to V_{\varphi}\subseteq\R^n$ eine an $A$ angepasste Karte von $M$ und $K\subseteq A\cap U_{\varphi}$ kompakt. Für alle $f\in C^k(M,E)$ gilt
\[\Vert f\vert_A\Vert_{C^l,q,\varphi\vert_{A\cap U_{\varphi}},K}=\Vert f\Vert_{C^l,q,\varphi,K},\]
was zusammen mit Lemma \ref{u} (für $A$ statt $M$) die Stetigkeit beweist.
\end{proof}

Eine Teilmenge $W\subseteq M$ gegeben, schreiben wir $C^k_W(M,E)$ für den Untervektorraum aller $f\in C^k(M,E)$ mit $\text{supp}(f)\subseteq W$.

\begin{lemma}\label{y}
\begin{itemize}
\item[(a)] Sei $\varphi\colon U\to V\subseteq\R^n$ eine Karte von $M$. Dann ist die Abbildung
\[\varphi^*\colon C^k(V,E)\to C^k(U,E),\quad f\mapsto f\circ\varphi\]
ein Isomorphismus von topologischen Vektorräumen mit $(\varphi^*)^{-1}=(\varphi^{-1})^*$.
\item[(b)] Für jedes $g\in C^k(M,\R)$ ist die lineare Abbildung
\[\lambda_g\colon C^k(M,E)\to C^k(M,E),\quad f\mapsto g\cdot f\]
stetig.
\item[(c)] Sei $U\subseteq M$ eine offene Teilmenge. Dann ist für jede Teilmenge $W\subseteq M$ mit $\overline{W}\subseteq U$ die lineare Abbildung
\[C^k_W(U,E)\to C^k_W(M,E),\quad f\mapsto \tilde{f}\]
mit \[\tilde{f}(x)\coloneqq\begin{cases} f(x) &\text{wenn }x\in U; \\ 0 &\text{wenn }x\in M\setminus\overline{W}\end{cases}\]
stetig.
\item[(d)] Sei $(W_j)_{j\in J}$ eine lokalendliche Familie von Teilmengen $W_j\subseteq M$. Dann ist die lineare Abbildung
\begin{align*}\prod_{j\in J}C^k_{W_j}(M,E)\to C^k(M,E),\quad (f_j)_{j\in J}\mapsto \sum_{j\in J}f_j\end{align*}
mit $\big(\sum_{j\in J}f_j\big)(x)\coloneqq\sum_{j\in J}f_j(x)$ für $x\in M$ stetig.
\end{itemize}
\end{lemma}

\begin{proof}
\begin{itemize}
\item[(a)] Ist offensichtlich.
\item[(b)] 
Die Topologie auf $C^k(M,E)$ ist initial bezüglich den Abbildungen
\[C^k(M,E)\to C^k(V_{\varphi},E),\quad f\mapsto f\circ\varphi^{-1},\]
wobei $\varphi\colon U_{\varphi}\to V_{\varphi}\subseteq\R^n$ die Karten von $M$ durchlaufe.
Wir setzen $h_{\varphi}\coloneqq g\circ\varphi^{-1}$
und betrachten das kommutative Diagramm
\[\begin{xy}\xymatrix{
C^k(M,E)\ar[r]^{\lambda_g}\ar[d] & C^k(M,E)\ar[d]\\
C^k(V_{\varphi},E)\ar[r]^{\lambda_{h_{\varphi}}}& C^k(V_{\varphi},E)
}\end{xy}.\]
Können wir zeigen, dass
\[\lambda_{h_{\varphi}}\colon C^k(V_{\varphi},E)\to C^k(V_{\varphi},E),\quad f\mapsto h_{\varphi}\cdot f\]
stetig ist, so folgt die Stetigkeit von $\lambda_g.$
Seien dazu $l\in\N_0$ mit $l\leq k$, $q$ eine stetige Halbnorm auf $E$, $K\subseteq V_{\varphi}$ eine kompakte Teilmenge und $f\in C^k(V_{\varphi},E)$. Dann gilt
\begin{align*}q((\partial^{\alpha}(h_{\varphi}\cdot f))(x))
&\leq\sum_{\beta\leq\alpha}\binom{\alpha}{\beta} \vert\partial^{\beta}h_{\varphi}(x)\vert q(\partial^{\alpha-\beta}f(x))\\
&\leq ((l+1)!)^n\Vert h_{\varphi}\Vert_{C^l,\vert\cdot\vert,K}\Vert f\Vert_{C^l,q,K}\end{align*}
für alle $\al\leq l$ und $x\in K$ und somit
\[\Vert h_{\varphi}\cdot f\Vert_{C^l,q,K}\leq ((l+1)!)^n \Vert h_{\varphi}\Vert_{C^l,\vert\cdot\vert,K}\Vert f\Vert_{C^l,q,K}.\]
Also ist $\lambda_{h_{\varphi}}$ stetig.
\item[(c)] Seien $l\in\N_0$ mit $l\leq k$, $q$ eine stetige Halbnorm auf $E$, $\varphi\colon U_{\varphi}\to V_{\varphi}\subseteq\R^n$ eine Karte von $M$ und $K\subseteq U_{\varphi}$ kompakt. Dann ist $\varphi\vert_{U\cap U_{\varphi}}\colon U\cap U_{\varphi}\to\varphi(U\cap U_{\varphi})$ eine Karte von $U$, und es gilt
\[\Vert \tilde{f}\Vert_{C^l,q,\varphi,K}=\Vert f\Vert_{C^l,q,\varphi\vert_{U\cap U_{\varphi}},K\cap\overline{W}}\]
für alle $f\in C^k_W(U,E)$, woraus die Stetigkeit folgt.
\item[(d)] Seien $l\in\N_0$ mit $l\leq k$, $q$ eine stetige Halbnorm auf $E$, $\varphi\colon U_{\varphi}\to V_{\varphi}\subseteq\R^n$ eine Karte von $M$ und $K\subseteq U_{\varphi}$ kompakt. Wegen der Lokalendlichkeit von $(W_j)_{j\in J}$ ist $J'\coloneqq \{j\in J\colon W_j\cap K\neq\emptyset\}$ eine endliche Menge. Für alle $(f_j)_{j\in J}\in \prod_{j\in J}C^k_{W_j}(M,E)$ gilt
\[\label{gl.jj}
\bigg\Vert\sum_{j\in J}f_j\bigg\Vert_{C^l,q,\varphi, K}
\leq\sum_{j\in J'}\Vert f_j\Vert_{C^l,q,\varphi, K}
\to 0\ \text{für }(f_j)_{j\in J}\to 0,\]
was die Stetigkeit beweist.
\end{itemize}
\end{proof}

\subsection*{Whitneysche $k$-Jets auf Teilmengen von Mannigfaltigkeiten}
\addcontentsline{toc}{subsection}{\protect\numberline{}Whitneysche $k$-Jets auf Teilmengen von Mannigfaltigkeiten}

Es seien $\varphi\colon U_{\varphi}\to V_{\varphi}\subseteq\R^n$ und $\psi\colon U_{\psi}\to V_{\psi}\subseteq\R^n$ zwei Karten von $M$. Nach Satz \ref{uu} induziert der $C^k$-Diffeomorphismus
\[\psi\circ\varphi^{-1}
\colon\varphi(U_{\varphi}\cap U_{\psi})\to \psi(U_{\varphi}\cap U_{\psi})\]
für jede Teilmenge $A\subseteq M$ einen Isomorphismus von topologischen Vektorräumen
\[(\psi\circ\varphi^{-1})^*\colon\EE^k(\psi(A\cap U_{\varphi}\cap U_{\psi}),E)\to \EE^k(\varphi(A\cap U_{\varphi}\cap U_{\psi}),E).\]

\begin{definition}\label{eee}
Wir sagen, dass zwei Whitneysche $k$-Jets $f_{\varphi}\in\EE^k(\varphi(A\cap U_{\varphi}),E)$ und $f_{\psi}\in\EE^k(\psi(A\cap U_{\psi}),E)$
miteinander \ee{korrespondieren} und schreiben $f_{\varphi}\sim f_{\psi}$, wenn
\[(\psi\circ\varphi^{-1})^*\big(f_{\psi}\vert_{\psi(A\cap U_{\varphi}\cap U_{\psi})}\big)=f_{\varphi}\vert_{\varphi(A\cap U_{\varphi}\cap U_{\psi})}.\]
\end{definition}

\begin{lemma}\label{vv}
\begin{itemize}
\item[(a)] Für alle $\varphi\in\A$ und $f_{\varphi},g_{\varphi}\in\EE^k(\varphi(A\cap U_{\varphi}),E)$ gilt $f_{\varphi}\sim g_{\varphi}$ genau dann, wenn $f_{\varphi}=g_{\varphi}$.
\item[(b)] Aus $f_{\varphi}\sim f_{\psi}$ folgt $f_{\psi}\sim f_{\varphi}$ für alle $\varphi,\psi\in\A$ und $f_{\varphi}\in\EE^k(\varphi(A\cap U_{\varphi}),E)$, $f_{\psi}\in\EE^k(\psi(A\cap U_{\psi}),E)$.
\item[(c)] Für alle $\psi_1,\psi_2\in\A$ und $f_{\psi_1}\in\EE^k(\psi_1(A\cap U_{\psi_1}),E)$, $f_{\psi_2}\in\EE^k(\psi_2(A\cap U_{\psi_2}),E)$ mit $f_{\psi_1}\sim f_{\psi_2}$ gilt
\begin{align*}&(\psi_1\circ\varphi^{-1})^*\big(f_{\psi_1}\vert_{\psi_1(A\cap U_{\varphi}\cap U_{\psi_1}\cap U_{\psi_2})}\big)
=(\psi_2\circ\varphi^{-1})^*\big(f_{\psi_2}\vert_{\psi_2(A\cap U_{\varphi}\cap U_{\psi_1}\cap U_{\psi_2})}\big)
\end{align*}
für jedes $\varphi\in\A$.
\item[(d)] Seien $\varphi_1,\varphi_2\in\A$ und $f_{\varphi_1}\in\EE^k(\varphi_1(A\cap U_{\varphi_1}),E), f_{\varphi_2}\in\EE^k(\varphi_2(A\cap U_{\varphi_2}),E)$. Falls für jedes $x\in A$ eine Karte $\psi\in\A$ um $x$ und ein $f_{\psi}\in\EE^k(\psi(A\cap U_{\psi}),E)$ mit $f_{\varphi_1}\sim f_{\psi}$ und $f_{\varphi_2}\sim f_{\psi}$ existiert, dann gilt $f_{\varphi_1}\sim f_{\varphi_2}$.
\end{itemize}
\end{lemma}

\begin{proof}
\begin{itemize}
\item[(a)] Es gilt $f_{\varphi}\sim g_{\varphi}$ genau dann, wenn
\[f_{\varphi}=(\varphi\circ\varphi^{-1})^*(g_{\varphi})=(\text{id}_{V_{\varphi}})^*(g_{\varphi})=g_{\varphi}.\]
\item[(b)] Aus 
\[(\psi\circ\varphi^{-1})^*\big(f_{\psi}\vert_{\psi(A\cap U_{\varphi}\cap U_{\psi})}\big)=f_{\varphi}\vert_{\varphi(A\cap U_{\varphi}\cap U_{\psi})}\]
folgt
\begin{align*}
f_{\psi}\vert_{\psi(A\cap U_{\varphi}\cap U_{\psi})}
&=((\psi\circ\varphi^{-1})^*)^{-1}\big(f_{\varphi}\vert_{\varphi(A\cap U_{\varphi}\cap U_{\psi})}\big)
=(\varphi\circ\psi^{-1})^*\big(f_{\varphi}\vert_{\varphi(A\cap U_{\varphi}\cap U_{\psi})}\big).
\end{align*}
\item[(c)] Mit Satz \ref{uu} (c) erhalten wir
\begin{align*}
&(\psi_1\circ\varphi^{-1})^*\big(f_{\psi_1}\vert_{\psi_1(A\cap U_{\varphi}\cap U_{\psi_1}\cap U_{\psi_2})}\big)\\
&=(\psi_1\circ\psi_2^{-1}\circ\psi_2\circ\varphi^{-1})^*\big(f_{\psi_1}\vert_{\psi_1(A\cap U_{\varphi}\cap U_{\psi_1}\cap U_{\psi_2})}\big)\\
&=(\psi_2\circ\varphi^{-1})^*\big((\psi_1\circ\psi_2^{-1})^*\big(f_{\psi_1}\vert_{\psi_1(A\cap U_{\varphi}\cap U_{\psi_1}\cap U_{\psi_2})}\big)\big)\\
&=(\psi_2\circ\varphi^{-1})^*\big(f_{\psi_2}\vert_{\psi_2(A\cap U_{\varphi}\cap U_{\psi_1}\cap U_{\psi_2})}\big).
\end{align*}
\item[(d)] Für jedes $\psi\in\A$, für welches ein $f_{\psi}\in\EE^k(\psi(A\cap U_{\psi}),E)$ mit $f_{\varphi_1}\sim f_{\psi}$ und $f_{\varphi_2}\sim f_{\psi}$ existiert, gilt
\begin{align*}
&\big((\varphi_2\circ\varphi_1^{-1})^*\big(f_{\varphi_2}\vert_{\varphi_2(A\cap U_{\varphi_1}\cap U_{\varphi_2})}\big)\big)\vert_{\varphi_1(A\cap U_{\varphi_1}\cap U_{\varphi_2}\cap U_{\psi})}\\
&=(\varphi_2\circ\varphi_1^{-1})^*\big(f_{\varphi_2}\vert_{\varphi_2(A\cap U_{\varphi_1}\cap U_{\varphi_2}\cap U_{\psi})}\big)\\
&=(\psi\circ\varphi_1^{-1})^*\big(f_{\psi}\vert_{\psi(A\cap U_{\varphi_1}\cap U_{\varphi_2}\cap U_{\psi})}\big)\\ \nopagebreak
&=f_{\varphi_1}\vert_{\varphi_1(A\cap U_{\varphi_1}\cap U_{\varphi_2}\cap U_{\psi})}.
\end{align*}
Da $\varphi_1(A\cap U_{\varphi_1}\cap U_{\varphi_2})$ 
von solchen Mengen $\varphi_1(A\cap U_{\varphi_1}\cap U_{\varphi_2}\cap U_{\psi})$ überdeckt wird, folgt
\[\big(\varphi_2\circ\varphi_1^{-1})^*\big(f_{\varphi_2}\vert_{\varphi_2(A\cap U_{\varphi_1}\cap U_{\varphi_2})}\big)
=f_{\varphi_1}\vert_{\varphi_1(A\cap U_{\varphi_1}\cap U_{\varphi_2})}.\]
\end{itemize}
\end{proof}

\begin{definition}\label{cc}
Sei $A\subseteq M$ eine Teilmenge.
\begin{itemize}
\item[(a)] Ist $\B\subseteq\A$ eine Menge von Karten für $M$ mit $A\subseteq\bigcup_{\varphi\in\B}U_{\varphi}$, so schreiben wir $\EE^k_{\B}(A,E)$ für den topologischen Untervektorraum aller Familien
\[(f_{\varphi})_{\varphi\in \B}\in\prod_{\varphi\in\B} \EE^k(\varphi(A\cap U_{\varphi}),E)\]
mit $f_{\varphi}\sim f_{\psi}$ für alle $\varphi,\psi\in\B$.
\item[(b)] Die Elemente von
\[\EE_M^k(A,E)\coloneqq \EE^k_{\A}(A,E)\]
nennen wir \ee{Whitneysche $k$-Jets} auf $A$ mit Werten in $E$.
\end{itemize}
\end{definition}

Später werden wir sehen, dass jedes $\EE_{\B}^k(A,E)$ zu $\EE_M^k(A,E)$ isomorph ist.

\begin{lemma}\label{qq}
Für jedes $(f_{\varphi,\alpha})_{\varphi\in\B,\al\leq k}\coloneqq((f_{\varphi,\alpha})_{\al\leq k})_{\varphi\in\B}\in \EE^k_{\B}(A,E)$ gilt
\[f_{\varphi,0}(\varphi(x))=f_{\psi,0}(\psi(x))\]
für alle $x\in A$ und Karten $\varphi,\psi\in\B$ um $x$. Wir erhalten eine Funktion
\[f_0\colon A\to E,\quad f_0\vert_{A\cap U_{\varphi}}\coloneqq f_{\varphi,0}\circ\varphi\vert_{A\cap U_{\varphi}}\ \text{für }\varphi\in\B.\]
\end{lemma}

\begin{proof}
Wenn wir die Komponente 0 auf beiden Seiten der Gleichung
\[f_{\varphi}\vert_{\varphi(A\cap U_{\varphi}\cap U_{\psi})}=(\psi\circ\varphi^{-1})^*\big(f_{\psi}\vert_{\psi(A\cap U_{\varphi}\cap U_{\psi})}\big)\]
an der Stelle $\varphi(x)$ auswerten, erhalten wir
\[f_{\varphi,0}(\varphi(x))=p_{0,0}(\psi(x))f_{\psi,0}(\psi(x))=f_{\psi,0}(\psi(x))\]
mit dem Fa\`a-di-Bruno-Polynom $p_{0,0}=1$.
\end{proof}

\begin{lemma}\label{aa}
\begin{itemize}
\item[(a)] Sei $l\leq k$. Für alle $(f_{\varphi})_{\varphi\in\B}\in\EE_{\B}^k(A,E)$ gilt
$(\text{\ee{pr}}_l(f_{\varphi}))_{\varphi\in\B}\in\EE^l_{\B}(A,E),\ftnote{Wenn wir die $M$ zugrunde liegende $C^l$-Mannigfaltigkeit mit rauem Rand betrachten, können wir von Whitneyschen $l$-Jets auf $A$ sprechen.}$ und die lineare Abbildung
\begin{align}\label{gl.n}\EE_{\B}^k(A,E)\to \EE_{\B}^l(A,E),\quad (f_{\varphi})_{\varphi\in\B}\mapsto (\text{\ee{pr}}_l(f_{\varphi}))_{\varphi\in\B}\end{align}
ist stetig.
\item[(b)] Sei $B\subseteq A$ eine Teilmenge. Für alle $(f_{\varphi})_{\varphi\in\B}\in\EE_{\B}^k(A,E)$ gilt $\big(f_{\varphi}\vert_{\varphi(B\cap U_{\varphi})}\big)_{\varphi\in\mathcal{B}}\in\EE^k_{\mathcal{B}}(B,E)$, und die lineare Abbildung
\begin{align}\label{gl.y}\EE^k_{\B}(A,E)\to \EE^k_{\mathcal{B}}(B,E),\quad (f_{\varphi})_{\varphi\in\B}\mapsto \big(f_{\varphi}\vert_{\varphi(B\cap U_{\varphi})}\big)_{\varphi\in\mathcal{B}}\end{align}
ist stetig.
\end{itemize}
\end{lemma}

\begin{proof}
Sei $(f_{\varphi})_{\varphi\in\B}\in\EE^k_{\B}(A,E)$. Dann gilt $\text{pr}_l(f_{\varphi})\in\EE^l(\varphi(A\cap U_{\varphi}),E)$ bzw. $f_{\varphi}\vert_{\varphi(B\cap U_{\varphi})}$ $\in\EE^k(\varphi(B\cap U_{\varphi}),E)$ für jedes $\varphi\in\B$. Mit Bemerkung \ref{ddd} (a) bzw. (b) können wir folgern, dass $\text{pr}_l(f_{\varphi})\sim \text{pr}_l(f_{\psi})$ bzw. $f_{\varphi}\vert_{\varphi(B\cap U_{\varphi})}\sim f_{\psi}\vert_{\psi(B\cap U_{\psi})}$ für alle $\varphi,\psi\in\B$ gilt und somit $(\text{pr}_l(f_{\varphi}))_{\varphi\in\B}\in\EE^l_{\B}(A,E)$ bzw. $\big(f_{\varphi}\vert_{\varphi(B\cap U_{\varphi})}\big)_{\varphi\in\mathcal{B}}\in\EE^k_{\mathcal{B}}(B,E)$. Im kommutativen Diagramm
\[\begin{xy}\xymatrix{
\EE^k_{\B}(A,E)\ar[r]^{\eqref{gl.n}}\ar@_{(->}[d]\ar@^{->}[d] & \EE^l_{\B}(A,E)\ar@_{(->}[d]\ar@^{->}[d] \\
\prod_{\varphi\in\B}\EE^k(\varphi(A\cap U_{\varphi}),E)\ar[r] & \prod_{\varphi\in\B}\EE^l(\varphi(A\cap U_{\varphi}),E)
}\end{xy}\]
bzw.
\[\begin{xy}\xymatrix{
\EE^k_{\B}(A,E)\ar[r]^{\eqref{gl.y}}\ar@_{(->}[d]\ar@^{->}[d] & \EE^k_{\B}(B,E)\ar@_{(->}[d]\ar@^{->}[d] \\
\prod_{\varphi\in\B}\EE^k(\varphi(A\cap U_{\varphi}),E)\ar[r] & \prod_{\varphi\in\B}\EE^k(\varphi(B\cap U_{\varphi}),E)
}\end{xy}\]
stehen links und rechts topologische Einbettungen; die untere Abbildung ist nach Lemma \ref{ff} (a) bzw. (b) stetig. Demnach ist \eqref{gl.n} bzw. \eqref{gl.y} stetig.
\end{proof}

\begin{satz}\label{ww}
Es seien $A\subseteq M$ eine Teilmenge und $\B\subseteq\A$ eine Menge von Karten für $M$ mit $A\subseteq\bigcup_{\varphi\in\B}U_{\varphi}$. Dann ist die Projektion
\[\ee{pr}_{\B}\colon \EE_M^k(A,E)\to \EE^k_{\B}(A,E),\quad (f_{\varphi})_{\varphi\in \A}\mapsto (f_{\varphi})_{\varphi\in \B}\]
ein Isomorphismus von topologischen Vektorräumen.
\end{satz}

\begin{proof}
Offensichtlich ist $\text{pr}_{\B}$ stetig und linear. Wir kon"-stru"-ie"-ren die Inverse zu $\text{pr}_{\B}$.
Für jedes $\varphi\in\A$ betrachten wir den Isomorphismus
\[\prod_{\psi\in\B}(\psi\circ\varphi^{-1})^*\colon\prod_{\psi\in\B}\EE^k(\psi(A\cap U_{\varphi}\cap U_{\psi}),E)\to\prod_{\psi\in\B}\EE^k(\varphi(A\cap U_{\varphi}\cap U_{\psi}),E).\]
Sei $f=(f_{\psi})_{\psi\in\B}\in\EE^k_{\B}(A\cap U_{\varphi},E)\subseteq \prod_{\psi\in\B}\EE^k(\psi(A\cap U_{\varphi}\cap U_{\psi}),E)$. 
Für alle $\psi_1, \psi_2\in\B$ gilt
\begin{align*}
(\psi_1\circ\varphi^{-1})^*\big(f_{\psi_1}\vert_{\psi_1(A\cap U_{\varphi}\cap U_{\psi_1}\cap U_{\psi_2})}\big)
=(\psi_2\circ\varphi^{-1})^*\big(f_{\psi_2}\vert_{\psi_2(A\cap U_{\varphi}\cap U_{\psi_1}\cap U_{\psi_2})}\big)
\end{align*}
nach Lemma \ref{vv} (c), was äquivalent ist zu
\begin{align*}
\big((\psi_1\circ\varphi^{-1})^*(f_{\psi_1})\big)\vert_{\varphi(A\cap U_{\varphi}\cap U_{\psi_1}\cap U_{\psi_2})}
=\big((\psi_2\circ\varphi^{-1})^*(f_{\psi_2})\big)\vert_{\varphi(A\cap U_{\varphi}\cap U_{\psi_1}\cap U_{\psi_2})}.
\end{align*}
Somit liegt $\big(\prod_{\psi\in\B}(\psi\circ\varphi^{-1})^*\big)(f)$ im topologischen Vektorraum
\[R_{\varphi}\coloneqq\begin{Bmatrix}
&(g_{\psi})_{\psi\in\B}\in\prod_{\psi\in\B}\EE^k(\varphi(A\cap U_{\varphi}\cap U_{\psi}),E)\colon \\
&g_{\psi_1}\vert_{\varphi(A\cap U_{\varphi}\cap U_{\psi_1}\cap U_{\psi_2})}=g_{\psi_2}\vert_{\varphi(A\cap U_{\varphi}\cap U_{\psi_1}\cap U_{\psi_2})}\ \forall\psi_1,\psi_2\in\B\end{Bmatrix}.\]
Wir können daher
$\prod_{\psi\in\B}(\psi\circ\varphi^{-1})^*$
zu einer stetigen linearen Abbildung
\[\alpha_{\varphi}\colon\EE^k_{\B}(A\cap U_{\varphi},E)\to R_{\varphi}\]
einschränken. Nach Satz \ref{tt} besitzt die Abbildung
\[\EE^k(\varphi(A\cap U_{\varphi}),E)\to R_{\varphi},\quad f\mapsto\big(f\vert_{\varphi(A\cap U_{\varphi}\cap U_{\psi})}\big)_{\psi\in\B}\]
eine stetige lineare Inverse
\[\beta_{\varphi}\colon R_{\varphi}\to\EE^k(\varphi(A\cap U_{\varphi}),E).\]
Komponieren wir die stetige lineare Abbildung
\[\EE^k_{\B}(A,E)\to\prod_{\varphi\in\A}\EE^k_{\B}(A\cap U_{\varphi},E),\quad (f_{\psi})_{\psi\in\B}\mapsto \big(\big(f_{\psi}\vert_{\psi(A\cap U_{\varphi}\cap U_{\psi})}\big)_{\psi\in\B}\big)_{\varphi\in\A}\]
mit
\[\prod_{\varphi\in\A}(\beta_{\varphi}\circ\alpha_{\varphi})\colon\prod_{\varphi\in\A}\EE^k_{\B}(A\cap U_{\varphi},E)\to\prod_{\varphi\in\A}\EE^k(\varphi(A\cap U_{\varphi}),E),\]
so erhalten wir eine stetige lineare Abbildung
\[I\colon \EE^k_{\B}(A,E)\to \prod_{\varphi\in\A}\EE^k(\varphi(A\cap U_{\varphi}),E).\]  
Es seien $f=(f_{\psi})_{\psi\in\B}\in\EE^k_{\B}(A,E)$ und $\tilde{f}=(\tilde{f}_{\varphi})_{\varphi\in\A}\coloneqq I(f).$ Dann gilt nach Kon"-struk"-tion
\[\tilde{f}_{\varphi}\vert_{\varphi(A\cap U_{\varphi}\cap U_{\psi})}=(\psi\circ\varphi^{-1})^*\big(f_{\psi}\vert_{\psi(A\cap U_{\varphi}\cap U_{\psi})}\big),\]
also $\tilde{f}_{\varphi}\sim f_{\psi}$
für alle $\varphi\in\A$ und $\psi\in\B$. 
Mit Lemma \ref{vv} (d) folgt $\tilde{f}_{\varphi_1}\sim\tilde{f}_{\varphi_2}$ für alle $\varphi_1,\varphi_2\in\A$ und somit $\tilde{f}\in\EE_M^k(A,E)$. Demnach gilt
\[I(\EE^k_{\B}(A,E))\subseteq\EE_M^k(A,E).\]
Für alle $\psi\in\B$ gilt $\tilde{f}_{\psi}\sim f_{\psi}$, also $\tilde{f}_{\psi}=f_{\psi}$ nach Lemma \ref{vv} (a). Daraus folgt
\[\text{pr}_{\B}\circ I\vert^{\EE^k_M(A,E)}=\text{id}_{\EE^k_{\B}(A,E)}.\]
Angenommen, es ist $f=\text{pr}_{\B}((f_{\varphi})_{\varphi\in\A})$ mit $(f_{\varphi})_{\varphi\in\A}\in\EE^k_M(A,E)$. Für jedes $\varphi\in\A$ gilt $f_{\varphi}\sim f_{\psi}$ und $\tilde{f}_{\varphi}\sim f_{\psi}$ für alle $\psi\in\B$; folglich gilt $\tilde{f}_{\varphi}\sim f_{\varphi}$, also $\tilde{f}_{\varphi}=f_{\varphi}$. Wir erhalten
\[I\vert^{\EE^k_M(A,E)}\circ \text{pr}_{\B}=\text{id}_{\EE^k_M(A,E)},\]
womit der Satz bewiesen ist.
\end{proof}

\begin{korollar}
Es seien $\varphi\colon U\to V\subseteq\R^n$ eine Karte von $M$ und $A\subseteq U$ eine Teilmenge. Dann ist die Projektion
\begin{align*}\EE^k_M(A,E)\to\EE^k(\varphi(A),E),\quad (f_{\psi})_{\psi\in\A}\mapsto f_{\varphi}\end{align*} ein Isomorphismus von topologischen Vektorräumen.
\end{korollar}

Insbesondere gilt für jede Teilmenge $A\subseteq\R^n$ also
\[\EE^k_{\R^n}(A,E)\cong\EE^k(A,E).\]

\begin{proof}
Die obige Projektion ist die Komposition der Isomorphismen
\[\text{pr}_{\{\varphi\}}\colon \EE^k(A,E)\to \EE^k_{\{\varphi\}}(A,E),\quad (f_{\psi})_{\psi\in \A}\mapsto (f_{\varphi})_{\varphi}\]
und
\[\EE^k_{\{\varphi\}}(A,E)\to \EE^k(\varphi(A),E),\quad (f_{\varphi})_{\varphi}\mapsto f_{\varphi}.\]
\end{proof}

\subsection*{Whitneysche $k$-Jets auf Untermannigfaltigkeiten}
\addcontentsline{toc}{subsection}{\protect\numberline{}Whitneysche $k$-Jets auf Untermannigfaltigkeiten}

Wir haben bereits gesehen, dass ein Isomorphismus von topologischen Vektorräumen $C^k(A,E)\cong\EE^k(A,E)$ für jede lokalkonvexe Teilmenge $A\subseteq\R^n$ mit dichtem Inneren existiert. Aber auch für jede volldimensionale Untermannigfaltigkeit $A\subseteq M$ gilt $C^k(A,E)\cong\EE^k_M(A,E)$.

\begin{satz}\label{v}
Es seien $A\subseteq M$ eine volldimensionale Untermannigfaltigkeit und $\B$ die Menge der an $A$ angepassten Karten $\varphi\colon U_{\varphi}\to V_{\varphi}\subseteq\R^n$ von $M$. Dann ist die Abbildung
\begin{align}\label{gl.w}C^k(A,E)\to\EE^k_{\B}(A,E),\quad f\mapsto \left(\p\left(f\circ\varphi^{-1}\vert_{\varphi(A\cap U_{\varphi})}\right)\right)_{\substack{\varphi\in\B\\ \al\leq k}}\end{align}
ein Isomorphismus von topologischen Vektorräumen. Zusammen mit $\EE^k_{\B}(A,E)\cong\EE^k_M(A,E)$ folgt
\[C^k(A,E)\cong\EE^k_M(A,E).\]
\end{satz}

\begin{proof}
Ein $f \in C^k(A,E)$ gegeben, gilt $f\circ\varphi^{-1}\vert_{\varphi(A\cap U_{\varphi})}\in C^k(\varphi(A\cap U_{\varphi}),E)$ für jedes $\varphi\in\B$. Mit Bemerkung \ref{ggg} folgt
\begin{align*} 
\big(\p\big(f\circ\varphi^{-1}\vert_{\varphi(A\cap U_{\varphi}\cap U_{\psi})}\big)\big)_{\al\leq k}
&=\big(\p\big(f\circ \psi^{-1}\circ\psi\circ\varphi^{-1}\vert_{\varphi(A\cap U_{\varphi}\cap U_{\psi})}\big)\big)_{\al\leq k}\\
&=(\psi\circ\varphi^{-1})^*\big(\big(\p\big(f\circ\psi^{-1}\vert_{\psi(A\cap U_{\varphi}\cap U_{\psi})}\big)\big)_{\al\leq k}\big)
\end{align*}
für alle $\varphi,\psi\in\B$,
also
\[\big(\p\big(f\circ\varphi^{-1}\vert_{\varphi(A\cap U_{\varphi})}\big)\big)_{\al\leq k}\sim \big(\p\big(f\circ\psi^{-1}\vert_{\psi(A\cap U_{\psi})}\big)\big)_{\al\leq k},\]
womit gezeigt ist, dass \eqref{gl.w} tatsächlich Werte in $\EE^k_{\B}(A,E)$ hat.
Für jedes $f=(f_{\varphi,\alpha})_{\varphi\in\B, \al\leq k}\in\EE^k_{\B}(A,E)$ sei $f_0\colon A\to E$ die Funktion, die durch 
\[f_0\vert_{A\cap U_{\varphi}}=f_{\varphi,0}\circ\varphi\vert_{A\cap U_{\varphi}}\ \text{für }\varphi\in\B\] 
definiert ist. Nach Satz \ref{gg} ist
$f_0\circ\varphi^{-1}\vert_{\varphi(A\cap U_{\varphi})}=f_{\varphi,0}$
für jedes $\varphi\in\B$ eine $C^k$-Funktion mit
\[\p\big(f_0\circ\varphi^{-1}\vert_{\varphi(A\cap U_{\varphi})}\big)=\p f_{\varphi,0}=f_{\varphi,\alpha}\]
für alle $\al\leq k$. Somit ist $f_0$ eine $C^k$-Funktion und die Abbildung
\begin{align}\label{gl.z}\EE_{\B}^k(A,E)\to C^k(A,E),\quad f\mapsto f_0\end{align}
rechtsinvers zu \eqref{gl.w}. Für jedes $f\in C^k(A,E)$ gilt
\[\Big(\big(\p\big(f\circ\varphi^{-1}\vert_{\varphi(A\cap U_{\varphi})}\big)\big)_{\substack{\varphi\in\B\\ \al\leq k}}\Big)_0\Big\vert_{A\cap U_{\psi}}
=\big(f\circ\psi^{-1}\vert_{\psi(A\cap U_{\psi})}\big)\circ\psi\vert_{A\cap U_{\psi}}=f\vert_{A\cap U_{\psi}}\]
für alle $\psi\in\B$, also
\[\Big(\big(\p\big(f\circ\varphi^{-1}\vert_{\varphi(A\cap U_{\varphi})}\big)\big)_{\substack{\varphi\in\B\\ \al\leq k}}\Big)_0=f.\]
Folglich ist \eqref{gl.z} auch linksinvers zu \eqref{gl.w}.
Nach Satz \ref{gg} ist die Abbildung
\[\prod_{\varphi\in\B}C^k(\varphi(A\cap U_{\varphi}),E)\to \prod_{\varphi\in\B}\EE^k(\varphi(A\cap U_{\varphi}),E),\quad (f_{\varphi})_{\varphi\in\B}\to ((\p f_{\varphi})_{\al\leq k})_{\varphi\in\B}\]
ein Isomorphismus von topologischen Vektorräumen. Mit den Einbettungen von topologischen Vektorräumen
\[C^k(A,E)\hookrightarrow \prod_{\varphi\in\B}C^k(\varphi(A\cap U_{\varphi}),E),\quad f\mapsto\big(f\circ\varphi^{-1}\vert_{\varphi(A\cap U_{\varphi})}\big)_{\varphi\in\B}\]
und
\[\EE^k_{\B}(A,E)\hookrightarrow \prod_{\varphi\in\B}\EE^k(\varphi(A\cap U_{\varphi}),E),\quad f\mapsto f\]
ist das Diagramm
\[\begin{xy}\xymatrix{
C^k(A,E)\ar[r]^{\eqref{gl.w}}\ar@_{(->}[d]\ar@^{->}[d] & \EE^k_{\B}(A,E)\ar@_{(->}[d]\ar@^{->}[d] \\
\prod_{\varphi\in\B}C^k(\varphi(A\cap U_{\varphi}),E)\ar[r]^{\cong} & \prod_{\varphi\in\B}\EE^k(\varphi(A\cap U_{\varphi}),E)
}\end{xy}\]
kommutativ. Somit ist \eqref{gl.w} ein Isomorphismus von topologischen Vektorräumen.
\end{proof}

Wenn $A\subseteq M$ eine Untermannigfaltigkeit von kleinerer Dimension ist, dann lässt sich $C^k(A,E)$ in $\EE^k_M(A,E)$ einbetten.

\begin{definition}[vgl. {\cite[Definition 3.5.14]{Gl}}]
Sei $m\in\{1,\ldots,n-1\}$. Eine Teilmenge $A\subseteq M$ heißt \ee{m-dimensionale Untermannigfaltigkeit}, wenn für jedes $x\in A$ eine Karte $\varphi=(\varphi_1,\ldots,\varphi_n)\colon U_{\varphi}\to V_{\varphi}\subseteq\R^n$ für $M$ um $x$ existiert, so dass
\[\varphi(A\cap U_{\varphi})=V_{\varphi}\cap(\R^m\times\{0\})\]
gilt und
\[W_{\varphi}\coloneqq\{x\in\R^m\colon (x,0)\in V_{\varphi}\}\]
ein dichtes Inneres bezüglich $\R^m$ hat. Man beachte, dass $W_{\varphi}$ auch lokalkonvex ist. Solche Karten $\varphi$ nennt man \ee{an A angepasst}, und die Funktionen
\[\varphi_A\colon A\cap U_{\varphi}\to W_{\varphi},\quad x\mapsto(\varphi_1(x),\ldots,\varphi_m(x))\]
bilden einen $C^k$-Atlas für $A$. Mit dem zugehörigen maximalen $C^k$-Atlas ist $A$ eine $C^k$-Mannigfaltigkeit mit rauem Rand.
\end{definition}

Es seien $\B$ die Menge der an $A$ angepassten Karten für $M$ und $\widetilde{\B}\coloneqq\{\varphi_A\colon \varphi\in\A\}$. Wir konstruieren eine Einbettung von topologischen Vektorräumen
\[\EE^k_{\widetilde{\B}}(A,E)\hookrightarrow\EE^k_{\B}(A,E).\]
Wie schon in Beispiel \ref{fff} gesehen, induzieren Inklusion
\[j\colon\R^m\to\R^n,\quad x\mapsto (x,0)\]
und Projektion
\[\pi\colon \R^n=\R^m\times\R^{n-m}\to\R^m,\quad (x,y)\to x\]
für jedes $\varphi\in\B$ stetige lineare Abbildungen
\[(j\vert_{W_{\varphi}})^*\colon\EE^k(W_{\varphi}\times\{0\},E)\to\EE^k(W_{\varphi},E)\]
bzw.
\[(\pi\vert_{W_{\varphi}\times\{0\}})^*\colon\EE^k(W_{\varphi},E)\to\EE^k(W_{\varphi}\times\{0\},E).\]
Sei $(f_{\varphi})_{\varphi\in\B}\in\EE^k_{\B}(A,E)$. Für alle $\varphi,\psi\in\B$ gilt
\begin{align*}
&\big(\psi_A\circ\varphi_A^{-1}\vert_{\varphi_A(A\cap U_{\varphi}\cap U_{\psi})}\big)^*\big(((j\vert_{W_{\psi}})^*(f_{\psi}))\vert_{\psi_A(A\cap U_{\varphi}\cap U_{\psi})}\big)\\
&=\big(\pi\circ\psi\circ\varphi^{-1}\circ j\vert_{\varphi_A(A\cap U_{\varphi}\cap U_{\psi})}\big)^*\big(\big(j\vert_{\psi_A(A\cap U_{\varphi}\cap U_{\psi})}\big)^*\big(f_{\psi}\vert_{\psi(A\cap U_{\varphi}\cap U_{\psi})}\big)\big)\\
&=\big(j\circ\pi\circ\psi\circ\varphi^{-1}\circ j\vert_{\varphi_A(A\cap U_{\varphi}\cap U_{\psi})}\big)^*\big(f_{\psi}\vert_{\psi(A\cap U_{\varphi}\cap U_{\psi})}\big)\\
&=\big(\psi\circ\varphi^{-1}\circ j\vert_{\varphi_A(A\cap U_{\varphi}\cap U_{\psi})}\big)^*\big(f_{\psi}\vert_{\psi(A\cap U_{\varphi}\cap U_{\psi})}\big)\\
&=\big(j\vert_{\varphi_A(A\cap U_{\varphi}\cap U_{\psi})}\big)^*\big(\big(\psi\circ\varphi^{-1}\vert_{\varphi(A\cap U_{\varphi}\cap U_{\psi})}\big)^*\big(f_{\psi}\vert_{\psi(A\cap U_{\varphi}\cap U_{\psi})}\big)\big)\\
&=\big(j\vert_{\varphi_A(A\cap U_{\varphi}\cap U_{\psi})}\big)^*\big(f_{\varphi}\vert_{\varphi(A\cap U_{\varphi}\cap U_{\psi})}\big)\big)\\
&=\big((j\vert_{W_{\varphi}})^*(f_{\varphi})\big)\vert_{\varphi_A(A\cap U_{\varphi}\cap U_{\psi})},
\end{align*}
also
$(j\vert_{W_{\varphi}})^*(f_{\varphi})\sim(j\vert_{W_{\psi}})^*(f_{\psi}).$
Insbesondere gilt $(j\vert_{W_{\varphi}})^*(f_{\varphi})=(j\vert_{W_{\psi}})^*(f_{\psi})$ für alle $\varphi,\psi\in\B$ mit $\varphi_A=\psi_A$. Somit ist
\begin{align}\label{gl.aa}
P_{\widetilde{\B}}\colon\EE^k_{\B}(A,E)\to\EE^k_{\widetilde{\B}}(A,E),\quad (f_{\varphi})_{\varphi\in\B}\mapsto\big((j\vert_{W_{\varphi}})^*(f_{\varphi})\big)_{\varphi_A\in\widetilde{\B}}
\end{align}
sinnvoll definiert und offensichtlich auch stetig und linear. Auf ähnliche Weise sieht man, dass
\begin{align*}
J_{\B}\colon\EE^k_{\widetilde{\B}}(A,E)\to\EE^k_{\B}(A,E),\quad (f_{\varphi_A})_{\varphi_A\in\widetilde{\B}}\mapsto\big((\pi\vert_{W_{\varphi}\times\{0\}})^*(f_{\varphi_A})\big)_{\varphi\in\B}
\end{align*}
tatsächlich Werte in $\EE^k_{\B}(A,E)$ hat sowie stetig und linear ist. Mit $(j\vert_{W_{\varphi}})^*\circ (\pi\vert_{W_{\varphi}\times\{0\}})^*$ $=\text{id}_{\EE^k(W_{\varphi},E)}$ für alle $\varphi\in\B$
folgt 
\[P_{\widetilde{\B}}\circ J_{\B}=\text{id}_{\EE^k_{\widetilde{\B}}(A,E)},\]
also ist $J_{\B}$ eine Einbettung von topologischen Vektorräumen. Die Sätze \ref{ww} und \ref{v} liefern uns Isomorphismen
\[C^k(A,E)\xrightarrow{\sim}\EE^k_A(A,E)\xrightarrow{\sim}\EE^k_{\widetilde{\B}}(A,E)\hookrightarrow\EE^k_{\B}(A,E)\xrightarrow{\sim}\EE^k_M(A,E),\]
womit wir schließen:

\begin{satz}
Für jede Untermannigfaltigkeit $A\subseteq M$ gibt es eine Einbettung von topologischen Vektorräumen
\[C^k(A,E)\hookrightarrow\EE^k_M(A,E).\]
\end{satz}

\subsection*{Der Fortsetzungssatz}
\addcontentsline{toc}{subsection}{\protect\numberline{}Der Fortsetzungssatz}

\begin{satz}\label{bb}
Es seien $k\in\N_0$, $(M,\A)$ eine endlichdimensionale $C^k$-Man"-nig"-fal"-tig"-keit mit rauem Rand und $A\subseteq M$ eine abgeschlossene Teilmenge. Angenommen, eine der beiden folgenden Bedingungen ist erfüllt:
\begin{itemize}
\item[(a)] $M$ ist $C^k$-parakompakt und $A$ lokalkompakt.
\item[(b)] $M$ ist regulär und $A$ kompakt.
\end{itemize}
Dann besitzt die stetige lineare Abbildung
\[\rho^k_{A,M}\colon C^k(M,E)\to \EE^k_M(A,E),\quad f\mapsto \big((\p(f\circ\varphi^{-1}))\vert_{\varphi(A\cap U_{\varphi})}\big)_{\substack{\varphi\in\A\\ \al\leq k}}\]
eine stetige lineare Rechtsinverse.
\end{satz}

\begin{proof}
(a) Angenommen, $M$ ist $C^k$-parakompakt und $A$ lokalkompakt. Wir wählen eine $C^k$-Partition der Eins $(h_{\varphi})_{\varphi\in\A}$ auf $M$ mit $W_{\varphi}\coloneqq\text{supp}(h_{\varphi})\subseteq U_{\varphi}$ für alle Karten $\varphi\colon U_{\varphi}\to V_{\varphi}\subseteq\R^n$ in $\A$. Die Projektion
\[F_{1,\varphi}\colon\EE^k_M(A,E)\to\EE^k(\varphi(A\cap U_{\varphi}),E),\quad (f_{\psi})_{\psi\in\A}\mapsto f_{\varphi}\]
ist für jedes $\varphi\in\A$ stetig und linear. Da $\varphi(A\cap U_{\varphi})$ eine relativ abgeschlossene, lokalkompakte Teilmenge von $V_{\varphi}$ ist, existiert Satz \ref{k} zufolge ein stetiger linearer Fortsetzungsoperator
\[F_{2,\varphi}\coloneqq\Phi^k_{V_{\varphi},\varphi(A\cap U_{\varphi})}\colon\EE^k(\varphi(A\cap U_{\varphi}),E)\to C^k(V_{\varphi},E).\]
Die linearen Abbildungen
\[F_{3,\varphi}\colon C^k(V_{\varphi},E)\to C^k(U_{\varphi},E),\quad f\mapsto f\circ\varphi\]
und
\[F_{4,\varphi}\colon C^k(U_{\varphi},E)\to C^k_{W_{\varphi}}(U_{\varphi},E),\quad f\mapsto h_{\varphi}\vert_{U_{\varphi}}\cdot f\]
sowie
\begin{align*}F_{5,\varphi}\colon C^k_{W_{\varphi}}(U_{\varphi},E)\to C^k_{W_{\varphi}}(M,E),\quad f\mapsto\left(x\mapsto \begin{cases} f(x) &\text{wenn }x\in U_{\varphi}; \\ 0 &\text{sonst }\end{cases}\right)\end{align*}
sind stetig nach Lemma \ref{y}. Somit ist die Komposition
\[F_{\varphi}\coloneqq  F_{5,\varphi}\circ F_{4,\varphi}\circ F_{3,\varphi}\circ F_{2,\varphi}\circ F_{1,\varphi}\colon \EE^k_M(A,E)\to C^k_{W_{\varphi}}(M,E)\]
stetig und linear.
\begin{samepage}Zusammen mit Lemma \ref{y} (d) folgt, dass die Abbildung
\[\Phi^k_{M,A}\colon \EE^k_M(A,E)\to C^k(M,E),\quad f\mapsto \sum_{\varphi\in\A} F_{\varphi}(f)\]
stetig und linear ist. \end{samepage}

Es seien $f=(f_{\varphi})_{\varphi\in\A}\in\EE^k_M(A,E)$ mit $f_{\varphi}=(f_{\varphi,\alpha})_{\al\leq k}$ für $\varphi\in\A$ und $F\coloneqq\Phi^k_{M,A}(f).$ Weiter seien $\varphi\in\A$, $\al\leq k$ und $x\in A\cap U_{\varphi}$. Es ist zu zeigen, dass
\[(\p(F\circ\varphi^{-1}))(\varphi(x))=f_{\varphi,\alpha}(\varphi(x)).\]
Somit gilt $\rho^k_{A,M}(F)=f$, also ist
$\Phi^k_{M,A}$ rechtsinvers zu $\rho^k_{A,M}$.

Für jedes $\psi\in\A$ schreiben wir $\tilde{f_{\psi}}\coloneqq \Phi^k_{V_{\psi},\psi(A\cap U_{\psi})}(f_{\psi}).$ Dann gilt \[\partial^{\beta}\tilde{f_{\psi}}(y)=f_{\psi,\beta}(y)\] für alle $\bet\leq k$ und $y\in\psi(A\cap U_{\psi}).$ Zudem gilt
\[F(z)=\sum_{\substack{\psi\in\A\colon\\ z\in U_{\psi}}}h_{\psi}(z)(\tilde{f}_{\psi}\circ\psi)(z)\]
für alle $z\in M$, also insbesondere
\[(F\circ\varphi^{-1})(y)=\sum_{\substack{\psi\in\A\colon\\ \varphi^{-1}(y)\in U_{\psi}}}(h_{\psi}\circ\varphi^{-1})(y)\big(\tilde{f}_{\psi}\circ\psi\circ\varphi^{-1}\vert_{\varphi(U_{\varphi}\cap U_{\psi})}\big)(y)\]
für alle $y\in V_{\varphi}$. 
Mit Bemerkung \ref{ggg} gilt
\begin{align*}\big(\big(\partial^{\beta}\big(\tilde{f}_{\psi}\circ\psi\circ\varphi^{-1}\vert_{\varphi(U_{\varphi}\cap U_{\psi})}\big)\big)\vert_{\varphi(A\cap U_{\varphi}\cap U_{\psi})}\big)_{\bet\leq k}
&=(\psi\circ\varphi^{-1})^*\big(f_{\psi}\vert_{\psi(A\cap U_{\varphi}\cap U_{\psi})}\big)\\
&=f_{\varphi}\vert_{\varphi(A\cap U_{\varphi}\cap U_{\psi})}
\end{align*}
für jedes $\psi\in\A$ und somit
\begin{align*}
\big(\partial^{\beta}\big(\tilde{f}_{\psi}\circ\psi\circ\varphi^{-1}\vert_{\varphi(U_{\varphi}\cap U_{\psi})}\big)\big)(y)=f_{\varphi,\beta}(y)=\partial^{\beta}\tilde{f_{\varphi}}(y)
\end{align*}
für alle $\bet\leq k$ und $y\in\varphi(A\cap U_{\varphi}\cap U_{\psi})$.
Mit $\sum_{\psi\in\A}h_{\psi}=1$ gilt
\[\tilde{f_{\varphi}}(y)=\sum_{\psi\in\A}(h_{\psi}\circ\varphi^{-1})(y)\tilde{f_{\varphi}}(y)\]
für alle $y\in V_{\varphi}$, also
\[\p\tilde{f_{\varphi}}(y)=\sum_{\psi\in\A}\sum_{\beta\leq\alpha}\binom{\alpha}{\beta}(\partial^{\beta}(h_{\psi}\circ\varphi^{-1}))(y)\partial^{\alpha-\beta}\tilde{f_{\varphi}}(y).\]
Insgesamt folgt
\begin{samepage}
\begin{align*}
&(\p(F\circ\varphi^{-1}))(\varphi(x))\\
&=\sum_{\substack{\psi\in\A\colon\\ x\in U_{\psi}}}\sum_{\beta\leq\alpha}\binom{\alpha}{\beta}(\partial^{\beta}(h_{\psi}\circ\varphi^{-1}))(\varphi(x))\big(\partial^{\alpha-\beta}\big(\tilde{f_{\psi}}\circ\psi\circ\varphi^{-1}\vert_{\varphi(U_{\varphi}\cap U_{\psi})}\big)\big)(\varphi(x))\\
&=\sum_{\psi\in\A}\sum_{\beta\leq\alpha}\binom{\alpha}{\beta}(\partial^{\beta}(h_{\psi}\circ\varphi^{-1}))(\varphi(x))\partial^{\alpha-\beta}\tilde{f_{\varphi}}(\varphi(x))
=\p\tilde{f_{\varphi}}(\varphi(x))=f_{\varphi,\alpha}(\varphi(x)),
\end{align*}
was zu zeigen war.\end{samepage}

(b) Nun behandeln wir den Fall, dass $M$ regulär und $A$ kompakt ist. Nach Lemma \ref{bbb} existieren $C^k$-Funktionen $h_1,\ldots,h_m\colon M\to\R$ mit Werten in $[0,1]$, so dass $(h_1+\ldots+h_m)\vert_W=1$ für eine offene Teilmenge $W\subseteq M$ mit $A\subseteq W$ gilt und für jedes $i\in\{1,\ldots,m\}$ eine Karte $\varphi_i\in\A$ mit $W_i\coloneqq\text{supp}(h_i)\subseteq U_{\varphi_i}$ existiert. Für alle $i\in\{1,\ldots,m\}$ definieren wir analog wie in Beweisteil (a) die Abbildungen $F_{1,\varphi_i},\ldots,F_{5,\varphi_i}$ und erhalten stetige lineare Abbildungen
\[F_{\varphi_i}\coloneqq F_{5,\varphi_i}\circ\ldots\circ F_{1,\varphi_i}\colon\EE^k_M(A,E)\to C^k_{W_i}(M,E),\]
deren Summe
\[\Phi_{M,A}^k\colon \EE^k_M(A,E)\to C^k(M,E),\quad f\mapsto\sum_{i=1}^m F_{\varphi_i}(f)\]
stetig und linear ist.

Sind $f=(f_{\varphi,\alpha})_{\varphi\in\A, \al\leq k}\in\EE^k_M(A,E)$ und $F\coloneqq\Phi^k_{M,A}(f)$ gegeben, so zeigt man auf ähnliche Weise wie in (a), dass
\[(\p(F\circ\varphi_i^{-1}))(\varphi_i(x))=f_{\varphi_i,\alpha}(\varphi_i(x))\]
für alle $i\in\{1,\ldots,m\}$, $\al\leq k$ und $x\in A\cap U_{\varphi_i}$. Somit gilt
\[\text{pr}_{\B}\circ\rho_{A,M}^k\circ\Phi_{M,A}^k=\text{pr}_{\B}\]
mit $\B\coloneqq\{\varphi_1,\ldots,\varphi_m\}$ und der Projektion $\text{pr}_{\B}\colon\EE^k_M(A,E)\to\EE^k_{\B}(A,E)$. Da $\text{pr}_{\B}$ ein Isomorphismus ist, folgt
\[\rho_{A,M}^k\circ\Phi_{M,A}^k=\text{id}_{\EE^k_M(A,E)}.\]
\end{proof}

\begin{korollar}
Es seien $k\in\N_0$, $M$ eine endlichdimensionale $C^k$-Man"-nig"-fal"-tig"-keit mit rauem Rand und $A\subseteq M$ eine abgeschlossene Un"-ter"-man"-nig"-fal"-tig"-keit. Angenommen, eine der beiden folgenden Bedingungen ist erfüllt:
\begin{itemize}
\item[(a)] $M$ ist $C^k$-parakompakt und $A$ lokalkompakt.
\item[(b)] $M$ ist regulär und $A$ kompakt.
\end{itemize}
Dann besitzt die Einschränkung
\begin{align}\label{gl.ll}C^k(M,E)\to C^k(A,E),\quad f\mapsto f\vert_A\end{align}
eine stetige lineare Rechtsinverse.
\end{korollar}

\begin{proof}
Sei $\B$ die Menge der an $A$ angepassten Karten von $M$.
Wenn $A$ eine volldimensionale Untermannigfaltigkeit ist, betrachten wir das kommutative Diagramm von topologischen Vektorräumen und stetigen linearen Abbildungen
\[\begin{xy}\xymatrix{
& \EE^k_M(A,E)\ar[rd]^{\text{pr}_{\B}}_{\cong} & \\
C^k(M,E)\ar[rd]_{\eqref{gl.ll}}\ar[ru]^{\rho_{A,M}^k} & & \EE^k_{\B}(A,E).\\
& C^k(A,E)\ar[ru]^{\cong}_{\eqref{gl.w}} &
}\end{xy}\]
Da $\rho^k_{A,M}$ eine stetige lineare Rechtsinverse besitzt, gilt dies auch für \eqref{gl.ll}.

Wenn $A$ eine Untermannigfaltigkeit von kleinerer Dimension ist, betrachten wir das kommutative Diagramm
\[\begin{xy}\xymatrix{
& \EE^k_M(A,E)\ar[r]_{\cong}^{\text{pr}_{\B}} & \EE^k_{\B}(A,E)\ar[rd]^{P_{\widetilde{\B}}} & \\
C^k(M,E)\ar[ru]^{\rho_{A,M}^k}\ar[rd]_{\eqref{gl.ll}} & & & \EE^k_{\widetilde{\B}}(A,E) \\
& C^k(A,E)\ar[r]^{\cong} & \EE^k_A(A,E)\ar[ru]^{\cong}_{\text{pr}_{\widetilde{\B}}} &
}\end{xy}\]
mit $\widetilde{\B}\coloneqq\{\varphi_A\colon\varphi\in\B\}$ und $P_{\widetilde{\B}}$ wie in \eqref{gl.aa}.
Da $\rho^k_{A,M}$ und $P_{\widetilde{\B}}$ beide stetige lineare Rechtsinverse besitzen, hat auch \eqref{gl.ll} eine solche.
\end{proof}

\begin{satz}
Es seien $E$ metrisierbar, $(M,\A)$ eine endlichdimensionale $C^{\infty}$-Man"-nig"-fal"-tig"-keit mit rauem Rand und $A\subseteq M$ eine abgeschlossene Teilmenge.
Angenommen, eine der beiden folgenden Bedingungen ist erfüllt:
\begin{itemize}
\item[(a)] $M$ ist $C^{\infty}$-parakompakt und $A$ lokalkompakt.
\item[(b)] $M$ ist regulär und $A$ kompakt.
\end{itemize}
Dann ist die stetige lineare Abbildung
\[\rho^{\infty}_{A,M}\colon C^{\infty}(M,E)\to \EE^{\infty}_M(A,E),\quad f\mapsto \big((\p(f\circ\varphi^{-1}))\vert_{\varphi(A\cap U_{\varphi})}\big)_{\substack{\begin{subarray}{l}\varphi\in\A\\ \alpha\in(\N_0)^n\end{subarray}}}\]
surjektiv.
\end{satz}

\begin{proof}
Wir beweisen die Aussage unter der Voraussetzung (a); für (b) verfährt man auf ähnliche Weise.
Ein $\varphi\in\A$ gegeben, können wir nach Satz \ref{ll} für alle $f\in\EE^{\infty}(\varphi(A\cap U_{\varphi}),E)$ eine Fortsetzung $F_{2,\varphi}(f)\in C^{\infty}(V_{\varphi},E)$ wählen. So erhalten wir eine (nicht notwendig stetige oder lineare) Abbildung 
\[F_{2,\varphi}\colon\EE^{\infty}(\varphi(A\cap U_{\varphi}),E)\to C^{\infty}(V_{\varphi},E),\quad f\mapsto F_{2,\varphi}(f).\]
Analog wie im Beweis von Satz \ref{bb} konstruiert man eine Abbildung
\[\Phi^{\infty}_{M,A}\colon \EE^{\infty}_M(A,E)\to C^{\infty}(M,E),\]
so dass für alle $f\in\EE^{\infty}_M(A,E)$ die $C^{\infty}$-Funktion $\Phi^{\infty}_{M,A}(f)$ ein Urbild von $f$ unter $\rho^{\infty}_{A,M}$ ist.
\end{proof}

\begin{korollar}
Es seien $E$ metrisierbar, $M$ eine endlichdimensionale $C^{\infty}$-Man"-nig"-fal"-tig"-keit mit rauem Rand und $A\subseteq M$ eine abgeschlossene Un"-ter"-man"-nig"-fal"-tig"-keit.
Angenommen, eine der beiden folgenden Bedingungen ist erfüllt:
\begin{itemize}
\begin{samepage}
\item[(a)] $M$ ist $C^{\infty}$-parakompakt und $A$ lokalkompakt.
\item[(b)] $M$ ist regulär und $A$ kompakt.
\end{samepage}
\end{itemize}
Dann ist die Einschränkung
\[C^{\infty}(M,E)\to C^{\infty}(A,E),\quad f\mapsto f\vert_A\]
surjektiv.
\end{korollar}

\subsection*{Der projektive Limes $\EE^{\infty}_M(A,E)=\lim_{\leftarrow}\EE^k_M(A,E)$}
\addcontentsline{toc}{subsection}{\protect\numberline{}Der projektive Limes $\EE^{\infty}_M(A,E)=\lim_{\leftarrow}\EE^k_M(A,E)$}

Es seien $(M,\A_{\infty})$ eine endlichdimensionale $C^{\infty}$-Mannigfaltigkeit mit rauem Rand und $A\subseteq M$ eine Teilmenge. Für jedes $k\in\N_0$ sei $(M,\A_k)$ die $M$ zugrunde liegende $C^k$-Mannigfaltigkeit mit rauem Rand. Dann gilt $\A_0\supseteq\A_1\supseteq\A_2\supseteq\ldots$ und $\A_k\supseteq\A_{\infty}$ für alle $k\in\N_0$ sowie $\EE^k_M(A,E)=\EE^k_{\A_k}(A,E)$. Für alle $k\leq l\leq m$ in $\N_0\cup\{\infty\}$ betrachten wir die Projektionen
\[\pi^m_{k,l}\colon \EE^{l}_{\A_m}(A,E)\to\EE^{k}_{\A_m}(A,E),\quad (f_{\varphi})_{\varphi\in\A_m}\to (\text{pr}_k(f_{\varphi}))_{\varphi\in\A_m}\]
sowie
\[p^k_{m,l}\colon \EE^k_{\A_l}(A,E)\overset{\cong}{\longrightarrow}
\EE^k_{\A_m}(A,E),\quad (f_{\varphi})_{\varphi\in\A_l}\mapsto (f_{\varphi})_{\varphi\in\A_m}\]
und erhalten stetige lineare Abbildungen
\[q_{k,l}\coloneqq (p_{l,k}^k)^{-1}\circ \pi^l_{k,l}\colon \EE^l_M(A,E)\to\EE^k_M(A,E).\]
Wir haben also eine projektive Folge von lokalkonvexen Räumen
\[\EE^0_M(A,E)\overset{q_{0,1}}{\longleftarrow}\EE^1_M(A,E)\overset{q_{1,2}}{\longleftarrow}\EE^2_M(A,E)\overset{q_{2,3}}{\longleftarrow}\ldots\]
gegeben.

\begin{satz}
Es gilt 
\[\EE^{\infty}_M(A,E)=\lim_{\substack{\longleftarrow}}\EE^k_M(A,E),\]
wobei wir die Abbildungen $q_{k,\infty}\colon\EE^{\infty}_M(A,E)\to \EE^k_M(A,E)$ für $k\in\N_0$ als Limesabbildungen verwenden.
\end{satz}

\begin{proof}
Zunächst betrachten wir die projektive Folge von lokalkonvexen Räumen
\begin{align}\label{gl.r}\EE^0_{\A_{\infty}}(A,E)\overset{\pi^{\infty}_{0,1}}{\longleftarrow}\EE^1_{\A_{\infty}}(A,E)\overset{\pi^{\infty}_{1,2}}{\longleftarrow}\EE^2_{\A_{\infty}}(A,E)\overset{\pi^{\infty}_{2,3}}{\longleftarrow}\ldots.\end{align}
Da die Abbildungen
\[\EE^{\infty}_M(A,E)\hookrightarrow\prod_{\varphi\in\A_{\infty}}\EE^{\infty}(\varphi(A\cap U_{\varphi}),E),\quad f\mapsto f\]
und
\[\EE^{\infty}(\varphi(A\cap U_{\varphi}),E)\hookrightarrow\prod_{k\in\N_0}\EE^k(\varphi(A\cap U_{\varphi}),E),\quad f\mapsto (\text{pr}_k(f))_{k\in\N_0}\]
für $\varphi\in\A_{\infty}$ Einbettungen von topologischen Vektorräumen sind, ist die Co-"-Ein"-schrän"-kung von
\[\EE^{\infty}_M(A,E)\hookrightarrow \prod_{\varphi\in\A_{\infty}}\prod_{k\in\N_0}\EE^k(\varphi(A\cap U_{\varphi}),E)\cong \prod_{k\in\N_0}\prod_{\varphi\in\A_{\infty}}\EE^k(\varphi(A\cap U_{\varphi}),E)\]
in den Raum $\prod_{k\in\N_0}\EE^k_{\A_{\infty}}(A,E)$, nämlich die Abbildung
\[\phi\colon\EE^{\infty}_M(A,E)\hookrightarrow \prod_{k\in\N_0}\EE^k_{\A_{\infty}}(A,E),\quad f\mapsto(\pi_{k,\infty}^{\infty}(f))_{k\in\N_0},\]
eine Einbettung von topologischen Vektorräumen. Nun ist aber das Bild
\begin{align*}\phi(\EE^{\infty}_M(A,E))&=\{(\pi_{k,\infty}^{\infty}(f))_{k\in\N_0}\colon f\in\EE^{\infty}_M(A,E)\}\\
&=\bigg\{(f_k)_{k\in\N_0}\in \prod_{k\in\N_0}\EE^k_{\A_{\infty}}(A,E)\colon \pi^{\infty}_{k,k+1}(f_{k+1})=f_k\text{ für alle }k\in\N_0\bigg\}\end{align*}
der projektive Standardlimes von \eqref{gl.r}. Somit gilt
\[\EE^{\infty}_M(A,E)=\lim_{\substack{\longleftarrow}}\EE_{\A_{\infty}}^k(A,E)\]
mit den Projektionen $\pi^{\infty}_{k,\infty}\colon\EE^{\infty}_M(A,E)\to\EE^k_{\A_{\infty}}(A,E)$ für $k\in\N_0$ als Limesabbildungen. Für alle $k\leq l\leq m$ in $\N_0\cup\{\infty\}$ gilt
\begin{align*}
q_{k,l}\circ q_{l,m}
&=((p_{l,k}^k)^{-1}\circ\pi_{k,l}^l)\circ((p_{m,l}^l)^{-1}\circ\pi_{l,m}^m)
=(p_{l,k}^k)^{-1}\circ(\pi_{k,l}^l\circ(p_{m,l}^l)^{-1})\circ\pi_{l,m}^m\\
&=(p_{l,k}^k)^{-1}\circ ((p_{m,l}^k)^{-1}\circ\pi_{k,l}^m)\circ\pi_{l,m}^m
=(p_{m,l}^k\circ p_{l,k}^k)^{-1}\circ(\pi_{k,l}^m\circ\pi_{l,m}^m)\\
&=(p_{m,k}^k)^{-1}\circ\pi_{k,m}^m=q_{k,m}.
\end{align*}
Für jedes $k\in\N_0$ gilt insbesondere $q_{k,k+1}\circ q_{k+1,\infty}=q_{k,\infty}$, und auf ähnliche Weise sieht man, dass $q_{k,k+1}\circ(p_{\infty,k+1}^{k+1})^{-1}=(p_{\infty,k}^k)^{-1}\circ\pi_{k,k+1}^{\infty}$. Somit ist das Diagramm
\[\begin{xy}\xymatrix{& \EE^{\infty}_M(A,E)\ar[ldd]^/0.8cm/{q_{k,\infty}}\ar[rdd]_/0.8cm/{q_{k+1,\infty}}\ar[ld]_{\pi^{\infty}_{k,\infty}}\ar[rd]^{\pi^{\infty}_{k+1,\infty}} & \\ \EE^{k}_{\A_{\infty}}(A,E)\ar[d]^{\cong}_{(p^k_{\infty,k})^{-1}}& & \EE^{k+1}_{\A_{\infty}}(A,E)\ar[ll]^{\pi^{\infty}_{k,k+1}}\ar[d]_{\cong}^{(p_{\infty,k+1}^{k+1})^{-1}}\\
\EE^{k}_M(A,E) & & \EE^{k+1}_M(A,E)\ar[ll]^{q_{k,k+1}}
}\end{xy}\]
kommutativ, woraus die Behauptung des Satzes folgt.
\end{proof}

\section*{Schlussbemerkung}
\addcontentsline{toc}{section}{\protect\numberline{}Schlussbemerkung}
In der Literatur wird vielfach der Frage nachgegangen, wann Whitneysche $k$-Jets oder $C^k$-Funktionen auf abgeschlossenen Teilmengen $A\subseteq\R^n$ mittels stetigen linearen Fortsetzungsoperatoren zu $C^k$-Funktionen auf $\R^n$ fortgesetzt werden können. 
Einen lokalkonvexen Raum $E$ gegeben, haben wir in Anlehnung an Whitney \cite{Wh} die Existenz von stetigen linearen Fortsetzungsoperatoren \[\EE^k(A,E)\to C^k(\R^n,E)\] für $k\in\N_0$ und beliebige abgeschlossene Teilmengen $A\subseteq\R^n$ gezeigt.

Glöckner hat im Falle eines folgenvollständigen lokalkonvexen Raumes $E$ für $k\in\N_0\cup\{\infty\}$ und abgeschlossene konvexe Teilmengen $A\subseteq\R^n$ mit dichtem Inneren nachgewiesen, dass die Existenz eines stetigen linearen Fortsetzungsoperators $C^k(A,E)\to C^k(\R^n,E)$ für $E$-wertige Funktionen aus der Existenz eines stetigen linearen Fortsetzungsoperators $C^k(A,\R)\to C^k(\R^n,\R)$ für skalarwertige Funktionen folgt (s. \cite[Theorem 1.3]{Gl2}).
Als Hilfsmittel gebraucht der Autor topologische Tensorprodukte von lokalkonvexen Räumen, wie sie von Grothendieck \cite{Gr} entwickelt wurden. Er stellt fest, dass
\[C^k(A,\widetilde{E})\cong \widetilde{E} \widetilde{\otimes}_{\eps} C^k(A,\R)\]
als lokalkonvexer Raum gilt mit der Vervollständigung $\widetilde{E}$ von $E$.
Ist ein stetiger linearer Fortsetzungsoperator
$\Phi\colon C^k(A,\R)\to C^k(\R^n,\R)$
gegeben, dann kann man die stetige lineare Abbildung
\[\text{id}_{\widetilde{E}} \widetilde{\otimes}_{\eps} \Phi\colon \underbrace{ \widetilde{E} \widetilde{\otimes}_{\eps} C^k(A,\R)}_{\cong C^k(A,\widetilde{E})}\to \underbrace{\widetilde{E} \widetilde{\otimes}_{\eps} C^k(\R^n,\R)}_{\cong C^k(\R^n,\widetilde{E})}\]
mit einer Abbildung
$C^k(A,\widetilde{E})\to C^k(\R^n,\widetilde{E})$
identifizieren, die sich
zu einem stetigen linearen Fortsetzungsoperator
$C^k(A,E)\to C^k(\R^n,E)$
einschränken lässt (s. \cite[Proof of Theorem 1.3]{Gl2}).
Mit obiger Aussage folgert Glöckner aus Ergebnissen von \cite{Ro, Fr, Bi}, dass stetige lineare Fortsetzungsoperatoren \[C^{\infty}(A,E)\to C^{\infty}(\R^n,E)\] für alle abgeschlossenen konvexen Teilmengen $A\subseteq\R^n$ mit dichtem Inneren existieren (s. \cite[Corollary 1.6]{Gl2}).

Als weiterführende Frage könnte man untersuchen, für welche lokalkonvexen Räume $E$ und abgeschlossenen Teilmengen $A\subseteq\R^n$ die Beziehung
\[\EE^k(A,E)\cong E \widetilde{\otimes}_{\eps} \EE^k(A,\R)\]
erfüllt ist, um weitere Rückschlüsse auf die Existenz stetiger linearer Fortsetzungsoperatoren für vektorwertige Jets zu ziehen.

Ähnliche Methoden mit $\eps$-Produkten nutzen z.B. Kruse \cite{Kr} oder Frerick et al. \cite{FJW}, um den Zusammenhang zwischen Fortsetzbarkeit und schwacher Fortsetzbarkeit von vektorwertigen Funktionen aus anderen Funktionsklassen zu untersuchen.

Tidten \cite{Ti} klassifiziert alle abgeschlossenen Teilmengen $A\subseteq\R^n$, die einen stetigen linearen Fortsetzungsoperator
\[\EE^{\infty}(A,\R)\to C^{\infty}(\R^n,\R)\]
erlauben.
Für weitere diesbezügliche Studien verweisen wir auf die Arbeit von Frerick \cite{Fr}. Bierstone \cite{Bi} diskutiert hinreichende Bedingungen für die Existenz von stetigen linearen Fortsetzungsoperatoren $\EE^{\infty}(A,\R)\to C^{\infty}(\R^n,\R)$.

Seeley \cite{Se} hat für $k\in\N_0\cup\{\infty\}$ stetige lineare Fortsetzungsoperatoren
\[\Phi^k\colon C^k(\R^{n-1}\times[0,\infty[,\R)\to C^k(\R^n,\R)\]
konstruiert, welche in dem Sinne natürlich sind, dass für alle $k\leq l$ in $\N_0\cup\{\infty\}$ das Diagramm
\[\begin{xy}\xymatrix{
C^l(\R^{n-1}\times[0,\infty[,\R)\ar@_{(->}[d]\ar@^{->}[d]\ar[r]^{\ \ \ \ \ \ \ \Phi^l} & C^l(\R^n,\R)\ar@_{(->}[d]\ar@^{->}[d]\\ C^k(\R^{n-1}\times[0,\infty[,\R)\ar[r]^{\ \ \ \ \ \ \ \Phi^k} &C^k(\R^n,\R)
}\end{xy}\]
kommutativ ist. Hanusch \cite{Ha} verallgemeinert Seeleys Ansatz für vektorwertige Funktionen, indem er u.a. für $k\in\N_0\cup\{\infty\}$ und $m\in\{0,\ldots,n\}$ stetige lineare Fortsetzungsoperatoren
\[C^k(\R^{n-m}\times [0,\infty[^m,E)\to C^k(\R^n,E)\]
definiert (s. \cite[Corollary 1.8]{Gl2}).

Federer beweist einen Fortsetzungssatz für Jets mit Werten in einem normierten Raum $E$ (s. \cite[S. 225]{Fed}, zit. nach \cite{Ba}). Baldi \cite{Ba} setzt Jets mit Werten in einer Skala von Banachräumen mithilfe von Glättungsoperatoren fort. 
Margalef-Roig und Outerelo Dominguez diskutieren die Fortsetzung von Jets, die auf Oktanten in einem Banachraum $E$ definiert sind, also auf Mengen der Form \[\{x\in E\colon \lambda_1(x)\geq 0,\ldots,\lambda_m(x)\geq 0\}\] mit linear unabhängigen stetigen linearen Funktionalen $\lambda_1,\ldots,\lambda_m\colon E\to \R$ (s. \cite[Theorem 2.1.36]{Ma}).

Fefferman \cite{Fe} behandelt ein zu unserer Fragestellung komplementäres Problem. Für $k\in\N$ sei der Raum $C^k_b(\R^n,\R)$ aller $C^k$-Funktionen $f\in C^k(\R^n,\R)$, deren partielle Ableitungen $\p f$ für alle $\al\leq k$ beschränkt sind, mit der Norm \[C^k_b(\R^n,\R)\to[0,\infty[,\quad f\mapsto \max_{\substack{\al\leq k}}\sup_{x\in\R^n}\vert\p f(x)\vert\]
versehen. Für jede abgeschlossene Teilmenge $A\subseteq\R^n$ hat die surjektive lineare Abbildung
\begin{align}\tag{$\ast$}\label{gl.mm}C^k_b(\R^n,\R)\to C^k_b(\R^n,\R)\vert_A,\quad f\mapsto f\vert_A\end{align}
mit $C^k_b(\R^n,\R)\vert_A\coloneqq\{f\vert_A\colon f\in C^k_b(\R^n,\R)\}$ einen abgeschlossenen Kern, so dass man $C^k_b(\R^n,\R)\vert_A$ mit der zugehörigen Quotientennorm versehen kann. Fefferman beweist die Existenz einer stetigen linearen Rechtsinversen von \eqref{gl.mm}.

Darüber hinaus haben wir in der vorliegenden Arbeit gezeigt, dass für $k\in\N_0$ stetige lineare Fortsetzungsoperatoren
\[\EE^k_M(A,E)\to C^k(M,E)\]
für abgeschlossene Teilmengen $A\subseteq M$ einer endlichdimensionalen $C^k$-Man"-nig"-fal"-tig"-keit $M$ mit rauem Rand existieren, wenn $M$ $C^k$-parakompakt und $A$ lokalkompakt (oder $M$ regulär und $A$ kompakt) ist. Für weitere Fortsetzungsresultate, die Funktionen auf Mannigfaltigkeiten zum Gegenstand haben, sei auf die Arbeiten von Glöckner \cite{Gl2} sowie Roberts und Schmeding \cite{Ro} verwiesen.

\addcontentsline{toc}{section}{\protect\numberline{}Literatur}

Johanna Jakob, Universität Paderborn, Warburger Str. 100, 33098 Paderborn,
Germany


\begin{thebibliography}{99}
\bibitem{Ba} \textsc{Baldi, P.}: \emph{A Whitney extension theorem for functions taking values in scales of Banach spaces}, J. Funct. Anal. \textbf{283} (2022), Artikelnr. 109492.
\bibitem{Bi} \textsc{Bierstone, E.}: \emph{Differentiable functions}, Bol. Soc. Brasil. Math. \textbf{11} (1980), 139–189.
\bibitem{CH} \textsc{Clark, D. E.; Houssineau, J.}: \emph{Fa\`a di Bruno’s formula for Gâteaux
diﬀerentials and interacting stochastic population processes}, Vorveröffentlichung,
2012, arXiv:1202.0264.
\bibitem{Fed} \textsc{Federer, H.}: \emph{Geometric measure theory}, Springer, Berlin Heidelberg, 1969, Nachdruck 1996.
\bibitem{Fe} \textsc{Fefferman, C.}: \emph{$C^m$ extension by linear operators}, Ann. of Math. \textbf{166} (2007), 779–835.
\bibitem{Fr} \textsc{Frerick, L.}: \emph{Extension operators for spaces of inﬁnite differentiable Whitney jets}, J. reine angew. Math. \textbf{602} (2007), 123–154.
\bibitem{FJW} \textsc{Frerick, L.; Jordá, E.; Wengenroth, J.}: \emph{Extension of bounded vector-valued functions}, Math. Nachr. \textbf{282} (2009), 690–696.
\bibitem{Gl2} \textsc{Glöckner, H.}: \emph{Smoothing operators for vector-valued functions and extension operators}, Vorveröffentlichung, 2022, arXiv:2006.00254v2.
\bibitem{Gl} \textsc{Glöckner, H.; Neeb, K.-H.}: \emph{Inﬁnite dimensional Lie groups}, Buch in Vorbereitung.
\bibitem{Gr} \textsc{Grothendieck, A.}: \emph{Topological tensor products and nuclear spaces}, Memoirs Amer. Math. Soc. \textbf{16} (1955), 196+140 ff.
\bibitem{Ha} \textsc{Hanusch, M.}: \emph{A $C^k$-Seeley-extension-theorem for Bastiani’s diﬀerential calculus}, Canad. J. Math. (2021), 1–32.
\bibitem{J} \textsc{Jakob, J.}: \emph{Der Whitneysche Fortsetzungssatz für vektorwertige Funktionen}, Masterarbeit, Universität Paderborn, 2022, Betreuer: Prof. Dr. H. Glöckner.
\bibitem{Mi} \textsc{Kriegl, A.; Michor, P. W.}: \emph{The convenient setting of global analysis}, Amer. Math. Soc., Providence, 1997.
\bibitem{Kr} \textsc{Kruse, K.}: \emph{Extension of vector-valued functions and weak-strong principles for differentiable functions of finite order}, Ann. Funct. Anal. \textbf{13} (2022), Artikelnr. 10.
\bibitem{Ma} \textsc{Margalef-Roig, J.; Outerelo Dominguez, E.}: \emph{Differential topology}, North-Holland, Amsterdam, 1992.
\bibitem{Ro} \textsc{Roberts, D. M.; Schmeding, A.}: \emph{Extending Whitney’s extension theorem: nonlinear function spaces}, Ann. Inst. Fourier \textbf{71} (2021), 1241–1286.
\bibitem{Se} \textsc{Seeley, R. T.}: \emph{Extension of $C^{\infty}$-functions deﬁned in a half space}, Proc. Amer. Math. Soc. \textbf{15} (1964), 625–626.
\bibitem{Ti} \textsc{Tidten, M.}: \emph{Fortsetzungen von $C^{\infty}$-Funktionen, welche auf einer abgeschlossenen Menge in $\R^n$ definiert sind}, Manuscr. Math. \textbf{27} (1979),  291–312.
\bibitem{Wh} \textsc{Whitney, H.}: \emph{Analytic extensions of differentiable functions defined in closed sets}, Trans. Amer. Math. Soc. \textbf{36} (1934), 63–89.
\end{thebibliography}
\end{document}